\documentclass[a4paper,11pt,draft]{amsart}
\usepackage{amsmath,amsfonts,amssymb,amsthm}
\usepackage{graphics}
\usepackage{enumerate}
\usepackage{epsfig}
\usepackage{esint}
\usepackage{newlfont}
\usepackage{fancyhdr}
\usepackage{stackrel}
\usepackage{color}
\usepackage[body={15cm,21.5cm},centering]{geometry}

\newcommand{\lessim}{\stackrel{<}{\sim}}

\newcommand{\ignore}[1]{}

\newtheorem{definition}{Definition}
\newtheorem{proposition}{Proposition}
\newtheorem{theorem}{Theorem}
\newtheorem{remark}{Remark}
\newtheorem{lemma}{Lemma}
\newtheorem{corollary}{Corollary}

\newcommand{\sym}{\mathrm{sym}}
\newcommand{\loc}{\mathrm{loc}}

\newcommand{\ho}{\mathrm{hom}}

\newcommand{\dist}{\mathrm{dist}}

\newcommand{\R}{\mathbb{R}}
\newcommand{\Z}{\mathbb{Z}}
\newcommand{\N}{\mathbb{N}}

\newcommand{\Id}{\text{Id}}
\newcommand{\e}{\varepsilon}

\newcommand{\partn}{\mathcal{P}}

\newcommand{\calI}{\mathcal{I}}

\newcommand{\Ent}{\mathrm{Ent}\,}

\newcommand{\felix}{\frac18}
\newcommand{\Felix}{3}

\newcommand{\ru}{\underline{r_*}}

\mathchardef\emptyset="001F

\def\Exc{\operatorname{Exc}}

\newcommand{\diam}[1]{\mathrm{diam}\left( #1\right)}

\newcommand{\varL}[1]{\mathrm{var}_L\left[#1\right]}
\newcommand{\covL}[2]{\mathrm{cov}_L\left[#1;#2\right]}
\newcommand{\expec}[1]{\left\langle #1 \right\rangle}

\newcommand{\step}[1]{\noindent \textit{Step} #1.}
\newcommand{\substep}[1]{\noindent \textit{Substep} #1.}
\newcommand{\supp}[1]{\mathrm{supp} #1}

\newcommand{\fun}{\mathrm{fct}}
\newcommand{\osc}{\mathrm{osc}}

\title[A regularity theory for random elliptic operators]
{A regularity theory for random elliptic operators}
\author[A. Gloria]{Antoine Gloria}
\author[S. Neukamm]{Stefan Neukamm}
\author[F. Otto]{Felix Otto}
\date{\today}
\address[Antoine Gloria]{Sorbonne Universit\'e, CNRS, Universit\'e de Paris, Laboratoire Jacques-Louis Lions (LJLL), F-75005 Paris, France 
\& 
Universit\'e Libre de Bruxelles, Belgium}
\email{antoine.gloria@upmc.fr}
\address[Stefan Neukamm]{Faculty of Mathematics, TU Dresden, Germany}
\email{stefan.neukamm@tu-dresden.de}
\address[Felix Otto]{Max Planck Institute for Mathematics in the Sciences, Leipzig, Germany}
\email{otto@mis.mpg.de}

\begin{document}

\maketitle


{\bf Abstract}: 
Since the seminal results by Avellaneda \& Lin it is known that elliptic operators with periodic coefficients 
enjoy the same regularity theory as the Laplacian on large scales.
In a recent inspiring work, Armstrong \& Smart proved large-scale Lipschitz estimates for such operators with random coefficients satisfying a finite-range of dependence assumption.
In the present contribution, we extend the \emph{intrinsic large-scale} regularity of Avellaneda \& Lin (namely,  intrinsic large-scale Schauder and Calder\'eron-Zygmund estimates)
to  elliptic systems with random coefficients. The scale at which this improved regularity kicks in is characterized by 
a stationary field $r_*$ which we call the minimal radius. This regularity theory is \textit{qualitative} in the sense that $r_*$ is almost surely finite (which yields a new Liouville theorem) under mere ergodicity, and it is \textit{quantifiable} in the sense that $r_*$ has high stochastic integrability provided the coefficients satisfy quantitative mixing assumptions.
We illustrate this by establishing \emph{optimal} moment bounds on $r_*$ for a class of coefficient fields satisfying 
a multiscale functional inequality, and in particular for Gaussian-type coefficient fields with arbitrary slow-decaying correlations.

\bigskip

\tableofcontents

\section{Introduction}

This article is the first of a series that develops a quantitative theory for large-scale properties of random elliptic operators.
It presents digested and optimized versions of the proofs of the first complete version of the manuscript (dated August 2015).
The series consists of three parts: An \textit{intrinsic large-scale regularity theory} in the present contribution, applications to \textit{quantitative stochastic homogenization} in \cite{GNO-quant},
and the \textit{characterization of large-scale fluctuations} in \cite{DGO,DGO-Gaus,DFGO-Gaus}.

\medskip

The classical theory of homogenization for elliptic systems  $-\nabla\cdot a\nabla$ with periodic, uniformly elliptic coefficients $a$ started with contributions of the French, the Italian and the Russian schools (e.g.~see \cite{Tartar-CP,Murat-78,Bensoussan-Lion-Papa-78, DeGiorgi-75, Spagnolo-75, JKON-79,JKO-94}). Classical homogenization states that on \textit{large scales} (i.e., scales much larger than
the period of $a$) the resolvent of $-\nabla\cdot a\nabla$ is close to the resolvent of the so-called homogenized operator $-\nabla\cdot a_{\hom}\nabla$, where $a_{\hom}$ are \textit{spatially homogeneous} coefficients. In the seminal work \cite{Avellaneda-Lin-87} Avellaneda and Lin observed that homogenization can be used to lift the regularity theory for the homogenized (constant-coefficient) operator $-\nabla\cdot a_{\hom}\nabla$ to the original variable-coefficient operator $-\nabla\cdot a\nabla$. Since elliptic systems with measurable coefficients basically only enjoy $L^2$-regularity theory (as opposed to the maximal regularity of elliptic systems with constant coefficients), Avellaneda and Lin's results yield a strong improvement of regularity (on large scales). More precisely, in \cite[Section 3.1]{Avellaneda-Lin-87} Avellaneda and Lin derive \textit{intrinsic} $C^{1,1-}$-a priori
estimates on $a$-harmonic functions (where $C^{k,1-}$ means $C^{k,\alpha}$ for all $\alpha<1$). Here \textit{intrinsic} refers to the fact that the estimates are formulated  not in Euclidean (flat) coordinates, but with help of the so-called \textit{harmonic coordinates}, which are based on the notion of the \textit{corrector} -- a key object in the theory of homogenization. In the present paper we extend the intrinsic large-scale regularity theory of Avellaneda and Lin to the case of elliptic systems with random (in particular stationary \& ergodic) coefficients. This extension from the periodic to the random setting is non-trivial since the original argument of Avellaneda and Lin crucially relies on a compactness argument (related to the compactness of the torus associated with the periodic coefficients). 
\smallskip

Qualitative stochastic homogenization of uniformly elliptic equations with random coefficients was first established by Papanicolaou and Varadhan \cite{Papanicolaou-Varadhan-79} and by Kozlov \cite{Kozlov-78}.  The argument of Papanicolaou and Varadhan \cite{Papanicolaou-Varadhan-79} is based on Tartar's method of oscillating test-functions, and, in the core of the analysis, extends the notion of corrector to the random setting: Roughly speaking, if $a$ denotes a random coefficient field, and $e$ a fixed unit direction of $\R^d$, then the associated corrector $\phi$ is defined as a sublinearly growing solution of
\begin{equation}\label{intro-corr00}
-\nabla \cdot a \nabla \phi=\nabla \cdot a e\qquad\text{in }\R^d
\end{equation}
(see Lemma~\ref{si} below for the precise statement). The first example of a large-scale regularity result due to randomness  is the higher stochastic integrability of the gradient $\nabla\phi$ of the corrector obtained in \cite{Gloria-Otto-09,Gloria-Otto-10b,Gloria-Neukamm-Otto-14}  (in the course of proving quantitative results in stochastic homogenization). 
Developing a quantitative theory obviously requires quantitative ergodicity assumptions, which we make in the form of a functional inequality (e.g.~a spectral gap estimate), inspired by the unpublished work \cite{Naddaf-Spencer-98} by Naddaf and Spencer. Among other estimates, we proved that if the random coefficients $a$ satisfy a spectral gap estimate (see Definition~\ref{def:sLSI}), then for all $1\leq p<\infty$
$$
\Big\langle \big(\fint_{B_1(0)}|\nabla \phi|^2\big)^\frac p2\Big\rangle^\frac1p \,\lesssim\, 1\,\sim\, \Big\langle \big(\fint_{B_1(0)}|ae|^2\big)^\frac p2\Big\rangle^\frac1p,
$$
where the multiplicative constant depends on $p$. (Note that the case $p=2$ follows by an elementary energy estimate.) In view of \eqref{intro-corr00}, this result is reminiscent of a Calder\'on-Zygmund estimate. It is a \textit{large-scale regularity} result since it involves taking the expectation $\expec{\cdot}$, which by Birkhoff's ergodic theorem turns into the ``large-scale'' spatial average $\lim_{R\to\infty}\fint_{B_R}(\cdot) $.
In terms of H\"older regularity, the first large-scale regularity result in the random setting was by Marahrens and the third author  in \cite{Marahrens-Otto-13}: For scalar equations and 
under a strong quantitative ergodicity assumption in form of a Logarithmic Sobolev Inequality (see Definition~\ref{def:sLSI}),  
(large-scale) $C^{0,1-}$-estimates for $a$-harmonic functions were established (see also \cite{Gloria-Marahrens-14}).
Common key elements to these works are \textit{functional inequalities} to quantify ergodicity, a \textit{sensitivity calculus} to estimate the dependence of a solution, like $\nabla \phi$, on the coefficient field $a$, and input from \textit{deterministic regularity theory}, which in \cite{Gloria-Otto-09,Gloria-Otto-10b,Gloria-Neukamm-Otto-14,Marahrens-Otto-13} are encoded in form of Green's function estimates that crucially use De Giorgi-Nash-Moser regularity theory. The latter restricts the results of these works to scalar equations (see however  \cite{BO17, BMN17} for systems, and, thanks to \cite{CGO-17}, also \cite{Marahrens-Otto-13} extends to systems).
 A motivation for the present work is to replace the deterministic regularity theory used in these works by the intrinsic large-scale regularity theory developed in the present paper  -- this will be addressed in our follow-up work \cite{GNO-quant}.
 
\smallskip

With another flavor and under the sole assumption of stationarity and ergodicity (as opposed to the strong quantitative ergodicity assumption of \cite{Marahrens-Otto-13}), Benjamini, Duminil-Copin, Kozma, and Yadin proved in \cite{BCKY-11} a Liouville theorem in a very general context which states that strictly sublinear $a$-harmonic functions are constants. Note that Liouville theorems and Schauder theory are intimately connected: E.g.~Simon derived Schauder estimates \cite[Theorem
  1]{Simon-97} indirectly from a Liouville result \cite[Lemma~1]{Simon-97}; while Avellaneda and Lin  \cite{Avellaneda-Lin-89} obtained Liouville theorems of any order for elliptic systems in the periodic setting by appealing to their large-scale regularity theory. 
\smallskip

In a recent inspiring work, Armstrong and Smart \cite{Armstrong-Smart-14} developed a large-scale regularity theory in the random setting. It is the first result that implements the general strategy \cite{Avellaneda-Lin-87} of Avellaneda and Lin (of lifting the regularity theory for the homogenized operator to the original operator) in a situation where the above-mentioned compactness argument fails. Roughly speaking, in their approach the compactness argument is replaced by a quantitative estimate of the homogenization error that can be established under  quantitative ergodicity. In contrast to our framework based on \emph{nonlinear mixing conditions} (as developed in \cite{Gloria-Otto-09,Gloria-Otto-10b,Gloria-Neukamm-Otto-14} and the present paper), Armstrong and Smart quantify ergodicity in terms of \textit{linear mixing conditions}. In \cite{Armstrong-Smart-14} they consider the scalar random case under a finite-range of dependence assumption (the strongest of the linear mixing conditions). 
On the one hand they reformulate the Campanato iteration of \cite{Avellaneda-Lin-87} in an abstract functional-analytic form that is oblivious to the PDE (see \cite[Lemma~5.1]{Armstrong-Smart-14}) and essentially states that if a function is close at all scales (down to unit scale) to functions with improvement of flatness, then that function must itself have an improvement of flatness (down to unit scale). Next, they show that if for Dirichlet problems the homogenization error decays algebraically, then $a$-harmonic functions are indeed close at all scales to functions with improvement of flatness (see \cite[Proposition~4.1]{Armstrong-Smart-14}), so that they are themselves Lipschitz (from unit scale onwards). On the other hand, they establish the algebraic (although largely suboptimal) decay of the homogenization error (at an $L^2$-level) within their assumptions (scalar equation, finite-range of dependence, symmetric coefficients) using an ingenious combination of subadditivity methods, duality, and a concentration argument. In contrast to the corrector-based, \textit{intrinsic} regularity approach of Avellaneda and Lin, their method is Euclidean.
The improvement over \cite{Marahrens-Otto-13} is twofold: In terms of regularity ($C^{1,0}$ versus $C^{0,1-}$) and in terms of stochastic integrability (finite nearly-optimal exponential moment versus finite algebraic moments). 

\medskip

In the first version of the present paper (cf.~the arXiv preprint \cite{GNO-preprint} in fall 2014), inspired by the works \cite{Avellaneda-Lin-87} of Avellaneda and Lin and  \cite{Armstrong-Smart-14} of Armstrong and Smart, we developed the first \textit{intrinsic} large-scale regularity theory for (possibly \emph{non-symmetric}) elliptic \emph{systems} (including linear elasticity) in the random setting. 
This large-scale regularity goes beyond Lipschitz estimates and is at the same time qualitative (it applies to merely ergodic coefficients) and quantifiable (in terms of stochastic integrability under 
quantitative mixing conditions). In place of a finite-range of dependence condition, we quantify ergodicity via a Logarithmic Sobolev Inequality in a variant that is flexible enough to cover coefficients with both weak and strong  correlations (which is also new).
On the one hand, this regularity theory (since it applies to the general ergodic case) is of interest for the study of the random conductance model in probability theory (see \cite{Kumagai-14, Biskup-11} for recent surveys). 
Quenched invariance principles and heat kernel estimates for degenerate general ergodic conductances 
indeed received much attention recently --- e.g.~see \cite{ADS16, DNS17}, and the Liouville theorems obtained by the approach of the present paper in \cite{NPhd17} (for the random conductance model) and in \cite{BFO16} (for degenerate elliptic systems).
On the other hand, the quantification of ergodicity via functional inequalities (in the form of multiscale functional inequalities as in the present paper) is particularly well-suited for the class 
of coefficient fields considered in the applied sciences: Most models studied in materials science are indeed generated starting from a (typically hidden) product or Gaussian structure (see~\cite{Torquato-02},
a reference textbook on random heterogeneous materials), and therefore not only satisfy \emph{linear mixing conditions} but also \emph{nonlinear mixing conditions} that can be captured in form of multiscale functional inequalities (see~\cite{DG1,DG2}). As we shall see, such nonlinear mixing conditions are crucial to establish the optimal stochastic integrability for the large-scale regularity theory for these models, see Remark~\ref{rem:AL} below.

\medskip

Between the first posted version \cite{GNO-preprint} and the current version of this paper, 
Armstrong and Mourrat \cite{Armstrong-Mourrat-14}, and subsequently Armstrong, Mourrat and Kuusi \cite{Armstrong-Mourrat-Kuusi-16,Armstrong-Mourrat-Kuusi-17,AKM-book} significantly extended the regularity theory of \cite{Armstrong-Smart-14} in several directions: First, using the framework of the Fitzpatrick duality theory, they were able to treat not only convex integral functionals with quadratic growth, but also monotone operators with quadratic growth (recovering the case of non-symmetric systems we studied in \cite{GNO-preprint}, albeit with a stronger notion of coercivity). Second, they showed that the subadditivity method of \cite{Armstrong-Smart-14} can be pushed forward to treat weaker linear mixing conditions on the coefficients (such as $\alpha$-mixing), lifted the Lipschitz theory to higher-order regularity, and proved moment bounds on the gradient of the corrector that depend on the alpha-mixing decay rate (algebraic decay rate yields algebraic moments, exponential decay rate yields exponential moments, albeit with an arbitrarily small loss of stochastic integrability). Third, they established optimal
growth estimates on the corrector and characterized its large-scale fluctuations 
 (as well as other related quantities) under the finite-range of dependence assumption --- their proof cannot does not extend in a straightforward way to the setting of functional inequalities (and therefore Gaussian statistics).  Likewise, the results of the first version \cite{GNO-preprint} of this paper have been extended in several directions, see discussion in the next paragraph.
Whereas  \cite{Armstrong-Mourrat-14,Armstrong-Mourrat-Kuusi-16,Armstrong-Mourrat-Kuusi-17} rely on variational arguments, the present work and its extensions rely on PDE analysis.

\medskip

To conclude this introduction, let us summarize the achievements of the present contribution.
We develop a  complete \textit{intrinsic large-scale regularity theory} for random elliptic operators that shows that the regularity theory for constant-coefficients elliptic systems extends to random elliptic operators at large scales.
We illustrate this by establishing maximal regularity at the level of  $C^{1,1-}$, i.e.~Schauder theory, and 
of $\dot H^{1,p}$, i.e.~Calder\'on-Zygmund theory. 
A key object that we introduce in this contribution is a stationary random field  $r_*$, which we call the \textit{minimal radius}. 
It characterizes the scale at which the improved intrinsic regularity theory kicks in.  In particular, this large-scale regularity theory is
\begin{itemize}
\item[(i)] \textit{qualitative} in the sense that $r_*$ is almost surely finite under the mere qualitative assumption of stationarity and ergodicity.
This is crucial to prove a Liouville result for subquadratic $a$-harmonic functions, see Corollary~\ref{c};
\item[(ii)] and at the same time  \textit{quantifiable} in the sense that $r_*$ can be proved to have stretched exponential moments if the ensemble of coefficient fields satisfies suitable functional inequalities. 
We prove this on  the representative example of a family of Gaussian coefficient fields the covariance of which decays arbitrarily slowly, cf.  Theorems~\ref{gbis}~\&~\ref{gbis2}, and obtain the optimal stochastic integrability (as opposed to the nearly-optimal stochastic integrability one would get using the approach of \cite{Armstrong-Smart-14}).
\end{itemize}
The definition of $r_*$ is based on an extended corrector $(\phi,\sigma)$, see Lemma~\ref{si}, which allows to represent the residuum of the homogenization error in divergence form. In addition to the standard corrector $\phi$, the extended corrector involves a skew-symmetric tensor field, which we
call $\sigma$. While it has not not been used in stochastic homogenization before, it is a standard object in \textit{periodic} homogenization, see for instance \cite[p.27]{JKO-94}, where it is used to establish quantitative two-scale expansions. The tensor $\sigma$ is related to the flux of the corrector and appears to be as important as the corrector itself. (This does not come as a surprise in view of the very definition of qualitative H-convergence: Weak convergence of the gradient of the solution \emph{and} weak convergence of the flux.) 
This notion of flux corrector turns out to be fundamental in stochastic homogenization, and has been taken up by subsequent work: Intrinsic higher-order regularity and Liouville theorems of all orders \cite{Fischer-Otto-15},
in half spaces \cite{FR-16}, in degenerate environments \cite{BFO16}, quantitative estimates on the corrector and error estimates both 
for ensembles that satisfy functional inequalities and ensembles of finite range of dependence  \cite{GNO-quant,GO-16,Mourrat-Gu-16,BFFO},  characterization of fluctuations (not only of the corrector but also of the solution operator) \cite{DGO,DGO-Gaus,DFGO-Gaus}, notion 
of multipoles \cite{BGO17}, long-time homogenization of the wave equation \cite{BG}, results on non-symmetric discrete models \cite{BMN17}, quantification of invariance principles for the random conductance
in degenerate environments \cite{Andres-Neukamm-1}, etc.

\medskip

Our approach to large-scale regularity is mainly inspired by the work of Avellaneda and Lin.
In particular, it is close to \cite[Section 3.1]{Avellaneda-Lin-87}, see Remark~\ref{rem:AL}. Incidentally, $\sigma$ is not used for that result and only used marginally in that paper
\cite[p.845]{Avellaneda-Lin-87}, and not capitalizing on its skew-symmetry. 
Compared to the series \cite{Armstrong-Mourrat-14,Armstrong-Mourrat-Kuusi-16,Armstrong-Mourrat-Kuusi-17}, our approach has 
advantages both in terms of stochastic results and large-scale regularity.
Compared to the large-scale Lipschitz regularity of \cite{Armstrong-Smart-14} (which we call here the mean-value property), our result is based on a mere \emph{smallness condition} as opposed to an algebraic convergence rate. This has two consequences: First it allows to cover the case of qualitative ergodicity, and second it allows to capture the \textit{optimal stochastic integrability} of $r_*$. 
In terms of large-scale regularity, our approach is intrinsic (which allows us to place ourselves at scale 1, and not at mesoscales) and the large-scale Calder\'on-Zygmund theory we derive here does not come with a loss 
of integrability and holds for all $1<p<\infty$.

\section{Statement of the main results}

\subsection{Assumptions and notation}

We start by specifying our assumptions on the coefficient fields,
and then recall the standard definition of the corrector and the (slightly less standard) definition
of the flux corrector.

\medskip

{\bf Assumptions on the ensemble of coefficient fields}.
Our two assumptions on the space of (admissible) coefficient fields $a$
are pointwise boundedness and uniform ellipticity. Without loss of generality,
we may assume that the bound is unity:
\begin{equation}\label{f.56}
|a(x)\xi|\le|\xi|\quad\mbox{for all}\;\xi\in\mathbb{R}^d\;\mbox{and}\;x\in\mathbb{R}^d.
\end{equation}
We require uniform ellipticity with constant $\lambda>0$ only in the integrated form of
\begin{equation}\label{f.40}
\int\nabla\zeta\cdot a\nabla\zeta\ge\lambda\int|\nabla\zeta|^2\quad
\mbox{for all smooth and compactly supported}\;\zeta.
\end{equation}
This form of ellipticity is weaker than pointwise ellipticity for systems, and allows
one to consider linear elasticity tensors  $a$ that do not necessarily have a sign (so that the duality theory 
at the basis of \cite{Armstrong-Smart-14} does not automatically apply).
Throughout this paper, we use {\it scalar notation} for convenience.
However, we only use arguments that are available in the case of systems, that is,
when $\mathbb{R}$-valued functions $\zeta$ are replaced by fields with values in some
finite dimensional Euclidean space $H$. More precisely, we only use the energy estimate
and consequences thereof, like the Caccioppoli estimate and the higher integrability
coming from the hole-filling argument.
In particular, we do not appeal to De Giorgi's
theory. Clearly, in the case of systems, all constants acquire an additional
dependence on $H$.

\medskip

We now address the minimal assumptions on the ``ensemble'' $\langle\cdot\rangle$,
a probability measure on the space of (admissible) coefficient fields as introduced in \cite[Section~2]{Papanicolaou-Varadhan-79},
which will be assumed throughout the paper. They are related
to the operation of the shift group $\mathbb{R}^d$ on the space of coefficient fields, that is,
for any shift vector $z\in\mathbb{R}^d$ and any coefficient field $a$, the shifted field
$a(\cdot+z)\colon x\mapsto a(x+z)$ is again a coefficient field. 
The first assumption is \emph{stationarity}, which means that for any
shift $z\in\mathbb{R}^d$ the random coefficient fields $a$ and $a(\cdot+z)$
have the same (joint) distribution. 
The second assumption is \emph{ergodicity}, which means that any (integrable) random variable
$F(a)$ that is shift invariant, that is, $F(a(\cdot+z))=F(a)$ for all shift vectors $z\in\mathbb{R}^d$
and $\langle\cdot\rangle$-almost coefficient field $a$, is actually constant, that is $F(a)=\langle F \rangle$
for $\langle\cdot\rangle$-almost every coefficient field~$a$. 
 
\medskip

Under assumptions~\eqref{f.56}, \eqref{f.40}, stationarity, and ergodicity, homogenization (in the sense of Murat and Tartar's notion of H-convergence, see \cite{Murat-78}) holds (for Dirichlet boundary data in the
case of systems due to the weak notion of ellipticity), and the homogenized coefficients also satisfy 
\begin{equation*}
\int\nabla\zeta\cdot a_{\ho}\nabla\zeta\ge\lambda\int|\nabla\zeta|^2\quad
\mbox{for all smooth and compactly supported}\;\zeta 
\end{equation*}
and $|a_{\ho}\xi|\le \frac{(1+\lambda^2)^\frac12}{\lambda}|\xi|$. In particular, $a_\ho$ is uniformly elliptic in the scalar case,
\begin{equation}\label{si.30}
\xi\cdot a_{\ho}\xi\ge\lambda|\xi|^2\quad\mbox{and}\quad|a_{\ho}\xi|\le\frac{1}{\lambda}|\xi|
\quad\mbox{for all}\;\xi\in\mathbb{R}^d,
\end{equation}
and satisfies the Legendre-Hadamard condition in the case of systems. The proof of these statements is the same as the proof given by Murat and Tartar in the case of a second-order elliptic equation (e.g.~see \cite{Murat-78, Murat-97}). (In both cases, the elliptic operator $-\nabla \cdot a_\ho \nabla$ enjoys a full regularity theory.)

\medskip

{\bf Extended corrector}. 
Throughout this paragraph $i=1,\cdots,d$ denotes a coordinate
direction.
We recall the definition of the extended corrector $(\phi_i,\sigma_i)$ in the following lemma, the (rather standard) proof of which is displayed for the reader's convenience (see in particular   \cite[Section 7.2]{JKO-94}).
\begin{lemma}\label{si}
Let $\langle\cdot\rangle$ be stationary and ergodic. Then there exist two random tensor fields
$\{\phi_i\}_{i=1,\cdots,d}$ and $\{\sigma_{ijk}\}_{i,j,k=1,\cdots,d}$ with the following properties:
The gradient fields $\nabla\phi_i$ and $\nabla\sigma_{ijk}$ are stationary, by which we understand that for $\expec{\cdot}$-a.e.~$a$ and any shift vector $z\in\R^d$ we have $\nabla\phi_i(a;\cdot{+}z)=\nabla\phi_i(a(\cdot+z);\cdot)$ and $\nabla\sigma_{ijk}(a;\cdot{+}z)=\nabla\sigma_{ijk}(a(\cdot+z);\cdot)$ a.e.~in $\R^d$, 
and have bounded second moments and vanishing expectations:
\begin{equation}\label{si.2}
\langle|\nabla\phi_i|^2\rangle\le \frac{1}{\lambda^2},\quad
\sum_{j,k=1,\cdots,d}\langle|\nabla\sigma_{ijk}|^2\rangle
\le {4d}(\frac{1}{\lambda^2}+1),\quad
\langle\nabla\phi_i\rangle=\langle\nabla\sigma_{ijk}\rangle=0.
\end{equation}
Moreover, the field $\sigma$ is skew-symmetric in its last indices, that is,
\begin{equation}\label{f.19}
\sigma_{ijk}=-\sigma_{ikj}.
\end{equation}
Finally, we have for $\langle\cdot\rangle$-a.e.~$a$ the equations 
\begin{eqnarray}
-\nabla\cdot a(\nabla\phi_i+e_i)&=&0,\label{f.2}\\
\nabla\cdot\sigma_i&=&q_i,\label{f.5}\\
-\triangle\sigma_{ijk}&=&\partial_jq_{ik}-\partial_kq_{ij},\label{si.5}
\end{eqnarray}
in the distributional sense on $\R^d$
with $\{q_{ij}\}_{i,j=1,\cdots,d}$ given by
\begin{equation}\label{si.6}
q_i:=a(\nabla\phi_i+e_i)-a_{\ho}e_i,\qquad a_{\ho}e_i:=\expec{a(\nabla\phi_i+e_i)},
\end{equation}
where the (distributional) divergence of a tensor field is defined as
$(\nabla\cdot\sigma_i)_j:=\sum_{k=1}^d\partial_k\sigma_{ijk}$.
\qed
\end{lemma}
In the rest of this paper we use the abbreviations $\phi=(\phi_1,\cdots,\phi_d)$,
$\sigma=(\sigma_{ijk})_{i,j,k=1,\dots,d}$, and $\phi_\xi=\sum_{i=1}^d\xi_i\phi_i$ for $\xi\in\R^d$.

\subsection{Regularity theory and the minimal radius}

In the Euclidean context, the $C^{1,\alpha}$-seminorm of a function
measures its local deviation from linear functions.
As is customary in the $C^{1,\alpha}$-theory based on energy estimates, 
that deviation is measured in the $L^2$-sense at the level of gradients, giving
rise to Campanato spaces that are equivalent to H\"older spaces. 
We name this expression ``excess'', cf.~(\ref{o.8}), in (linear) analogy to the
quantity in the regularity theory for minimal surfaces introduced by De Giorgi,
\cite[Teorema 3.3]{DGCP-72}. In the context
of homogenization, it is natural to replace the space of linear functions (which is
$d$-dimensional once one factors out constants) by 
the $d$-dimensional set of harmonic coordinates, that is, $\{x\mapsto
\xi\cdot x+\phi_\xi(x)\}_{\xi\in\mathbb{R}^d}$.
We therefore define for any square-integrable vector field $g$ and ball $B$ the excess as
\begin{equation}\label{o.8}
  \Exc(g;B):=\inf_{\xi\in\R^d}\fint_B|g-(\xi+\nabla\phi_\xi)|^2,
\end{equation}
which measures the deviation of $g$ from $a$-linear functions on $B$.
The theorem below shows that for an $a$-harmonic function $u$, $\Exc(\nabla u;B)$ 
can be controlled provided the corrector is well-behaved in the sense that
$(\phi,\sigma)$ has sufficiently small linear growth in a spatially averaged, but quantitative, sense.
To make this precise, we associate to a given constant $C>0$ a random variable $r_*$ defined by the expression
\begin{equation}\label{O.1}
    r_*:=\inf\Big\{r>0 \,\Big|
    \ \forall R\ge r\,:\,
    \frac{1}{R^2}\fint_{B_R}|(\phi,\sigma)-\fint_{B_R}(\phi,\sigma)|^2\le\frac{1}{C}\Big\},
\end{equation}
with the understanding that $r_*=\infty$ if the set is empty.
Since the extended corrector $(\phi,\sigma)$ exists in all directions $\xi\in\R^d$ almost surely for ergodic coefficients, we implicitly assume above and in the sequel (and in particular for deterministic estimates) that $a$ belongs to the set of full measure of coefficients for which the extended corrector, or at least its stationary gradient on which $r_*$ only depends, is well-defined.

\medskip

In the rest of the article, we write $C$ for a generic positive constant that may change from line to line in the statements and in the proofs (unless otherwise stated), 
and display its dependence upon the parameters of the problem in the form of $C=C(\cdot)$ (e.g.~$C=C(d,\lambda)$ if $C$ only depends on the dimension $d$ and the ellipticity ratio $\lambda$).


\begin{theorem}\label{O}
For any H\"older exponent $\alpha\in(0,1)$ there exists a constant $C(d,\lambda,\alpha)<\infty$ 
with the following properties: Let $r_*$ \emph{(the ``minimal radius associated with $\alpha$'')} be defined by \eqref{O.1} with $C=C(d,\lambda,\alpha)$.
Let $u\in H^1(B_R)$ with $R\geq r_*$ denote an $a$-harmonic function in $B_R$, that is,
\begin{equation}\label{p.12}
-\nabla\cdot a\nabla u=0\quad\mbox{in}\;B_R.
\end{equation}
Then we have ``excess decay'' in the sense of
\begin{equation}\label{o.6}
  \forall r\in[r_*,R],\quad\Exc(\nabla u;B_r)\;\le\; C(d,\lambda,\alpha)\,(\frac{r}{R})^{2\alpha}\,\Exc(\nabla u;B_R).
\end{equation}
Moreover the correctors enjoy the following non-degeneracy property
\begin{equation}\label{p.22}
 \forall r\ge r_*, \forall \xi \in\R^d, \quad \frac{1}{2}|\xi|^2\le\fint_{B_r}|\xi+\nabla\phi_\xi|^2
\le C(d,\lambda,\alpha)|\xi|^2.
\end{equation}
Finally, we have the mean-value property (for which $\alpha>0$ can be fixed, say, $\alpha=\frac{1}{2}$)
\begin{equation}\label{o.30}
  \forall r\in[r_*,R],\quad\fint_{B_r}|\nabla u|^2\;\le\; C(d,\lambda)\,\fint_{B_R}|\nabla u|^2.
\end{equation}
\qed
\end{theorem}

The regularity result provided by Theorem~\ref{O} is ``quenched'', that is, entirely deterministic 
in the sense that the smallness condition is expressed in terms of the
given ``realization'' of $(\phi,\sigma)$. (In case of thermal randomness, one would speak
of a ``pathwise result''.)  
On the one hand, qualitative ergodicity implies that for almost every coefficient field 
the ``minimal radius'' $r_*$ is finite, see \eqref{l.10} in the proof of Corollary~\ref{c}. On the other hand, mild quantitative ergodicity conditions imply that $r_*$ is a random variable with streched exponential moments. The latter is the content of our second main result, see Theorem~\ref{g} below. 
\begin{remark}
Let us compare Theorems~\ref{O} and~\ref{g} (see below) to the results of Armstrong \& Smart in \cite{Armstrong-Smart-14}.
The random field $\mathcal{Y}$ of \cite[Theorem~1.2]{Armstrong-Smart-14}, which plays a similar role as $r_*$ in terms of ``Lipschitz-regularity'' (cf.~\eqref{o.30} in Theorem~\ref{O}),  is defined there as the smallest radius from which on an \emph{algebraic} decay (arbitrarily small, yet fixed) holds, cf.~\cite[Lemma~5.1]{Armstrong-Smart-14}. 
First, our weaker \emph{quantitative smallness} condition  \eqref{O.1} allows one to treat the borderline case of qualitative ergodicity (for which no algebraic decay is available in general) and to avoid any loss of stochastic integrability.
In particular, for coefficients that satisfy a standard Logarithmic Sobolev Inequality, \cite[Theorem~1.2]{Armstrong-Smart-14} would essentially take the form $\expec{\exp(\mathcal Y^{d-\e})}<\infty$ for any $\e>0$ whereas Theorem~\ref{g} yields $\expec{\exp(\frac1C r_*^d)}\le 2$ for some $0<C<\infty$ large enough.
Second, \cite[Theorem~1.2]{Armstrong-Smart-14} is not formulated using harmonic coordinates, and the estimate that essentially corresponds to \eqref{o.6} only holds on ``mesoscales'' and for exponents $0\le \alpha<\beta$ (where $\beta$ depends on $\lambda$ and $d$), that is typically for $r\in [R^{c\alpha},R]$
for some $0<c<1$ instead of $[r_*,R]$.
\qed
\end{remark}
%


A fairly easy consequence of Theorem~\ref{O} in form of (\ref{o.6}) is the Liouville property
for sub{\it quadratic} functions stated in Corollary~\ref{c}. This partially answers to the affirmative 
a specific version of a question raised
in \cite[Question 5, p.33]{BCKY-11} 
on whether the dimension of the space of $a$-harmonic functions of a given
growth exponent agrees with the dimension in the Euclidean case.
The answer is partial, because only subquadratic growth is treated, and deals with a special case,
because only the case of uniformly elliptic coefficient fields is treated. In the 
even more special case of periodic 
coefficient fields the answer is affirmative for all growth rates
\cite{Avellaneda-Lin-89}. Our qualitative result
holds, as it should, under the purely qualitative condition of ergodicity.

\begin{corollary}\label{c}
Let $\langle\cdot\rangle$ be stationary and ergodic. Then for $\langle\cdot\rangle$-a.e.~coefficient
field $a$, the following Liouville property holds: Suppose that $u$ is $a$-harmonic,
that is $-\nabla\cdot a\nabla u=0$ in all $\mathbb{R}^d$, and that it grows subquadratically in the sense
that there exists an exponent $\alpha<1$ such that
\begin{equation}\label{l.11}
\lim_{R\uparrow\infty}R^{-2(1+\alpha)}\fint_{B_R}u^2=0.
\end{equation}
Then $u$ is $a$-linear in the sense that there exists $(c,\xi)\in\mathbb{R}\times\mathbb{R}^d$ such that
\begin{equation}\label{l.12}
u(x)=c+\xi\cdot x+\phi_\xi(x)\quad\mbox{for Lebesgue-a.\ e.}\;x\in\mathbb{R}^d.
\end{equation}
\qed
\end{corollary}

\medskip

The next corollary establishes a $C^{1,1-}$-a priori estimate for $a$-harmonic functions 
similar to the one for plain harmonic functions. 
There are two restrictions: As expected from Theorem~\ref{O}, such an estimate only holds 
on scales that are large with respect to the minimal radius $r_*$, see (\ref{m.17}). 
Moreover, it only holds for an {\it effective} gradient
which is the projection of the microscopic gradient onto $a$-linear
functions, a projection localized at the level of the minimal radius $r_*$, cf.\  (\ref{m.15}). 

\begin{corollary}\label{co}
Let $\alpha\in(0,1)$ be given and let $r_*$ denote the associated minimal radius (see Theorem~\ref{O}). We denote by $r_*(a,x):=r_*(a(\cdot+x))$ the stationary extension of the minimal radius, so that
\begin{equation}\label{m.17}
\forall\;x\in\mathbb{R}^d,\;\;\forall r\ge r_*(x),  \quad
\frac{1}{r^2}\fint_{B_r(x)}|(\phi,\sigma)-\fint_{B_r(x)}(\phi,\sigma)|^2\,\le\,
\frac{1}{C(d,\lambda,\alpha)} .
\end{equation}
For any $a$-harmonic function $u$ in a ball $B_R$,
cf.~(\ref{p.12}), consider the vectors $\xi_+$ and $\xi_-$ characterized by
\begin{equation}\label{m.15}
\fint_{B_{r_*(\pm x)}(\pm x)}|\nabla u-(\xi_\pm+\nabla\phi_{\xi_\pm})|^2
=\Exc(\nabla u;B_{r_*(\pm x)}(\pm x)),
\end{equation}
which we think of as the effective gradient of $u$ in $x$ and $-x$ at scale $r_*$,
respectively. Then we have, provided $R\ge 8\max\{|x|,r_*(\pm x)\}$,
\begin{equation}\label{m.11}
|\xi_+-\xi_-|^2
\le C(d,\lambda,\alpha)(\frac{\max\{|x|,r_*(\pm x)\}}{R})^{2\alpha}\Exc(\nabla u;B_R).
\end{equation}
\qed
\end{corollary}
Another application of Theorem~1 are intrinsic Schauder estimates for elliptic systems in divergence form.
\begin{corollary}[Large-scale Schauder estimates]\label{cor:Lip}
Let a H\"older exponent $\alpha\in(0,1)$ be given and let $r_*$ denote the minimal radius  (\ref{O.1}) associated 
with the constant $C=C(d,\lambda,\alpha')$ of Theorem~\ref{O} for some fixed $\alpha'\in(\alpha,1)$. Below the notation $\lesssim$ stands for $\le C$ for a generic multiplicative 
constant $C$ that only depends on $d$, $\lambda$, and $\alpha$.
Let the function $u$ (with square-integrable gradient) and the 
(square-integrable) vector fields $g$ and $h$ on $B_R$ be related by
\begin{equation}\label{cw51}
-\nabla\cdot a(\nabla u+g)=\nabla\cdot h\qquad\text{in }B_R.
\end{equation}
Then we have  
\begin{align}\label{eq:SE:1}
  \lefteqn{\sup_{r\in[r_*,R]}   \frac1{r^{2\alpha}}  \Exc(\nabla u+g;B_{r})}\nonumber\\
&\lesssim\frac{1}{R^{2\alpha}} \Exc(\nabla u+g;B_{R})
+\sup_{r\in [r_*,R]}\frac1{r^{2\alpha}}\fint_{B_r}(|g-\fint_{B_r}g|^2+|h-\fint_{B_r}h|^2).
\end{align}
If $R=+\infty$, we obtain
\begin{equation}\label{eq:SE:3}
\sup_{r\ge r_*}\frac1{r^{2\alpha}}  \Exc(\nabla u+g;B_{r})
\lesssim\sup_{r\ge r_*}\frac1{r^{2\alpha}}\fint_{B_r}(|g-\fint_{B_r}g|^2+|h-\fint_{B_r}h|^2)
\end{equation}
(with the understanding that $\nabla u+g$ is square-integrable on $\R^d$).
For later purpose, we also state the following extension of \eqref{o.30}
\begin{align}\label{eq:SE:2}
\lefteqn{\sup_{r\in[r_*,R]}      \fint_{B_{r}}|\nabla u+g|^2}\nonumber\\
&\lesssim
\fint_{B_{R}}|\nabla u+g|^2 + \sup_{r\in [r_*,R]}
\big(\frac{R}{r}\big)^{2\alpha}\fint_{B_r}(|g-\fint_{B_r}g|^2+|h-\fint_{B_r}h|^2).
\end{align}
\qed
\end{corollary}

A last application of  Theorem~\ref{O} are the following large-scale Calder\'on-Zygmund estimates, 
which are an improved form of a result proved in the first version of \cite{DGO}.
\begin{corollary}[Large-scale Calder\'on-Zygmund estimates]\label{cor:Lp}
Let $C_0=C(d,\lambda,\frac12)<\infty$ be the constant associated with $\alpha=\frac12$ in Theorem~\ref{O}.
There exists a $\felix$-Lipschitz stationary field $\ru$ that satisfies
$r_*(C_0)\le \ru \le r_*(\Felix^{d+2}C_0)$ (where $r_*(C)$ is the minimal radius associated with
the constant $C$ as defined in \eqref{O.1}) and such that for any suitably decaying 
scalar field $u$ and vector field $g$ related in $\R^d$ by
\begin{equation}\label{l1-0}
-\nabla \cdot a \nabla  u=\nabla \cdot g,
\end{equation}
and any exponent $1< p<\infty$, we have 
\begin{equation}\label{I1}
\left(\int\Big(\fint_{B_*(x)}|\nabla u|^2\Big)^\frac{p}{2}dx\right)^{\frac1p}\,\lesssim\,
\left(\int\Big(\fint_{B_*(x)}|g|^2\Big)^\frac{p}{2}dx\right)^{\frac1p},
\end{equation}
where $B_*(x):=B_{\underline{r_*}(x)}(x)$, and  $\lesssim$ means $\le C$ for a constant only depending on $d$, $\lambda$, $p$.
\qed
\end{corollary}
There is nothing particular in the definition of $C_0=C(d,\lambda,\alpha)$, and we may choose $C_0=C(d,\lambda,\alpha)$ for any $\alpha>0$ so that large-scale
$C^{1,\alpha}$-regularity also holds on scales $R\ge \ru$ (which we use in the proof of Corollary~\ref{cor:Lp} only in the form of the mean-value property).
For applications in \cite{GNO-quant} and \cite{DGO-Gaus}, we shall need weighted versions of these large-scale Calder\'on-Zygmund estimates --- which
are, as opposed to standard weighted Calder\'on-Zygmund estimates (e.g.~\cite[Chapter~V, Section~3]{Stein}),
restricted to a very specific class of weights (although for these specific weights, the condition $0\le \gamma < d(p-1)$ below is equivalent 
to the condition $\omega \in A_p$, the correct Muckenhoupt class).
\begin{corollary}[Large-scale weighted Calder\'on-Zygmund estimates]\label{cor:Lp-weight}
Let $2\le p<\infty$, $\gamma<d(p-1)$, and $\omega=\omega(r)>0$ satisfy
\begin{align}\label{EM08}
\omega(r)\le\omega(R)\le(\frac{R}{r})^\gamma\omega(r)\quad\mbox{for all}\;r\le R.
\end{align}
In the notation of Corollary~\ref{cor:Lp}, 
\begin{equation}\label{EM10}
{\Big(\int\omega(|x|+\underline{r_*}(0))\big(\fint_{B_{*}(x)}|\nabla u|^2\big)^\frac{p}{2}dx\Big)^\frac{1}{p}}
\lesssim
\Big(\int\omega(|x|+\underline{r_*}(0))\big(\fint_{B_{*}(x)}|g|^2\big)^\frac{p}{2}dx\Big)^\frac{1}{p},
\end{equation}
where $\lesssim$ means $\le C$ with $C$ only depending on $d$, $\lambda$, $p$, and $\gamma$.
\qed
\end{corollary}
Note that a  weaker (and unweighted) version of Corollary~\ref{cor:Lp} is proved in \cite{AD}: However,  there is a loss in the exponent in the LHS (and only $p\ge 2$ is addressed).
Loosely speaking, Corollaries~\ref{cor:Lip} \&~\ref{cor:Lp} state that from the minimal radius $r_*$ onwards,
one is in the regime of $C^{1,\alpha}$- and $L^p$-maximal regularity, respectively. 
For higher-order regularity, we refer the reader to the subsequent work~\cite{Fischer-Otto-15,Fischer-Otto-17} by Fischer and the third author.
\begin{remark}\label{rem:duality}
At the price of including the adjoint of the extended corrector  (that is, the extended corrector associated with
the pointwise tranpose coefficient field $a^*$) in the definition of $r_*$, one may w.l.o.g.~assume that
the above regularity theory holds for both the operators $-\nabla \cdot a\nabla$ and $-\nabla \cdot a^*\nabla$. 
The estimates of $r_*$ obtained in this article remain unchanged since pointwise transposition is a local linear operation that does not
change statistical properties.
This will be used in the proofs of \eqref{I1} (for the range $1<p\le 2$) and  Corollary~\ref{cor:Lp-weight} for which we argue by duality.
\qed
\end{remark}

\medskip

We close this section by stating the main ingredient for Theorem~\ref{O}; namely Proposition~\ref{F}. It relies on the following observation:
Harmonic functions $u$ have the property that for all radii $r\le R$
there exists a $\xi\in\mathbb{R}^d$ (in fact, $\xi=\nabla u(0)$ or $\xi=\fint_{B_r}\nabla u$ will do) such that 
\begin{equation}\label{F.2}
\fint_{B_r}|\nabla u-\xi|^2\le C(d)(\frac{r}{R})^2\fint_{B_R}|\nabla u|^2.
\end{equation}
Proposition~\ref{F} establishes a perturbation of (\ref{F.2}) for $a$-harmonic functions,
provided the affine function $x\mapsto\xi\cdot x$ is replaced by its
$a$-harmonic version $x\mapsto\xi\cdot x+\phi_\xi(x)$, and 
where the perturbation is controlled by the amount of linear growth of the corrector $(\phi,\sigma)$.

\begin{proposition}\label{F}
There exists an exponent $\e=\e(d,\lambda)>0$ with the following property.
Let $u$ be an $a$-harmonic function in a ball $B_R$ with radius $R>0$, 
cf.~(\ref{p.12}).
Then for all $r\le R$, there exists a vector $\xi\in\mathbb{R}^d$ such
that 
\begin{equation}\label{F.1}
\fint_{B_r}|\nabla u-(\xi+\nabla\phi_\xi)|^2
\,\le\, C(d,\lambda)
\Big((\frac{r}{R})^2+\delta^{2\e}(\frac{R}{r})^{d+2}\Big)\fint_{B_R}|\nabla u|^2,
\end{equation}
where we have set for abbreviation
\begin{equation}\label{p.21}
  \delta:=\frac{1}{R}\Big(\fint_{B_R}|(\phi,\sigma)-\fint_{B_R}(\phi,\sigma)|^2 \Big)^\frac12.
\end{equation}
Moreover, we have the following non-degeneracy property
\begin{equation}\label{p.20}
\frac1{C(d,\lambda)}(1-C(d)\delta) |\xi|\le
\Big(\fint_{B_{\frac R2}}|\xi+\nabla\phi_\xi|^2\Big)^\frac12 \le
C(d,\lambda)(1+\delta)|\xi|\quad\mbox{for all}\;\xi\in\mathbb{R}^d.
\end{equation}
\qed
\end{proposition}

\begin{remark}\label{rem:AL}
Theorem~\ref{O} and its main ingredient, Proposition~\ref{F}, should be compared to the work of Avellaneda \& Lin, more
precisely, to \cite[Section 3.1]{Avellaneda-Lin-87}: Like in~\eqref{F.1}, the distance between $\nabla u$ and $\xi+\nabla\phi_\xi$ for a suitable 
$\xi$ (there, it is given by the spatial average of $\nabla u$) is monitored --- however, on an $L^\infty$ instead of an $H^1$-level, see \cite[Lemma 14]{Avellaneda-Lin-87}, 
which is the analogue of Proposition~\ref{F}. Like for Proposition~\ref{F}, \cite[Lemma~14]{Avellaneda-Lin-87} is a perturbation of an estimate for the 
(constant-coefficients) homogenized operator. In fact, \cite[Lemma~14]{Avellaneda-Lin-87} does not use periodicity in an explicit way, but only H-convergence 
of the elliptic operator $-\nabla\cdot a\nabla$ (see \cite{Murat-78,Murat-97}), in its scaled-down version,
to the homogenized limit $-\nabla\cdot a_{\ho}\nabla$. More precisely, it uses an upgraded version of H-convergence, where the
solutions converge in $L^\infty$, an upgrade which in case of scalar equations may be obtained by appealing to the uniform H\"older
regularity of $a$-harmonic functions (De Giorgi's result) and which in \cite[Section 2.2]{Avellaneda-Lin-87} is obtained in the system's
case by first deriving a $C^{0,\alpha}$-estimate by a similar strategy to the $C^{1,\alpha}$-estimate. Incidentally, \cite[Lemma~14]{Avellaneda-Lin-87}
also uses implicitly the sublinear growth of the corrector $\phi$. The main new ingredient in Proposition~\ref{F}
is a \emph{quantification of H-convergence} (which is a purely qualitative concept) in terms of the \emph{sublinear growth of $\phi$ and
$\sigma$}. 
This also requires a suitable cut-off argument since we want to use the {\it
whole}-space corrector $(\phi,\sigma)$ and thus need to introduce a boundary layer.
The passage from Proposition~\ref{F} to Theorem~\ref{O} mimics the
passage from \cite[Lemma 14]{Avellaneda-Lin-87} to \cite[Lemma~15]{Avellaneda-Lin-87}.  Note that, in contrast to our work, \cite{Avellaneda-Lin-87} assumes smoothness of $a$ which helps handle the small scales.
\qed
\end{remark}

\medskip


In view of Theorem~\ref{O} and Corollaries~\ref{cor:Lip} and~\ref{cor:Lp}, it is clearly of interest to control the size
of the stationary random field $r_*$, which is almost surely finite 
under mere ergodicity of the coefficients, cf.~the proof of Corollary~\ref{c}. In order to obtain a quantitative
control, one needs to make quantitative assumptions.

\subsection{Control of the minimal radius}

In this section $\alpha\in(0,1)$ is fixed and we denote by $r_*$ the associated minimal radius (see Theorem~\ref{O}), so that the mean-value property~\eqref{o.30} for gradients of
$a$-harmonic functions $u$ on $B_R$ holds for balls centered at the origin and of radius larger than $r_*$:
\begin{equation}\label{o37}
\forall\;r\in[r_*,R],\quad\fint_{B_r}|\nabla u|^2\le C(d,\lambda,\alpha)\fint_{B_R}|\nabla u|^2.
\end{equation}
We shall also use the $C^{1,\alpha}$-Schauder estimate from Corollary~\ref{cor:Lip}.

\medskip

We consider two extreme situations on the statistics of $\expec{\cdot}$:
\begin{itemize}
\item Strong decay of correlation that leads to the best integrability of $r_*$ one can expect, cf. Theorem~\ref{g};
\item Arbitrarily slow decay of correlation that leads to weaker (typically stretched exponential or algebraic) integrability of $r_*$, cf. Theorems~\ref{gbis} and~\ref{gbis2}.
\end{itemize}
These results are proved using a mixing condition in the form of functional inequalities,  which ensure strong nonlinear concentration properties (typically stronger than other more standard
mixing conditions).
We split the rest of this section into two parts.
In the first part we specialize to standard functional inequalities, state Theorem~\ref{g}, and describe the general structure of the proof in that setting.
In the second part, we address more general fields based on multiscale functional inequalities (as first introduced 
using non-uniform partitions coarsening away from the origin in the first version of this article, and recently reformulated and extensively 
studied in the form of the multiscale inequalities we use here in \cite{DG1,DG2}), which allows us to treat most coefficient fields considered in the applied
sciences (cf.~e.g.~\cite{Torquato-02}). The structure of the proof is similar for Theorems~\ref{g},~\ref{gbis}, and~\ref{gbis2}.
As a general rule, in the actual proofs of these results, we first focus on standard functional inequalities so that the core argument appears as clear as possible,
and only later on address the multiscale case.

\subsubsection{Control of the minimal radius using standard functional inequalities}

By definition, controlling the minimal radius $r_*$ means controlling the sublinear growth of the corrector $(\phi,\sigma)$.
The sublinear growth of the corrector (a key element to most
homogenization results) is a result of  the cancellations coming from $\langle\nabla(\phi,\sigma)\rangle=0$,
which due to ergodicity translates into $\lim_{r\uparrow\infty}\fint_{B_r}\nabla(\phi,\sigma)=0$, cf.\ the proof of Corollary~\ref{c}.
The quantification of this relies on two distinct ingredients:
\begin{itemize}
\item On the one hand, one needs good {\it locality} properties of  
$\frac{1}{r^2}\fint_{B_r}|(\phi,\sigma)-\fint_{B_r}(\phi,\sigma)|^2$, and thus of $(\phi,\sigma)$.
By this it is meant
that the solution $(\phi,\sigma)$ of the elliptic system (\ref{f.2}) \& (\ref{si.5})
at some point $x$ depends only weakly on the coefficient field $a$ far away from that point.
To establish this locality, we shall use the modified extended corrector $(\phi_T,\sigma_T)$, cf. \eqref{o14}--\eqref{o17} below, 
and relate the sublinear growth of $(\phi,\sigma)$ to that of $(\phi_T,\sigma_T)$, cf.~Proposition~\ref{p1}.
\item On the other hand, one needs good {\it mixing} properties of the ensemble $\langle\cdot\rangle$
of random coefficient fields $a$. By this it is meant that the random value of $a$ at some point $x$ 
statistically depends only weakly on its values far away. For that purpose we appeal to the \textit{multiscale functional inequalities} (MFI) introduced by Duerinckx and the first author in \cite{DG1}.
\end{itemize}
We start by recalling the standard logarithmic Sobolev inequality, cf.~\cite{Ledoux-97,Ledoux-01}.  In what follows, $\sup$ is a shorthand notation for the essential supremum.
\begin{definition}\cite[Section~2]{DG1}\label{def:sLSI}
For all $\ell\ge 0$ and $x\in \R^d$, denote by  $B_{\ell+1}(x)$ the ball centered at $x\in \R^d$ and of radius $\ell+1$.
We consider two types of derivative for a function $F$ on the space of coefficient fields $a$:
\begin{itemize}
\item $|\partial^{\fun}_{x,\ell+1} F|$  denotes  the $L^1(B_{\ell+1}(x))$-norm of the functional (or Malliavin) derivative of $F$ with respect to $a$, that is
$$
|\partial^{\fun}_{x,\ell+1} F|\,:=\,\int_{B_{\ell+1}(x)} |\frac{\partial F}{\partial a(z)}|dz.
$$
\item $|\partial^{\osc}_{x,\ell+1} F|$ denotes the (essential) oscillation\footnote{To make this quantity measurable, the supremum and infimum in the oscillation have to be slightly modified, see \cite{DG1} based on \cite{Barron}. This is not essential in this article since we shall bound such oscillations by measurable quantities.} of $F$ with respect to the restriction of $a$ on $B_{\ell+1}(x)$, that is,
$$
|\partial^{\osc}_{x,\ell+1} F(a)|\,:=\, \sup\big\{F(a')- F(a'')\,:\,a'=a''=a\text{ in }\R^d\setminus B_{\ell+1}(x)\,\big\}.
$$
\end{itemize}
We say that $\expec{\cdot}$ satisfies a standard Logarithmic Sobolev Inequality (LSI) if there exists $\kappa>0$ such that for all random variables $F$ we have
\begin{equation}
\forall L\geq 1,\quad \Ent_L(F):=\expec{F^2\log F^2}_L-\expec{F^2}_L\expec{\log F^2}_L \leq \frac1\kappa \expec{ \int_{\R^d} |\partial^{\fun/\osc}_{x,1} F|^2 dx}_L,\label{e:def-sLSI}
\end{equation}
where $\expec{\cdot}_L$ denotes the ensemble scaled by $L$, i.e.~for rescaled coefficient fields $a(L\cdot)$.
\qed
\end{definition}
\begin{remark}\label{rem:scalingMFI}
Imposing \eqref{e:def-sLSI} for all $L\ge 1$ (instead of just $L=1$) is not restrictive: It is automatically met in the constructive approach of \cite{DG2}, and essentially follows
from the fact that mixing properties improve under such rescaling --- see \cite[Remark~2.2]{DG2} for details. 
\qed
\end{remark}
Under assumption~\eqref{e:def-sLSI}, we have the following result on the integrability of $r_*$.
\begin{theorem}\label{g}
Assume that $\langle\cdot\rangle$ is stationary and satisfies the standard Logarithmic Sobolev Inequality~\eqref{e:def-sLSI}.
For all $\alpha\in(0,1)$, consider $r_*$  defined in (\ref{O.1}) with constant $C(d,\lambda,\alpha)$.
Then there exists a positive constant $C=C(d,\lambda,\kappa,\alpha)$ such that
\begin{equation}\label{g.0}
\big\langle\exp\big({\textstyle\frac{1}{C}}  r_*^d \big)\big\rangle\le  2.
\end{equation}
\qed
\end{theorem}
The following remark shows that \eqref{g.0} is the best stochastic integrability one can hope for.
\begin{remark}
Consider the case of discrete stationarity in
the form of a Bernoulli random checkerboard and a
scalar equation. On a square of side length $R$, the probability to
approximate the coefficients the classical counter example
to (large-scale)  De Giorgi's regularity (e.g.~the quasi-conformal mapping of \cite[Section~12.1]{Gilbarg-Trudinger-98} in dimension 2) is at least $(\frac12)^{R^{d}}$,
in which case $r_*$ has to be larger than $R$. As already observed
by Armstrong and Smart in \cite{Armstrong-Smart-14}, this directly
implies that $\expec{\exp(\frac{1}{C}g(r_*))}=\infty$ if $g\ge 0$ is
such that $\liminf_{r\to \infty}\frac{g(r)}{r^d}=\infty$. On the
other hand, the fact that \eqref{g.0} holds owes to the large
constant $0<C<\infty$. Indeed, this constant quantifies the fraction
of the subset of coefficients that do not satisfy the
Lipschitz regularity at size $R$, which is best seen by writing \eqref{g.0} in form of
$
\expec{I(r_*\ge R)}\,\leq \, 2 \exp(-\frac{1}{C}R^d).
$
In the case of the random checkerboard, this quantifies the ratio
between the number of ``bad" coefficients on a square of size $R$ and the total
number $2^{R^d}$ of realizations: It 
does not exceed $2\big(\exp(-\frac{1}{C})\big)^{R^d}$.
\qed 
\end{remark}

Let us now describe the main ingredients to the proof of Theorem~\ref{g}, which relies on the one hand on five deterministic results, and on the other hand on (nonlinear) concentration properties.
We start with the five deterministic ingredients to the proof of Theorem~\ref{g}, namely
\begin{itemize}
\item Proposition~\ref{p1}, which states that the sublinear growth (in a locally square-averaged form) of the extended corrector is controlled by the sublinear growth of the \textit{modified} extended corrector
$(\phi_T,\sigma_T)$ introduced below;
\item Proposition~\ref{p1bis}, which shows that  the {modified} extended corrector can be controlled by the \emph{average} of the (squared) $H^{-1}$-norm of field
and flux of a \emph{more localized} modified extended corrector $(\phi_t,\sigma_t)$;
\item Proposition~\ref{p1ter}, which shows that this squared $H^{-1}$-norm is indeed approximately local -- a property that we use to control fluctuations of its spatial averages by concentration of measure (cf.~Lemma~\ref{lem:concLSI});
\item Lemma~\ref{L:cw}, which allows to control the $H^{-1}$-norm of field and flux by their  spatial averages, on the level of second stochastic moments;
\item Lemma~\ref{lem:sensitivity}, which establishes a deterministic sensitivity estimate for these spatial averages.
\end{itemize}
The modified extended corrector $(\phi_T,\sigma_T)$ associated with a fixed direction $e$ and a cut-off scale $\sqrt T\geq 1$ is defined as the unique solution to
\begin{eqnarray}\label{o14}
\frac{1}{T}\phi_T-\nabla\cdot a(\nabla\phi_T+e)&=&0,
\\
\label{o16}
q_T&:=&a(\nabla\phi_T+e),\\
\label{o17}
\frac{1}{T}\sigma_T-\triangle\sigma_T&=&\nabla\times q_T,
\end{eqnarray}
(in the distributional sense on $\R^d$) in the class
\begin{equation*}
\Big\{v\in H^1_\loc(\R^d)\,:\, \sup_{x\in \R^d} \int_{B_1(x)} |v|^2+|\nabla v|^2<\infty\,\Big\},
\end{equation*}
where $B_1(x)$ denotes the unit ball centered at $x$.
In \eqref{o17}  and in the rest of this paper, we use the notation $\nabla\times$ inspired by the case of $d=3$ for the differential operator in the RHS of \eqref{si.5}.
Note that the ``massive'' term $\frac1T\phi_T$ introduces a length scale $R=\sqrt{T}$ that plays a role in the sequel.
These whole-space problems are well-posed for all coefficient fields on a purely deterministic level (see e.g.~\cite[Appendix~A.1]{Gloria-Otto-10b}). Moreover,
as it is well-known in homogenization, if $\expec{\cdot}$ is stationary and ergodic, then $\expec{\cdot}$-almost surely  $(\nabla\phi,\nabla\sigma)$ can be recovered as the weak limit as $T\uparrow\infty$ in $L^2_\loc(\R^d)$ of $(\nabla\phi_T,\nabla\sigma_T)$, while $\frac1T \phi_T$ and $\frac1T \sigma_T$ converge to zero in $L^2_\loc(\R^d)$
(see for instance Step~1 in the proof of Proposition~\ref{p1}).

\medskip

Proposition~\ref{p1} below states that the sublinear growth of the extended
corrector $(\phi,\sigma)$ and thus $r_*$ is controlled by the sublinear growth of the modified extended corrector
$\{(\phi_T,\sigma_T)\}_T$, provided both are slightly quantified.
\begin{proposition}\label{p1}
Suppose that for some exponent $\nu>0$ and radius $r_{**}$ we have
for all dyadic $R=2^k$
\begin{align}\label{gr05}
\fint_{B_R}\frac{1}{T}|(\phi_T,\sigma_T)|^2\le(\frac{r_{**}}{R})^{2\nu}\quad\mbox{for all}\;R\ge r_{**}\;\mbox{and}\;T=R^2.
\end{align}
Then there exists a constant $C=C(d,\lambda,\alpha,\nu)$ such that 
\begin{align}\label{gr01}
\frac{1}{R^2}\fint_{B_R}|(\phi,\sigma)-\fint_{B_R}(\phi,\sigma)|^2\le C(\frac{r_{**}}{R})^{2\nu}
\quad\mbox{for all}\;R\ge r_{**};
\end{align}
furthermore,
\begin{align}\label{gr02}
r_*\le Cr_{**}.
\end{align}
\qed
\end{proposition}
The following Proposition \ref{p1bis} relates the sublinear growth of the modified extended corrector $(\phi_T,\sigma_T)$
to the smallness of a negative norm of the corresponding field $\nabla\phi_T$ and flux $q_T=a(\nabla\phi_T+e)$. 
More precisely, the negative norm monitors the fluctuations
$q_T-\langle q_T\rangle$ (note that $\nabla\phi_T-\langle\nabla\phi_T\rangle=\nabla\phi_T$).
Whereas the (homogeneous) $H^{-1}$-norm of $\nabla\phi_T$ is conveniently given by the $L^2$-norm of $\phi_T$ itself, 
for the negative norm of $q_T-\langle q_T\rangle$ we introduce the vector field
\begin{align}\label{cw01}
\frac{1}{T}g_T-\triangle g_T=\frac{1}{\sqrt{T}}(q_T-\langle q_T\rangle)
\end{align}
and take the $L^2$-norm of $(g_T,\sqrt{T}\nabla g_T)$ as a version of
the $H^{-1}$-norm of $q_T-\langle q_T\rangle$ with a cut-off for scales $\gtrsim\sqrt{T}$ (the latter being
important for the locality property).
The normalization in (\ref{cw01}) is chosen such that $g_T$ and $\phi_T$ live on the same footing.
The point of Proposition \ref{p1bis} is that the sublinear growth of $(\phi_T,\sigma_T)$ is controlled
by negative norms of $(\nabla\phi_t,q_t-\langle q_t\rangle)$ for any $t\le T$.
It is convenient to introduce a scale of exponential averaging functions 
\begin{align}\label{cw15}
\omega_T(x):={\textstyle\frac{1}{|\partial B_1|(d-1)!}}\frac{1}{(C\sqrt{T})^d}\exp(-\frac{|x|}{C\sqrt{T}}),
\end{align}
with a constant $C=C(d,\lambda)$ chosen such that the localized energy estimate for
$\frac{1}{T}-\nabla\cdot a\nabla$ holds, see (\ref{cw16}) below --- for some estimates, we shall further need to increase the constant $C$
without changing notation.

\begin{proposition}\label{p1bis}
For all $0<t\le T$ we have
\begin{align}\label{cw02}
\int\omega_T\frac{1}{T}|(\phi_T,\sigma_T)|^2\le C(d,\lambda)
\int\omega_T\big(\frac{1}{t}\phi_t^2+\frac{1}{t}|g_t|^2+|\nabla g_t|^2\big).
\end{align}
\qed
\end{proposition}
Proposition \ref{p1ter} establishes the locality of the RHS integrand $F_t$ in (\ref{cw02}).
More precisely, it considers local averages $F_t$ on scale $\sqrt{t}$ of the integrand, cf.~(\ref{cw03}). 
By locality of such a random variable, i.e.~a function(al) $F_t=F_t(a)$ of the coefficient field $a$,
we understand that it essentially does not depend on $a$ at distances $\gg\sqrt{t}$. 
Proposition \ref{p1ter} establishes that this is true up to an exponential error, see (\ref{cw04}).
\begin{proposition}\label{p1ter}
For all $t>0$ consider the function $F_t=F_t(a)$ given by
\begin{align}\label{cw03}
F_t:=\int\omega_t\big(\frac{1}{t}\phi_t^2+\frac{1}{t}|g_t|^2+|\nabla g_t|^2\big).
\end{align}
Then for every $\lambda$-uniformly elliptic coefficient field $a$ we have
\begin{align}\label{cw20}
|F_t(a)|\le C(d,\lambda),
\end{align}
and $F_t$ is approximately $\sqrt t$-local in the sense that for all balls $B_R$
and all $\lambda$-uniformly elliptic coefficient fields $a,a'$ we have
\begin{equation}\label{cw04}
  \begin{aligned}
    &a=a'\text{ in }B_R\ \Rightarrow\ |F_t(a)-F_t(a')|\le
    C(d,\lambda)\exp\big(-\frac{1}{C(d,\lambda)}\frac{R}{\sqrt{t}}\big).
  \end{aligned}
\end{equation}
\qed
\end{proposition}
It remains to provide the deterministic elements for the estimate of the expectation of
$F_t$ defined in (\ref{cw03}), which by stationarity is given by
$\langle\frac{1}{t}\phi_t^2+\frac{1}{t}|g_t|^2+|\nabla g_t|^2\rangle$. 
The following lemma shows that this truncated version of the $H^{-1}$-norm of the field/flux pair $(\nabla\phi_t,q_t-\langle q_t\rangle)$ can be estimated by spatial averages $((\nabla \phi_T)_{*t},(q_T)_{*t}-\expec{q_T})$
of  $(\nabla \phi_T,q_T-\expec{q_T})$, where for a field $f$ we denote by $f_{*t}$ the convolution by a Gaussian of variance $t$ (and thus length-scale $\sqrt{t}$). 
Because of the semi-group property, it is indeed convenient to take spatial averages by convolving with Gaussians.
\begin{lemma}\label{L:cw}
For all $T>0$ we have
\begin{equation}\label{cw06}
{\langle\frac{1}{T}\phi_T^2+\frac{1}{T}|g_T|^2+|\nabla g_T|^2\rangle}\,\le\, C(d)
\frac{1}{T}\int_0^T\langle|(\nabla\phi_{T})_{*t}|^2+|(q_{T})_{*t}-\langle q_T\rangle|^2\rangle dt.
\end{equation}
\qed
\end{lemma}
The next result is a (suboptimal) sensitivity estimate for the RHS of~\eqref{cw06}.
\begin{lemma}\label{lem:sensitivity}
There exist an exponent $\e=\e(d,\lambda)>0$ (coming from hole-filling) and a constant $C= C(d,\lambda)$ such that for all $1\leq t\leq T$ we have 
\begin{equation}\label{e.lem:sensitivity}
(\ell+1)^{-d}{\int_{\R^d} \Big|\partial^{\fun/\osc}_{x,\ell+1}\big(\nabla\phi_T,q_T\big)_{*t}\Big|^2dx}\,\leq\, C \Big(\frac{\sqrt T}{\sqrt t}\Big)^d (\frac{\ell+1}{\sqrt{T}}\wedge 1)^{\e d}.
\end{equation}
\qed
\end{lemma}
We conclude with the stochastic arguments.
Equipped with Lemmas~\ref{L:cw} and~\ref{lem:sensitivity}, we obtain using LSI the following control of the expectation of $F_t$:
\begin{corollary}\label{c.fluct.T-LSI}
Let $\langle\cdot\rangle$ be stationary and satisfy the standard Logarithmic Sobolev Inequality~\eqref{e:def-sLSI}.
Then there exists a constant $C=C(d,\lambda,\kappa)$ such that for all ${T}>0$ we have
\begin{equation} \label{o8-LSI}
\langle\frac{1}{T}\phi_T^2+\frac{1}{T}|g_T|^2+|\nabla g_T|^2\rangle\,\leq\, CT^{-\e},
\end{equation}
where $\e=\e(d,\lambda)>0$ is defined in Lemma~\ref{lem:sensitivity}.
\qed
\end{corollary}
We finally recall a result to control fluctuations of random variables that behave like simple averages (cf.~\cite{DG1}, see also \cite{Ledoux-97,Ledoux-01}).
\begin{lemma}[Concentration for averages]\cite[Proposition~4.3]{DG1}\label{lem:concLSI}
Assume that $\langle\cdot\rangle$ is stationary and satisfies the standard Logarithmic Sobolev Inequality~\eqref{e:def-sLSI}. Let $t\ge 1$ and let $F_t$ denote a bounded random variable
that is approximately $\sqrt{t}$-local in the sense of \eqref{cw20} and~\eqref{cw04}.
Then there exists a positive constant $C=C(d,\lambda,\kappa)$ such that 
we have for all $\delta>0$ and $T\geq t$
\begin{equation}\label{e.app-concLSI}
\expec{I\Big(\Big|\int \omega_T(x)F_t(a(\cdot+x))\,dx-\expec{F_t}\Big|>\delta\Big)}\, \le \, \exp\Big(-\frac {\delta^2} {C} \big(\frac{\sqrt{T}}{\sqrt{t}}\big)^{d}\Big).
\end{equation}
\qed
\end{lemma}
The proof of Theorem~\ref{g} is now as follows. By Proposition~\ref{p1}, $r_*$ is controlled by $r_{**}$, a minimal radius
based on the modified corrector.
Averages of the modified corrector in turn can be controlled by averages of an even more localized quantity, cf.~Proposition~\ref{p1bis}.
On the one hand, by Proposition~\ref{p1ter}, this quantity is local enough to apply Lemma~\ref{lem:concLSI} to control the size of its fluctuations. 
On the other hand, its expectation is controlled by Corollary~\ref{c.fluct.T-LSI} using Lemma~\ref{L:cw}.
By a union bound argument, this yields moment bounds for $r_{**}$, and therefore for $r_*$.

\subsubsection{Control of the minimal radius using multiscale functional inequalities}

In this paragraph we extend Theorem~\ref{g} to more general situations where the coefficient field
is more strongly correlated. We shall prove two results (Theorems~\ref{gbis} and~\ref{gbis2}) that address the general case of multiscale functional inequalities with functional derivative and oscillation. 
We start with the definitions of these functional inequalities.
\begin{definition}\cite[Definition~2.2]{DG1}\label{def:MFI}
Let $\pi:\R_+\to \R_+$ be an integrable function. For $L\ge 1$ we denote by $\expec{\cdot}_L$ the scaled ensemble.
\begin{itemize}
\item We say that $\expec{\cdot}$ satisfies a multiscale Logarithmic Sobolev Inequality (MLSI) for the weight $\pi$ with the functional derivative/oscillation if for all random variables $F$ we have
\begin{equation}\label{e:def-MLSI}
  \forall L\geq 1,\  \Ent_L(F)\,\leq\, \expec{  \int_0^\infty \pi(\ell)(\ell+1)^{-d}\int_{\R^d} |\partial^{\fun/\osc}_{x,\ell+1} F|^2 dxd\ell}_L.
\end{equation}
\item We say that $\expec{\cdot}$ satisfies a multiscale covariance inequality (MCI) for the weight $\pi$ with the oscillation  if for all random variables $F,G$ we have
\begin{align}
  \forall L\geq 1,\ &\covL{F}{G}\ :=\ \expec{(F-\expec{F}_L)(G-\expec{G}_L)}_L\nonumber
  \\ &\quad \leq \,     \int_0^\infty \pi(\ell)(\ell+1)^{-d}\int_{\R^d} \expec{  |\partial^{\osc}_{x,\ell+1} F|^2}_L^\frac12\expec{  |\partial^{\osc}_{x,\ell+1} G|^2}^\frac12_L  dxd\ell.\label{e:def-MCI}
\end{align}
\item We say that $\expec{\cdot}$ satisfies a multiscale spectral gap (MSG) for the weight $\pi$ with the oscillation  if for all random variables $F$ we have
\begin{align}
  \forall L\geq 1,\ &\varL{F}\ :=\ \expec{F^2}_L-\expec{F}_L^2\nonumber \\
  &\quad \leq\,   \expec{ \int_0^\infty \pi(\ell)(\ell+1)^{-d}\int_{\R^d}   |\partial^{\osc}_{x,\ell+1} F|^2dxd\ell}_L.\label{e:def-MSG}
\end{align}
\end{itemize}
\qed
\end{definition}
\begin{remark}
Note that both \eqref{e:def-MLSI} (for the oscillation) and \eqref{e:def-MCI} imply \eqref{e:def-MSG}.
If an ensemble satisfies MLSI with a compactly supported weight, then it satisfies the standard LSI (after suitable rescaling of space).
\qed
\end{remark}
The weight $\pi$ typically is related to the covariance function of $\expec{\cdot}$ (see in particular \eqref{e.weight-Gaussian} below for details).
In order to state our main result on the integrability of $r_*$, we need to introduce some further quantities.
Given a weight $\pi$, we consider the antiderative $\gamma$ of $-\pi$ defined by
\begin{subequations}
\begin{equation}\label{m2}
\gamma(\ell):=\int_{\ell}^\infty \pi(s)\,ds,
\end{equation}
and assume that it decays at least algebraically in the sense that there exist $0<\beta\ll 1$ and $C(\beta)$ (both depending only on $\pi$) such that
\begin{equation}\label{min.decay}
\gamma(\ell)\le C(\beta) (\ell+1)^{-\beta}\qquad\text{and}\qquad\liminf\limits_{\ell\to\infty}\frac{\gamma(\ell \theta)}{\gamma(\ell)}\geq \theta^{-\beta} \text{ for all }\theta\in(0,1).
\end{equation}
We finally introduce the function
\begin{equation}\label{m_*}
\pi_*(r):=\Big(\fint_{B_r}\gamma(|x|)\,dx\Big)^{-1},
\end{equation}
\end{subequations}
which,  as the following result shows, captures the stochastic integrability of $r_*$ for  MLSI with functional derivative.
\begin{theorem}\label{gbis}
Assume that $\langle\cdot\rangle$ is stationary and satisfies the  MLSI \eqref{e:def-MLSI}  for the functional derivative
with integrable weight $\pi$, where $\pi$ satisfies \eqref{m2} \& \eqref{min.decay}, and let
$\pi_*$ be as in~\eqref{m_*}.
For all $\alpha\in(0,1)$, consider $r_*$  defined in (\ref{O.1}) with constant $C(d,\lambda,\alpha)$.
Then there exists a positive  constant $C=C(d,\lambda,\pi,\alpha)$  such that 
\begin{equation}
\big\langle\exp\big({\textstyle\frac{1}{C}}  \pi_*(r_*) \big)\big\rangle\le 2.
\label{g.0bis}
\end{equation}
\qed
\end{theorem}
Before we turn to MFI with the oscillation, let us present the prototypical example of Gaussian coefficient fields.
Apply a local pointwise Lipschitz nonlinear transform to possibly several independent copies of a stationary 
Gaussian scalar field with covariance function $|c(x)|\le  \gamma(|x|)$ where $\gamma$ is non-increasing and decays at least algebraically at infinity. Then, relying on the Brascamp-Lieb inequality,  one can prove (\cite[Section~3.1]{DG2})  that the ensemble $\expec{\cdot}$ is stationary and satisfies MLSI \eqref{e:def-MLSI} with weight
\begin{equation}\label{e.weight-Gaussian}
\pi(\ell)\sim -\gamma'(\ell).
\end{equation}
We then abusively say that $\expec{\cdot}$ is Gaussian with correlation function $\gamma$.
The application of Theorem~\ref{gbis} to this example takes the following form.
\begin{corollary}\label{coro:Gauss}
Let $\expec{\cdot}$ be Gaussian with correlation function $\gamma(\ell)= (\ell+1)^{-\beta}$
for some $\beta>0$. Then  
\begin{equation}\label{e.pi*-Gaus}
\pi_*(r) \sim
\begin{cases}
     (r+1)^{\beta}&\beta<d,\\
     {(r+1)^d}\log^{-1}(r+2)&\beta=d,\\
     (r+1)^{d}&\beta>d,
\end{cases}
\end{equation}
and \eqref{g.0bis} thus takes the form
\begin{itemize}
\item subcritical case $\beta>d$: 
$\big\langle\exp\big({\textstyle\frac{1}{C}} r_*^d \big)\big\rangle\le 2$;
\item critical case $\beta=d$:  $\big\langle\exp\big({\textstyle\frac{1}{C}} \frac{r_*^d}{\log(r_*+1)} \big)\big\rangle\le 2$;
\item supercritical case  $0<\beta<d$: $\big\langle\exp\big({\textstyle\frac{1}{C}} r_*^\beta \big)\big\rangle\le 2$,
\end{itemize}
for some positive constant $C=C(d,\lambda, \beta,\alpha)$.
\qed
\end{corollary}
This corollary directly follows from Theorem~\ref{gbis} provided we prove \eqref{e.pi*-Gaus}, which
is itself a consequence of the definition~\eqref{m_*} of $\pi_*$ and of the elementary calculation 
$$
\fint_{B_r}\gamma(|x|)\,dx\, =\,r^{-d}  \int_0^r(\ell+1)^{-\beta}\ell^{d-1}d\ell
\sim
\begin{cases}
     (r+1)^{-\beta}&\beta<d,\\
     (r+1)^{-d}\log (r+2)&\beta=d,\\
     (r+1)^{-d}&\beta>d.
\end{cases}
$$
Note that a stronger version of Corollary~\ref{coro:Gauss} was obtained by Fischer and the third author in \cite{Fischer-Otto-17}
in the range $0<\beta\ll 1$. There, by a more direct use of the Brascamp-Lieb inequality, the same stochastic integrability (up to the optimal iterated logarithm)
as in Corollary~\ref{coro:Gauss} is established for a minimal radius $r_{**}$ that satisfies \eqref{gr01} and \eqref{gr02} for $\nu=1$.

\medskip

For MFI with the oscillation, the control of the moments of $r_*$ is weaker than for  MFI with the functional derivative.
In terms of stochastic integrability of $r_*$,  MLSI yields stronger results than  MCI which in turn yields stronger results than  MSG (for the same weight).
\begin{theorem}\label{gbis2}
Assume that $\langle\cdot\rangle$ is stationary.
Let $\pi$ be an integrable weight and  let $\gamma$ be as in~\eqref{m2}.
For all $\alpha\in(0,1)$, consider $r_*$  defined in (\ref{O.1}) with constant $C(d,\lambda,\alpha)$.
Then there exists a positive constant $C=C(d,\lambda,\pi,\alpha)$ such that
\begin{itemize}
\item if $\gamma$ has algebraic decay at infinity, and $\langle\cdot\rangle$ satisfies the MSG \eqref{e:def-MSG},
then
\begin{equation}
\big\langle {\textstyle\frac{1}{C}} \gamma(r_*)^{-1} \big\rangle\le 2;
\label{g.1bis}
\end{equation}
\item if $\gamma (\ell)\le C_\pi \exp(-\frac1{C_\pi} \ell^\beta)$ for some $\beta>0$, and $\langle\cdot\rangle$ satisfies the MSG \eqref{e:def-MSG}, then
\begin{equation}
\big\langle \exp\big({\textstyle\frac{1}{C}}r_*^{\beta \wedge \frac d2}\big)\big\rangle\le 2;
\label{g.2bis}
\end{equation}
\item if $\gamma (\ell)\le C_\pi \exp(-\frac1{C_\pi} \ell^\beta)$ for some $\beta>0$, and $\langle\cdot\rangle$ satisfies the MCI \eqref{e:def-MCI}, then
\begin{equation}
\big\langle \exp\big({\textstyle\frac{1}{C}}r_*^{(\beta \wedge \frac d2) \vee \frac{\beta d}{\beta+d}}\big)\big\rangle\le 2;
\label{g.3bis}
\end{equation}
\item if $\gamma (\ell)\le C_\pi\exp(-\frac1{C_\pi} \ell^\beta)$ for some $\beta>0$, and $\langle\cdot\rangle$ satisfies the MLSI \eqref{e:def-MLSI} for the oscillation, then
\begin{equation}
\big\langle \exp\big({\textstyle\frac{1}{C}}r_*^{\beta \wedge d}\big)\big\rangle\le 2.
\label{g.4bis}
\end{equation}
\end{itemize}
\qed
\end{theorem}
Before we turn to the ingredients to the proof of Theorems~\ref{gbis} and~\ref{gbis2}, let us present prototypical examples of coefficient fields satisfying MLSI, MCI or MSG with the oscillation (borrowed from \cite[Section~3]{DG2}).
The first example considers \textbf{random tessellations} (RT) and is given by 
\begin{equation}\label{example:Voronoi}
a(x)=\sum_i \lambda_i \Id I(x\in V_i), 
\end{equation}
where $I$ is the indicator function, $\lambda_i$ are i.i.d.~random variables that take values in $[\lambda,1]$, and $V_i$ are the Voronoi cells associated with a Poisson point process (RT-PPP) of fixed intensity or with the random
parking measure (RT-RPM), see \cite{Penrose-01}.
In the case of RT-PPP (resp. RT-RPM), by \cite[Proposition~3.2]{DG2} (resp.  \cite[Proposition~3.3]{DG2}), $\expec{\cdot}$ satisfies MLSI, cf.~\eqref{e:def-MLSI}, with $\pi(\ell)\sim  \exp(-\frac{\ell^d}{C})$ (resp.~with $\pi(\ell)\sim  \exp(-\frac{\ell}{C})$), which yields
$\gamma(\ell)\sim \exp(-\frac{\ell^d}{C})$ (resp.~$\gamma(\ell)\sim \exp(-\frac{\ell}{C})$).
The second example considers \textbf{random inclusions} (RI) and consists of a constant background coefficient field perturbed by random inclusions centered at the points of a Poisson point process of fixed intensity with i.i.d.~random radii. More precisely, if $\{z_k\}_k$ denotes the Poisson points, and $r_k$ denotes the radius of the inclusion $B_{r_k}(z_k)$ centered at $z_k$, we consider the inclusion set $\calI:=\cup_{k} B_{r_k}(z_k)$, and might for instance define the coefficient field as
\begin{equation}\label{example:Inclusions}
a(x)= \lambda \Id +(1-\lambda) \Id \,  I(x\in  \calI).
\end{equation}
Let $\Gamma(\ell):=\expec{I(\ell-\frac12 \le r_1 \le \ell+\frac12)}$ for $\ell\ge 0$ and $\Gamma(\ell):=0$ for $\ell<0$ be the distribution function of the radii on $\R$.
The, by \cite[Proposition~3.4]{DG2},  there is some $C=C(d)>0$ such that $\expec{\cdot}$ satisfies the MCI, cf.~\eqref{e:def-MCI}, with weight
\begin{equation*}
\pi(\ell) \sim \sup_{|\ell'|\le \frac1C} (\ell+\ell')^d \Gamma(\ell+\ell').
\end{equation*}
We refer to this example as (RI-PPP).
If instead of the Poisson point process we consider inclusions centered at the points of the random parking measure (RI-RPM), then, by \cite[Proposition~3.4]{DG2},  $\expec{\cdot}$ satisfies the MCI with weight
\begin{equation*}
\pi(\ell)\sim  \sup_{|\ell'|\le \frac1C} (\ell+\ell')^d \Gamma(\ell+\ell') + \exp(-\frac1C \ell),
\end{equation*}
where the additional term comes from the correlations of the points (of the RPM) themselves.
Finally, if $\Gamma$ has compact support (that is, if the radii are uniformly bounded), then $\expec{\cdot}$ satisfies MLSI with the same weights as above.
Corollary~\ref{rem:compare} below applies Theorem~\ref{gbis2} to these examples (the proof, which is obvious, is omitted).
Next to these results, we display what can be proved using linear concentration argument based on the associated $\alpha$-mixing. In most of these examples the stochastic integrability of $r_*$ implied by the  MFI are strictly stronger than what would follow from the associated $\alpha$-mixing conditions (cf.~\cite[Remark~4.6]{DG1} for the spatial averages of the coefficient field itself).
\begin{corollary}\label{rem:compare}
{\color{white} rien}
\begin{enumerate}[(i)]
\item For \emph{RT-PPP} and \emph{RI-PPP} with uniformly bounded radii, \eqref{g.4bis} in Theorem~\ref{gbis2} takes the form
  \begin{equation*}
    \expec{\exp(\frac1C r_*^d)}\le 2,\qquad \text{while $\alpha$-mixing yields }\expec{\exp(\frac1C r_*^\frac{d}2)}\le 2.
  \end{equation*}
  For \emph{RT-RPM}, \eqref{g.2bis} in Theorem~\ref{gbis2} takes the form
  \begin{equation*}
    \expec{\exp(\frac1C r_*)}\le 2,\qquad\text{while $\alpha$-mixing yields }\expec{\exp(\frac1C r_*^\frac{d}{d+1})}\le 2.
  \end{equation*}
\item For the example of random inclusions \eqref{example:Inclusions} with random radii
and distribution function $\overline\Gamma:\R_+\to [0,1]$, Theorem~\ref{gbis2} takes the form
\begin{enumerate}
\item If $\overline \Gamma(\ell) \stackrel{\ell\gg 1}{\sim} \ell^{-d-1-\beta}$ for some $\beta>0$, then we have $\Gamma(\ell)\sim (\ell+1)^{-d-1-\beta}$, $\pi(\ell) \sim (\ell+1)^{-1-\beta}$, $\gamma(\ell)\sim (\ell+1)^{-\beta}$, and 
\eqref{g.1bis} yields
\begin{equation*}
\expec{\frac1C r_*^\beta}\le 2 \text{ (same with $\alpha$-mixing)}.
\end{equation*}
\item If $\overline\Gamma(\ell) \stackrel{\ell\gg 1}{\sim} e^{-\frac1{C} \ell^\beta}$ for some $\beta,C>0$, 
then $\Gamma(\ell)\sim e^{-\frac1C \ell^\beta}$, and \eqref{g.3bis} yields
\begin{eqnarray*}
\emph{RI-PPP}&:&\expec{\exp(\frac1C r_*^{\big(\beta \wedge \frac{d}{2}\big)\vee \frac{\beta d}{\beta+d}})}\le 2
\text{, while $\alpha$-mixing yields }\expec{\exp(\frac1C r_*^{\frac{\beta d}{\beta+d}})}\le 2;
\\
\emph{RI-RPM}&:&\expec{\exp(\frac1C r_*^{ \beta \wedge 1})}\le 2\text{, while $\alpha$-mixing yields }\expec{\exp(\frac1C r_*^{\frac{(\beta\wedge 1) d}{\beta\wedge 1+d}})}\le 2.
\end{eqnarray*}
\end{enumerate}
\end{enumerate}
\qed
\end{corollary}

\medskip

We finally turn to the proof of Theorems~\ref{gbis} and~\ref{gbis2}. The general structure is similar to that of Theorem~\ref{g}.
Two ingredients need to be refined: The concentration result of Lemma~\ref{lem:concLSI} and the control of 
the second moment of the extended corrector in Corollary~\ref{c.fluct.T-LSI}. 
We start with the former.
\begin{lemma}[Concentration for averages]\cite[Propositions~4.3 \&~4.5, and Remark~4.4]{DG1}\label{lem:conc}
Assume that  $\expec{\cdot}$ is stationary. Let $\pi$ denote an integrable weight satisfying \eqref{min.decay}, where  $\gamma$ and $\pi_*$ are defined in \eqref{m2} and \eqref{m_*}.
Let $t\ge 1$ and let $F_t$ denote a bounded random variable
that is approximately $\sqrt{t}$-local in the sense of \eqref{cw04}.
For all $T\ge t$,  set $(F_t)_T(a):=\int \omega_T(x)F_t(a(\cdot+x))\,dx$. Then there exists a positive constant $C=C(d,\lambda,\pi)$ 
such that for all $\delta>0$ and all $T\ge t$:
\begin{enumerate}[(i)]
\item if $\expec{\cdot}$ satisfies the MSG \eqref{e:def-MSG} with weight $\pi(\ell)\le C_\pi (\ell+1)^{-\beta-1}$ for some $\beta>0$, then 
\begin{equation}\label{e.app-conc3-0}
\expec{I\Big((F_t)_T-\expec{F_t}\ge \delta\Big)} \, \le \, C e^{-\delta} (1+\delta^{-2\frac{\beta}d}|\log \delta|) \big(\frac{\sqrt{t}}{\sqrt{T}}\big)^{-\beta};
\end{equation}
\item if $\expec{\cdot}$ satisfies the MSG \eqref{e:def-MSG} with weight $\pi(\ell)\le C_\pi \exp(-\frac1{C_\pi} \ell^\beta)$ for some $\beta>0$, then 
\begin{equation}\label{e.app-conc3}
\expec{I\Big((F_t)_T-\expec{F_t}\ge \delta\Big)}  \, \le \, \exp\Big(-\frac {\delta^2} {C} \big(\frac{\sqrt{t}}{\sqrt{T}}\big)^{\beta\wedge\frac d2}\Big);
\end{equation}
\item if $\expec{\cdot}$ satisfies the MCI \eqref{e:def-MCI} with weight $\pi(\ell)\le C_\pi \exp(-\frac1{C_\pi} \ell^\beta)$ for some $\beta>0$, then 
\begin{align*}
\expec{I\Big((F_t)_T-\expec{F_t}\ge \delta\Big)} \le C\exp\Big(-\frac{\delta^2}C\big(\frac{\sqrt{t}}{\sqrt{T}}\big)^{\beta \wedge \frac d2}-\frac{\delta^2}C\big(\big(\frac{\sqrt{t}}{\sqrt{T}}\big)\frac{1}{|\log \delta|+1}\big)^{\frac{d\beta}{d+\beta}}\Big);
\end{align*}
\item if $\expec{\cdot}$ satisfies the MLSI \eqref{e:def-MLSI} for the oscillation with weight $\pi(\ell)\le C_\pi \exp(-\frac1{C_\pi} \ell^\beta)$ for some $\beta>0$, then  
\begin{equation}\label{e.app-conc4}
\expec{I\Big((F_t)_T-\expec{F_t}\ge \delta\Big)} \, \le \, \exp\Big(-\frac {\delta^2} {C} \big(\frac{\sqrt{t}}{\sqrt{T}}\big)^{\beta\wedge d}\Big);
\end{equation}
\item if $\expec{\cdot}$ satisfies the MLSI \eqref{e:def-MLSI}  with weight $\pi$ for the functional derivative, then 
\begin{equation}\label{e.app-conc2}
\expec{I\Big((F_t)_T-\expec{F_t}\ge \delta\Big)} \, \le \, \exp\Big(-\frac{\delta^2} C \pi_*\big(\frac{\sqrt{t}}{\sqrt{T}}\big)\Big).
\end{equation}
\end{enumerate}
\qed
\end{lemma}
\begin{remark}
For $t=1$, the proof of Lemma~\ref{lem:conc} adapts the Herbst argument
 taking into account the specific properties of averages. For $t>1$ the  general statement reduces to the 
statement for $t=1$ by rescaling and using \eqref{e:def-MLSI}, \eqref{e:def-MCI}, or \eqref{e:def-MSG} with $L=\sqrt{t}$.  
\qed
\end{remark}
We conclude with the extension of Corollary~\ref{c.fluct.T-LSI}.
\begin{corollary}\label{c.fluct.T}
Assume that $\langle\cdot\rangle$ is stationary and satisfies 
the  MLSI  \eqref{e:def-MLSI} with the functional
derivative or the  MSG \eqref{e:def-MSG} with the oscillation for an integrable weight $\pi$ that satisfies \eqref{m2} \& \eqref{min.decay}.
Then there exist an exponent $\e=\e(d,\lambda,\pi)>0$ (depending on the hole-filling exponent) and a constant $C=C(d,\lambda,\pi)$ such that for all ${T}>0$ we have
\begin{equation} \label{o8}
\langle\frac{1}{T}\phi_T^2+\frac{1}{T}|g_T|^2+|\nabla g_T|^2\rangle \leq C  T^{-\e}.
\end{equation}
\qed
\end{corollary}
\begin{remark}
Corollaries~\ref{c.fluct.T-LSI} and~\ref{c.fluct.T} are the only places in our strategy of proof where the functional inequality is applied to a truly nonlinear random variable.
We refer the reader to \cite{GO-16} for a similar result that does not rely on functional inequalities.\qed
\end{remark}
%


\section{Proof of the regularity theory}

\subsection{Proof of Lemma~\ref{si}: Construction of correctors}

This proof is standard and essentially based on \cite[Section 7.2]{JKO-94}.
We could also argue by massive approximation, cf.~\eqref{o14}--\eqref{o17}.
The argument below relies on the Lax-Milgram theorem in the space of potential fields in probability.
Following  \cite{Papanicolaou-Varadhan-79} we define the ``horizontal derivatives'' $D_j$ as the generators of the $d$ (shift) $L^2(\Omega$)-semigroups. More explicitely, we set
\begin{align*}
&H^1(\Omega):=\{\zeta\in L^2(\Omega)\,:\,\lim_{h\rightarrow 0}\frac{1}{h}(\zeta(a(\cdot+he_j))-\zeta(a))\text{ exists as a limit in $L^2(\Omega)$ for all $j$}\,\},\\
&D_j:H^1(\Omega)\to L^2(\Omega),\qquad D_j\zeta:=\lim_{h\rightarrow 0}\frac{1}{h}(\zeta(a(\cdot+he_j))-\zeta(a)).
\end{align*}
In the following argument we suppress the index $i$ (which is fixed throughout the proof) in our notation for the tensor fields $\phi_i$, $\sigma_{ijk}$, and $q_{ij}$.

\medskip

\step{1} Construction of a potential field $g$ (playing the role of $\nabla\phi$) and of a solenoidal field $q$.

Consider the space of curl-free vector fields with vanishing expectation
\begin{equation}\nonumber
X:=\{g\in L^2(\Omega,\mathbb{R}^d)|D_jg_k=D_kg_j\;\mbox{distributionally},\;\langle g_j\rangle=0\}.
\end{equation}
This is a closed subspace of $L^2(\Omega,\R^{d})$. 
Because of stationarity,
$-D_j$ is the (formal) adjoint of $D_j$. By stationarity, ergodicity, and the density of
$\{D \phi|\phi\in H^1(\Omega)\}\subset X$ in $X$, (\ref{f.40}) translates into
\begin{equation}\label{g.30}
\forall g\;\in X\quad \langle g\cdot a(0) g\rangle\ge\lambda\langle|g|^2\rangle.
\end{equation}
By the Lax-Milgram theorem, there thus exists a unique
\begin{equation}\label{f.41}
g\in X\text{ such that }\forall \tilde g\;\in X\quad\langle \tilde g\cdot a(0)(g+e)\rangle=0.
\end{equation}
With help of \eqref{f.56} and (\ref{g.30}), we see that it satisfies the bound
\begin{equation}\label{si.4}
\langle|g|^2\rangle\le\frac{1}{\lambda^2}.
\end{equation}
Since $\{D \phi|\phi\in H^1(\Omega)\}\subset X$, (\ref{f.41}) implies in particular
\begin{equation}\label{f.16}
D\cdot a(0)(g+e)=0
\end{equation}
in a weak sense. 
We define the homogenized coefficients $a_{\ho}$ in direction $e$ as
\begin{equation}\label{f.42}
a_{\ho}e=\langle a(0)(g+e)\rangle.
\end{equation}
In particular, the random vector $q\in L^2(\Omega,\mathbb{R}^d)$
\begin{equation}\label{f.18}
q:=a(0)(g+e)-a_{\ho}e=a(0)(g+e)-\langle a(0)(g+e)\rangle,
\end{equation}
which we may think of as a flux correction, satisfies
\begin{equation}\label{f.13}
\langle|q|^2\rangle\le\frac{1}{\lambda^2}+1,\quad\langle q\rangle=0,
\quad D\cdot q=0,
\end{equation}
(the bound is seen as follows $\langle|q|^2\rangle\stackrel{(\ref{f.18})}{\le}\langle|a(0)(g+e)|^2\rangle
\stackrel{(\ref{f.56})}{\le}\langle|g+e|^2\rangle
=\langle|g|^2\rangle+1\stackrel{(\ref{si.4})}{\le}\frac{1}{\lambda^2}+1$, the $+1$ is the price to
pay for knowing (\ref{g.30}) only for $g$'s with $\langle g\rangle=0$)
which mimics the properties of the field correction, namely
\begin{equation}\nonumber
\langle|g|^2\rangle\le \frac{1}{\lambda^2},\quad\langle g\rangle=0,
\quad D_jg_{k}=D_k g_{j}.
\end{equation}

\medskip

\step{2} Construction of a curl-free matrix field $b$.

For the construction of $\sigma_{jk}$ we first introduce an auxiliary third-order tensor field $b=b_{jkl}$.
In this step, and throughout the article unless explicitly stated, we use the Einstein summation convention on repeated indices.
Let $\R_{sym}^{d\times d}$ denote the space of symmetric matrices and consider the space of curl-free {\it symmetric matrix} fields of vanishing expectation
\begin{equation}\label{f.12}
Y:=\{\tilde b\in L^2(\Omega,\mathbb{R}_{sym}^{d\times d})|
D_k\tilde b_{lm}=D_m\tilde b_{lk}\;\mbox{distributionally},\;\langle \tilde b_{kl}\rangle=0\},
\end{equation}
which is a closed subspace of $L^2(\Omega,\R_{\sym}^{d\times d})$. We denote by $b_j\in Y$ the $L^2(\Omega,\R^{d\times d}_\sym)$-orthogonal projection of the tensor field $q_j \Id$ onto $Y$, where 
$\Id$ denotes the identity matrix in $\R^{d\times d}$. As a projection,  $b_j$ satisfies the estimate
\begin{equation}\label{si.3}
  \langle |b_j|^2\rangle\le\langle|q_j \Id|^2\rangle=d\expec{q_j^2}.
\end{equation}
We claim that the third order tensor $b=b_{jkl}$ satisfies 
\begin{eqnarray}\label{f.10}
b_{jkk}&=&q_j,\\
  \label{f.17}
b_{kkj}&=&0.
\end{eqnarray}
We first prove \eqref{f.10}.
Define $H^2(\Omega)$ as the set of $H^1(\Omega)$ functions $\zeta$ such that $D \zeta \in H^1(\Omega,\R^d)$.
 Since $\{D^2\zeta|\zeta\in H^2(\Omega)\}\subset Y$, we have by orthogonality and the curl-free condition in the definition (\ref{f.12}) of $Y$
in form of $D_m b_{jkl}=D_k b_{jml}$ (that holds by symmetry and that we use for $m=l$),
\begin{multline*}
0=\langle D^2\zeta \cdot (b_j-q_j\Id)\rangle = \langle D_kD_l \zeta ( b_{jkl}-q_j \delta_{kl}) \rangle
=- \langle  D_k \zeta D_l  b_{jkl} \rangle - \langle  D_{kk}^2\zeta q_j \rangle
\\
=\langle D_kD_k \zeta  b_{jll} \rangle- \langle D_{kk} ^2\zeta q_j \rangle
= \langle D_{kk}^2\zeta (b_{jll}-q_j)\rangle.
\end{multline*}
This implies \eqref{f.10}, since by ergodicity the range of $\{D_{kk}^2\zeta\,|\,\zeta\in H^2(\Omega)\}$ is dense in $\{\zeta\in L^2(\Omega)\,|\,\expec{\zeta}=0\,\}$ and both $b_j$ and $q_j$ have vanishing expectation.

\medskip

The remaining identity \eqref{f.17} follows from $D\cdot q=0$, cf.\ (\ref{f.13}). Indeed, by the curl-free and symmetry conditions 
\begin{equation}\nonumber
D_lD_lb_{kkj}=D_lD_jb_{kkl}=D_lD_jb_{klk}=D_jD_kb_{kll} \quad \text{(without summation on repeated indices)}  
\end{equation}
in a distributional sense. Hence, for all $j$ we have
\begin{equation*}
\langle D_lD_l\zeta  b_{kkj}\rangle=\langle D_jD_k\zeta b_{kll} \rangle \stackrel{\eqref{f.10}}{=}\langle DD_j\zeta\cdot q\rangle\stackrel{\eqref{f.13}}{=}0,
\end{equation*}
and \eqref{f.17} follows from ergodicity.

\medskip

\step{3} Construction of potential scalar and vector fields for $g$ and $b$.

By construction of $g$ and $b$, these fields are horizontally curl-free in a distributional sense:
\begin{equation}\nonumber
D_jg_{k}=D_kg_{j}\quad\mbox{and}\quad
D_lb_{jkm}=D_mb_{jkl}.
\end{equation}
We extend the random variables $g$, $q$, and $b$ to stationary fields according to $g(a;x)=g(a(\cdot+x))$,
however keeping the same symbol so that in particular
(\ref{f.18}) is consistent with (\ref{si.6}). By definition of the horizontal derivative, spatial and horizontal derivatives
are then related by $(\partial_jg)(a;x)=(D_jg)(a(\cdot+x))$, so that we obtain in particular
\begin{equation}\nonumber
\partial_jg_{k}=\partial_kg_{j}\quad\mbox{and}\quad
\partial_lb_{jkm}=\partial_mb_{jkl}.
\end{equation}
Therefore, there exist fields $\phi=\phi(a;x)$ and $\sigma_{jk}=\sigma_{jk}(a;x)$
with the property that
\begin{equation}\label{f.15}
  g_{j}=\partial_j\phi,\quad b_{jkl}-b_{kjl}=\partial_l\sigma_{jk}.
\end{equation}
The fields $\phi$ and $\sigma_{jk}$ are uniquely
determined by (\ref{f.15}) up to a random additive constant in $x$, which we may
fix by requiring that their average on the unit ball centered at the origin vanishes, e.g. $\fint_B \phi=\fint_B \sigma_{jk}=0$. This makes the fields (generically) non-stationary and ensures that $\{\sigma_{jk}\}_{jk}$ inherits the build-in skew-symmetry of $\{b_{jkl}-b_{kjl}\}_{jk}$, and thus (\ref{f.19}) follows.
Clearly, the build-in vanishing expectation properties of $g$ and $b$
translate into those in (\ref{si.2}). Moreover, the bounds stated in \eqref{si.2} follow from the moment bounds
on $g$ and $q$, cf.\ (\ref{si.4}),  (\ref{f.13}), and \eqref{si.3}.

\medskip

We note that by definition (\ref{f.15}) and (\ref{f.16}), the latter rewritten in terms
of spatial instead of horizontal derivatives as $\nabla\cdot a(g+e)=0$, 
we obtain (\ref{f.2}). For (\ref{f.5}), we note that
\begin{equation}\nonumber
\partial_l\sigma_{jl}\stackrel{(\ref{f.15})}{=} b_{jll}-b_{ljl}=b_{jll}-b_{llj}\stackrel{(\ref{f.10}),(\ref{f.17})}{=}q_{j}.
\end{equation}
Finally (\ref{si.5}) can be seen as follows
\begin{eqnarray*}
\partial_l\partial_l\sigma_{jk}
&\stackrel{(\ref{f.15})}{=}&\partial_l b_{jkl}-\partial_l b_{kjl}\\
&=&\partial_l b_{jlk}-\partial_l b_{klj}\quad\mbox{by symmetry of}\;b\\
&=&\partial_k b_{jll}-\partial_j b_{kll}\quad\mbox{by curl-freeness of}\;b\\
&\stackrel{(\ref{f.10})}{=}&\partial_k q_{j}-\partial_jq_{k}.
\end{eqnarray*}
%


\subsection{Proof of Proposition~\ref{F}: Large-scale regularity by perturbation}

Following~\cite{Avellaneda-Lin-87}, we recover the improvement (\ref{F.1})
for $a$-harmonic functions as a perturbation of a result for $a_{\ho}$-harmonic
functions. 
By a scaling argument, we may assume $R=1$. To ease notation, we also assume $\fint_{B_1}(\phi,\sigma)=0$.

\medskip

\step1 Two PDE ingredients. 

\nobreak
On the one hand, we claim that there exists an exponent $\e=\e(d,\lambda)>0$ such that if $w$ and $g$ satisfy
\begin{align}\label{wi02}
-\nabla\cdot a\nabla w=\nabla\cdot g\;\;\mbox{in}\;B_1,\quad w=0\;\;\mbox{on}\;\partial B_1,
\end{align}
then we have the weighted energy estimate
\begin{align}\label{wi01}
\int_{B_1}(1-|x|)^\e|\nabla w|^2\lesssim\int_{B_1}(1-|x|)^\e |g|^2.
\end{align}
On the other hand, for any function $u_{\ho}$ with
\begin{align}\label{wi03}
-\nabla\cdot a_{\ho}\nabla u_{\ho}=0\quad\mbox{in}\;B_1,
\end{align}
we claim the inner regularity estimate
\begin{align}\label{wi05}
\sup_{B_{1-\rho}}\big(\rho|\nabla^2 u_{\ho}|+|\nabla u_{\ho}|\big)
\lesssim\big(\frac{1}{\rho^d}\int_{B_1}|\nabla u_{\ho}|^2\big)^\frac{1}{2}
\end{align}
for any boundary layer width $\rho\le1$. 

\medskip

We first address (\ref{wi01}). By Caccioppoli's estimate we have for any cut-off function $\eta$ in $B_1$
\begin{align}\label{wi11}
\int_{B_1}\eta^2|\nabla w|^2\lesssim\int_{B_1}\eta^2|g|^2+\int_{B_1}|\nabla\eta|^2w^2,
\end{align}
see \eqref{e.add-Caccio} in the proof of Lemma~\ref{L:reg} in Appendix~\ref{append:Cacc}.
Choosing $\eta^2=(1-|x|)^\e$ for some $\e\in(0,1)$ to be fixed later this turns into 
\begin{align*}
\int_{B_1}(1-|x|)^\e|\nabla w|^2\lesssim\int_{B_1}(1-|x|)^\e|g|^2+\e^2\int_{B_1}(1-|x|)^{\e-2}w^2,
\end{align*}
where $\lesssim$ stands for a constant that depends on $d$ and $\lambda$ but not on $\e$. In order to
absorb the second RHS term for $\e\ll 1$, we appeal to Hardy's inequality
\begin{equation}\label{hardy}
\int_{B_1}(1-|x|)^{\e-2}w^2\,\lesssim\, \int_{B_1}(1-|x|)^{\e}|\nabla w|^2.
\end{equation}
For the convenience of the reader, we display the standard argument. By polar coordinates, it is enough to establish
\begin{equation}\label{hardy-1}
  \int_0^1 (1-r)^{\e-2} w^2 r^{d-1}dr \,\le\, \frac{4}{(1-\e)^2} \int_0^1 (1-r)^\e (\partial_r w)^2 r^{d-1}dr,
\end{equation}
provided $w(1)=0$ (note that the finiteness of the RHS of \eqref{hardy-1} is enough to guarantee this trace). To this purpose we note $(1-r)^{\e-2}r^{d-1}\leq \frac{d}{dr}\big(\frac{1}{1-\e}(1-r)^{\e-1}r^{d-1}\big)$ so that
\begin{equation*}
  (1-r)^{\e-2}w^2 r^{d-1}\leq \frac{d}{dr}\big(\frac{1}{1-\e}(1-r)^{\e-1}w^2r^{d-1}\big)-\frac{2}{1-\e}(1-r)^{\e-1}w\partial_r w\, r^{d-1}.
\end{equation*}
When integrating this inequality over $r\in(0,1)$, we note that the boundary contribution from $r=0$ is non-positive, while the one from $r=1$ vanishes, since the finiteness of the RHS of \eqref{hardy-1} implies $w^2=o(1-r)$. By Cauchy-Schwarz' inequality we thus obtain
 $$
 \int_0^1 (1-r)^{\e-2} w^2 r^{d-1}dr \,\le\,\frac 2{1-\e}\Big(\int_0^1 (1-r)^{\e-2}w^2\,r^{d-1}\,dr\, \int_0^1(1-r)^{\e}(\partial_r w)^2\,r^{d-1}\,dr\Big)^\frac12,
 $$
 which implies \eqref{hardy-1} since by $w^2=o(1-r)$ the LHS is finite.

\medskip

We now turn to (\ref{wi05}) which is a consequence of the estimate 
\begin{align*}
  \rho|\nabla^2 u_{\ho}(z)|+|\nabla u_{\ho}(z)|
  \lesssim\big(\fint_{B_\rho(z)}|\nabla u_{\ho}|^2\big)^\frac{1}{2}\qquad\text{for all $z\in B_{1-\rho}$}.
\end{align*}
By translation and rescaling the latter follows from the inner regularity estimate
\begin{equation*}
\sup_{B_1}|\nabla^2 v|^2+\sup_{B_1}|\nabla v|^2\lesssim \int_{B_2}|\nabla v|^2,
\end{equation*}
for any $a_\ho$-harmonic function $v$ on $B_2$.
For the sake of brevity we focus on the first estimate, that is
\begin{equation}\nonumber
\sup_{B_1}|\nabla^2 v|^2\lesssim \int_{B_2}|\nabla v|^2.
\end{equation}
Since the coefficients $a_{\ho}$ are constant, to the effect that also the components of
$\nabla v$ are harmonic, this amounts to showing
\begin{equation}\nonumber
\sup_{B_1}|\nabla v|^2\lesssim \int_{B_2}|v|^2.
\end{equation}
By Sobolev's embedding, it is enough to show for some integer $k$ with $k>\frac{d}{2}+1$ that
\begin{equation}\nonumber
\int_{B_1}(|\nabla^k v|^2+|\nabla^{k-1}v|^2+\cdots+|\nabla v|^2)\lesssim \int_{B_2}|v|^2.
\end{equation}
Again, since the components of the tensor $\nabla^{m}v$, $m=0,\cdots,k-1$, are 
$a_{\ho}$-harmonic, this follows
from a $k$-fold application of the Caccioppoli estimate (\ref{f.51}) in Appendix~\ref{append:Cacc}, where the radius decreases
at every step by the amount of $\frac{1}{k}$.

\medskip

\step2 The harmonic approximation. 

\nobreak
We consider the Lax-Milgram solution $u_{\ho}$ to
\begin{align}\label{wi10}
-\nabla\cdot a_{\ho}\nabla u_{\ho}=0\;\;\mbox{in}\;B_1,\quad
u_{\ho}=u\;\;\mbox{on}\;\partial B_1.
\end{align}
We claim that
\begin{align}\label{wi06}
\int_{B_1}|\nabla u_{\ho}|^2\lesssim\int_{B_1}|\nabla u|^2.
\end{align}
Indeed, we rewrite (\ref{wi10}) as $-\nabla\cdot a_\ho\nabla(u_\ho-u)$ $=\nabla\cdot a_\ho\nabla u$
in $B_1$ with $u_\ho-u=0$ on $\partial B_1$, so that by testing with $u_\ho-u$ we obtain from the
ellipticity of $a_\ho$ that $\int_{B_1}|\nabla(u_\ho-u)|^2$ $\lesssim\int_{B_1}|\nabla u|^2$.
Now (\ref{wi06}) follows by the triangle inequality.

\medskip

\step3 Representation formula in conservative form.

For a given boundary layer thickness $\rho\in(0,\frac{1}{2}]$ we select a cut-off 
function $\eta$ with
\begin{align}\label{wi12}
\eta=1\;\mbox{in}\;B_{1-2\rho},\quad\eta=0\;\mbox{outside of}\;B_{1-\rho},\quad|\nabla\eta|\lesssim\frac{1}{\rho},
\end{align}
and consider the error in the two-scale expansion
\begin{align}\label{wi14}
w:=u-(1+\eta\phi_i\partial_i)u_{\ho},
\end{align}
which thanks to $\eta$ vanishes on $\partial B_1$. We claim that we have (\ref{wi02}) with RHS
\begin{align}\label{wi13}
g:=(1-\eta)(a-a_\ho)\nabla u_\ho+(\phi_ia-\sigma_i)\nabla(\eta\partial_iu_\ho).
\end{align}
Indeed, applying the gradient to (\ref{wi14}), we obtain by Leibniz' rule
\begin{equation}\label{F.5}
\nabla w=\nabla u-(\nabla u_{\ho}+\eta\partial_iu_{\ho}\nabla\phi_i+\phi_i\nabla(\eta\partial_iu_{\ho})).
\end{equation}
Applying $-\nabla\cdot a$, this yields because of (\ref{p.12})
\begin{eqnarray*}\nonumber
\lefteqn{-\nabla\cdot a\nabla w}\\
&=&\nabla\cdot\big(a\nabla u_{\ho}+\eta\partial_iu_{\ho}a\nabla\phi_i+\phi_i a\nabla(\eta\partial_iu_{\ho})\big)\\
&=&\nabla\cdot\big((1-\eta)a\nabla u_{\ho}+\eta\partial_iu_{\ho}a(\nabla\phi_i+e_i)+\phi_i a\nabla(\eta\partial_iu_{\ho})\big)\\
&\stackrel{(\ref{f.2})}{=}&\nabla\cdot\big((1-\eta)a\nabla u_{\ho}+\phi_i a\nabla(\eta\partial_iu_{\ho})\big)
+\nabla(\eta\partial_iu_{\ho})\cdot a(\nabla\phi_i+e_i).
\end{eqnarray*}
Writing $\nabla(\eta\partial_iu_{\ho})\cdot a_{\ho}e_i=\nabla\cdot(\eta\partial_iu_{\ho}a_{\ho}e_i)
=\nabla\cdot(\eta a_{\ho}\nabla u_{\ho})$, and appealing to (\ref{wi10})
in form of $\nabla\cdot(\eta a_{\ho}\nabla u_{\ho})=-\nabla\cdot((1-\eta) a_{\ho}\nabla u_{\ho})$,
we see that the above turns into
\begin{eqnarray*}\nonumber
-\nabla\cdot a\nabla w
&=&\nabla\cdot\big((1-\eta)(a-a_{\ho})\nabla u_{\ho})+\phi_i a\nabla(\eta\partial_iu_{\ho})\big)\\
&&+\nabla(\eta\partial_iu_{\ho})\cdot (a(\nabla\phi_i+e_i)-a_{\ho}e_i).
\end{eqnarray*}
Using $\nabla\cdot \sigma_i=q_i=a(\nabla\phi_i+e_i)-a_{\ho}e_i$, cf.\ (\ref{f.5}), and the skew-symmetry
of $\sigma_i$, cf.\ (\ref{f.19}), in form of
\begin{equation*}
\nabla\zeta\cdot(\nabla\cdot\sigma_i)\,=\,\partial_j\zeta\partial_k \sigma_{ijk}
\,\stackrel{(\ref{f.19})}{=}\, \partial_k(\partial_j\zeta \, \sigma_{ijk})\,\stackrel{(\ref{f.19})}{=}
\,-\nabla\cdot(\sigma_i\nabla\zeta)
,
\end{equation*}
we obtain (\ref{wi02}) with $g$ defined as in (\ref{wi13}).

\medskip

\step4 Estimate of $g$. 

We claim that
\begin{align}\label{wi09}
\int_{B_1}(1-|x|)^\e|g|^2
\lesssim\big(\rho^\e+\rho^{-d-2}\int_{B_1}|(\phi,\sigma)|^2\big)\int_{B_{1}}|\nabla u|^2.
\end{align}
Indeed, by definition (\ref{wi13}) of $g$ and definition (\ref{wi12}) of $\eta$ we have
\begin{multline*}
{\int_{B_1}(1-|x|)^\e|g|^2}
\,\lesssim \int_{B_{1}\setminus B_{1-2\rho}}(1-|x|)^\e|\nabla u_\ho|^2
\\
+\sup_{B_{1-\rho}}(|\nabla^2 u_\ho|^2+\rho^{-2}|\nabla u_\ho|^2)\int_{B_1}|(\phi,\sigma)|^2,
\end{multline*}
so that by (\ref{wi05})
\begin{align*}
\int_{B_1}(1-|x|)^\e|g|^2\lesssim\big(\rho^\e+\rho^{-d-2}\int_{B_1}|(\phi,\sigma)|^2\big)
\int_{B_{1}}|\nabla u_\ho|^2.
\end{align*}
Inserting (\ref{wi06}) yields (\ref{wi09}).

\medskip

\step5 Estimate by $w$. 

\nobreak
We claim that for any $r\le\frac{1}{4}$ we have
\begin{align}\label{wi07}
\lefteqn{\fint_{B_r}|\nabla u-\partial_iu_{\ho}(0)(e_i+\nabla\phi_i)|^2}\nonumber\\
&\lesssim\big(r^2+r^{-d}\fint_{B_1}|\phi|^2\big)\int_{B_1}|\nabla u|^2
+r^{-d-2}\fint_{B_1}(1-|x|)^\e|\nabla w|^2.
\end{align}
Indeed, since $u-(u_\ho(0)+\partial_iu_{\ho}(0)(x_i+\phi_i))$ is an $a$-harmonic function, we have by Caccioppoli's estimate
\begin{multline}\label{wi08}
{\fint_{B_r}|\nabla u-\partial_iu_{\ho}(0)(e_i+\nabla\phi_i)|^2}\\
\lesssim \, r^{-2}\fint_{B_{2r}}\Big(u-\big(u_\ho(0)+\partial_iu_{\ho}(0)(x_i+\phi_i)\big)\Big)^2.
\end{multline}
Using that $w=u-(1+\phi_i\partial_i)u_\ho$ on $B_{2r}$, cf.~(\ref{wi12}) \& (\ref{wi14}),
we obtain by the triangle inequality in $L^2$
\begin{multline*}
\fint_{B_{2r}}\Big(u-\big(u_{\ho}(0)+\partial_iu_{\ho}(0)(x_i+\phi_i)\big)\Big)^2
\\
\lesssim\fint_{B_{2r}}w^2
+\sup_{B_{2r}}\big(u_\ho-(u_\ho(0)+\partial_iu_\ho(0) x_i)\big)^2
+\sup_{B_{2r}}(\partial_iu_\ho-\partial_iu_\ho(0))^2\fint_{B_{2r}}\phi_i^2.
\end{multline*}
Combining  Taylor's formula, \eqref{wi05} with $\rho=\frac{1}{2}$, and
and \eqref{wi06}, yields
\begin{multline*}
\sup_{B_{2r}}\big(u_\ho-(u_\ho(0)+\partial_iu_\ho(0) x_i)\big)^2
+r^2 \sup_{B_{2r}}(\partial_iu_\ho-\partial_iu_\ho(0))^2\\
\lesssim r^4\sup_{B_\frac{1}{2}}|\nabla^2 u_\ho|^2\lesssim r^4\fint_{B_1}|\nabla u_\ho|^2
\lesssim r^4\fint_{B_1}|\nabla u|^2,
\end{multline*}
and thus
\begin{multline*}
{\fint_{B_{2r}}\Big(u-\big(u_{\ho}(0)+\partial_iu_{\ho}(0)(x_i+\phi_i)\big)\Big)^2}\nonumber\\
\,\lesssim \fint_{B_{2r}}w^2+\big(r^4+r^{2-d}\int_{B_{1}}|\phi|^2\big)\fint_{B_1}|\nabla u|^2.
\end{multline*}
We combine this with \eqref{hardy} 
\begin{equation*}
  \fint_{B_{2r}}w^2\lesssim r^{-d}\fint_{B_1}(1-|x|)^{\e-2} w^2\lesssim r^{-d}\fint_{B_1}(1-|x|)^\e |\nabla w|^2.
\end{equation*}
The combination with \eqref{wi08} yields \eqref{wi07}.

\medskip

\step6 Proof of (\ref{F.1}). 

\nobreak
Recall that by scaling we may assume $R=1$ so that $\delta=\left(\fint_{B_1}|(\phi,\sigma)|^2\right)^\frac12$. Inserting (\ref{wi01}) and (\ref{wi09}) into (\ref{wi07})
gives for $0<r,\rho\le\frac{1}{4}$ 
\begin{align*}
\fint_{B_r}|\nabla u-\partial_iu_{\ho}(0)(e_i+\nabla\phi_i)|^2\,\lesssim\,\Big(r^2+r^{-d-2}\big(\rho^\e+\rho^{-d-2}\fint_{B_1}|(\phi,\sigma)|^2\big)\Big)\fint_{B_1}|\nabla u|^2.
\end{align*}
Provided $\fint_{B_1}|(\phi,\sigma)|^2\ll 1$, we may choose $\rho=(\fint_{B_1}|(\phi,\sigma)|^2)^\frac{1}{d+2+\e}$
and obtain (\ref{F.1}) with $\frac{\e}{d+2+\e}$ playing the role of $\e$ 
and $\nabla u_\ho(0)$ the role of $\xi$.
If $\fint_{B_1}|(\phi,\sigma)|^2\gtrsim 1$ or $r\ge\frac{1}{4}$ we may choose $\xi=0$ and thus trivially
obtain (\ref{F.1}).

\medskip

\step7 Proof of (\ref{p.20}). 

\nobreak
The upper bound follows from Caccioppoli's estimate, cf.~(\ref{wi11}),
applied to the $a$-harmonic function $\xi\cdot x+\phi_\xi$ in form of
\begin{align*}
\big(\fint_{B_\frac{1}{2}}|\xi+\nabla\phi_\xi|^2\big)^\frac{1}{2}\lesssim
\big(\fint_{B_1}(\xi\cdot x+\phi_\xi)^2\big)^\frac{1}{2},
\end{align*}
followed by the triangle inequality in $L^2$. The lower bound follows from Poincar\'e's inequality
(with mean-value zero) and the triangle inequality in $L^2$
\begin{eqnarray*}
\big(\fint_{B_\frac{1}{2}}|\xi+\nabla\phi_\xi|^2\big)^\frac{1}{2}&\gtrsim&
\big(\fint_{B_\frac{1}{2}}(\xi\cdot x+\phi_\xi-\fint_{B_\frac{1}{2}}\phi_\xi)^2\big)^\frac{1}{2}
\\
&\ge& \big(\fint_{B_\frac{1}{2}}(\xi\cdot x)^2\big)^\frac12-\big(\fint_{B_\frac{1}{2}}(\phi_\xi-\fint_{B_\frac{1}{2}}\phi_\xi)^2\big)^\frac{1}{2}
\\
&\ge& \big(\fint_{B_\frac{1}{2}}(\xi\cdot x)^2\big)^\frac12-\big(\fint_{B_\frac{1}{2}} \phi_\xi^2\big)^\frac{1}{2}
\\
&\ge&C(d)|\xi|-\frac1{C(d)} \big(\fint_{B_1} \phi_\xi^2\big)^\frac{1}{2}.
\end{eqnarray*}
%



\subsection{Proof of Theorem~\ref{O}: Excess-decay and the minimal radius} 
We split the proof into two steps, and make use of the short-hand notation $\Exc(r):=\Exc(\nabla u;B_r)$.

\medskip

\step{1} Proof of \eqref{o.6} and \eqref{p.22}.

\noindent 

Given a $\delta\le 1$ to be fixed in the sequel as a function of $d$, $\lambda>0$, and $\alpha<1$,
we define $r_*$ in line with \eqref{O.1} as
\begin{align*}
r_*=\inf\Big\{r>0\;|\;\forall \rho \ge r \quad \frac{1}{\rho^2}\fint_{B_\rho}|(\phi,\sigma)-\fint_{B_\rho}(\phi,\sigma)|^2\le\delta\Big\}.
\end{align*}
Let $R\ge r_*$.
With this notation, \eqref{F.1} in Proposition~\ref{F} assumes the form
\begin{align}\label{f01}
\Exc(r)\le C_1\big((\frac{r}{R'})^2+\delta^{2\e}(\frac{R'}{r})^{d+2}\big)\fint_{B_{R'}}|\nabla u|^2
\end{align}
for all $a$-harmonic functions $u$ in $B_R$ and all radii $r_*\le r\le R'\le R$, 
where $C_1$ denotes some constant only depending on $d$, $\lambda$, and $\alpha$
the value of which we retain momentarily. Replacing the $a$-harmonic function $u$
by the $a$-harmonic function $x\mapsto u(x)-(\xi\cdot x+\phi_\xi(x))$, cf.~\eqref{f.2}, and optimizing in $\xi$, 
\eqref{f01} yields
\begin{align}\label{f02}
\Exc(r)\le C_1\big((\frac{r}{R'})^2+\delta^{2\e}(\frac{R'}{r})^{d+2}\big)\Exc(R').
\end{align}
We now first choose $\theta\le 1$, 
which is a placeholder for the ratio $\frac{r}{R'}$, so small that $C_1\theta^{2}\le\frac{1}{2}\theta^{2\alpha}$.
Since $\alpha<1$, this can be done and $\theta$ just depends on $d$, $\lambda$, and $\alpha$. 
We then choose $\delta\le 1$ so small that $C_1\delta^{2\e}(\frac{R'}{r})^{d+2}\le\frac{1}{2}\theta^{2\alpha}$;
again, this $\delta$ just depends on $d$, $\lambda$, and $\alpha$. With these choices, (\ref{f02}) assumes
the form
\begin{align*}
\Exc(\theta R')\le \theta^{2\alpha}\Exc(R')
\end{align*}
for all radii $R'\le R$ with $R'\ge r_*$. It is this form that may be iterated to yield
\begin{align*}
\Exc(\theta^n R)\le (\theta^n)^{2\alpha}\Exc(R)
\end{align*}
for all $n\in\mathbb{N}$ with $\theta^{n-1} R\ge r_*$. 

\medskip

For  $r_*\le r\le R$, choose now $n$ such that $\theta^{n+1}R<r\le\theta^n R$ and thus on the one hand
$\theta^n\le\theta^{-1}\frac{r}{R}$ while on the other hand $\Exc(r)\le\theta^{-d}\Exc(\theta^n R)$.
This implies the desired estimate (\ref{o.6}) 
\begin{equation}\nonumber
\Exc(r)\le\theta^{-(d+2\alpha)}(\frac{r}{R})^{2\alpha} \Exc(R).
\end{equation}

\medskip

Clearly, (\ref{p.22}) is an immediate consequence of (\ref{p.20}), possibly
further reducing the constant in (\ref{O.1}).

\medskip

\step{2} Proof of (\ref{o.30}). 

\noindent In view of the non-degeneracy condition (\ref{p.22}), for any $r_*\le \rho\le R$, there exists a unique
$\xi_{\rho}\in\mathbb{R}^d$ such that
\begin{equation}\label{o.33}
\fint_{B_\rho}|\nabla u-(\xi_{\rho}+\nabla\phi_{\xi_{\rho}})|^2
=\Exc(\rho),
\end{equation}
so that $\xi_{\rho}$ can be interpreted as an effective gradient of $u$ on scale $\rho$.
We claim that the dependence of $\xi_{\rho}$ on the scale $\rho$ is well-controlled by the excess
in the sense that for all $R\ge R'\ge r\ge r_*$
\begin{equation}\label{o.32}
|\xi_{r}-\xi_{R'}|^2\lesssim
\Exc(R'), 
\end{equation}
here and below $\lesssim$ denotes $\le$ up to a generic constant that
only depends on $d$ and $\alpha>0$. By a dyadic argument which we will sketch presently,
it is enough to consider two radii $\rho$ and $R'$
that are close in the sense of $\rho\le R'\le 2\rho$ and to show
\begin{equation}\label{o.31}
|\xi_{\rho}-\xi_{R'}|^2\lesssim\Exc(R').
\end{equation}
Here comes the dyadic argument: Let $N$ be the non-negative integer such that $2^{-(N+1)}R'< \rho\le 2^{-N}R'$.
By (\ref{o.31}) for $n=0,\cdots,N-1$ we have
\begin{equation}\nonumber
|\xi_{\rho}-\xi_{2^{-N}R'}|^2\lesssim\Exc(2^{-N}R'),\quad
|\xi_{2^{-(n+1)}R'}-\xi_{2^{-n}R'}|^2\lesssim\Exc(2^{-n}R'),
\end{equation}
and thus by the triangle inequality and since $\alpha>0$, we obtain (\ref{o.32}):
\begin{equation}\nonumber
|\xi_{\rho}-\xi_{R'}|^2\lesssim \left(\sum_{n=0}^N\sqrt{\Exc(2^{-n}R')}\right)^2
\stackrel{(\ref{o.6})}{\lesssim} \left(\sum_{n=0}^{N} (2^{-n})^{\alpha}\sqrt{\Exc(R')}\right)^2
\stackrel{\alpha>0}{\lesssim}\Exc(R').
\end{equation}
We now turn to the argument for (\ref{o.31}):
By the non-degeneracy condition (\ref{p.22}) on scale $\rho$ 
applied to $\xi_{\rho}-\xi_{R'}$, we have
\begin{eqnarray*}
|\xi_{\rho}-\xi_{R'}|^2&\lesssim&\fint_{B_\rho}|(\xi_{\rho}-\xi_{R'})+\nabla\phi_{\xi_{\rho}-\xi_{R'}}|^2,
\end{eqnarray*}
which by linearity we may rewrite as
\begin{eqnarray*}
|\xi_{\rho}-\xi_{R'}|^2&\lesssim&\fint_{B_\rho}|(\xi_{\rho}+\nabla\phi_{\xi_{\rho}})-(\xi_{R'}+\nabla\phi_{\xi_{R'}})|^2,
\end{eqnarray*}
so that by the triangle inequality in $L^2(B_\rho)$, and using $\rho\sim R'$, we obtain
\begin{eqnarray*}
|\xi_{\rho}-\xi_{R'}|^2&\lesssim&\fint_{B_\rho}|\nabla u-(\xi_{\rho}+\nabla\phi_{\xi_{\rho}})|^2
+\fint_{B_{R'}}|\nabla u-(\xi_{R'}+\nabla\phi_{\xi_{R'}})|^2.
\end{eqnarray*}
By definition (\ref{o.33}), and using once more $\rho\sim R'$  this turns as desired into
\begin{eqnarray*}
|\xi_{\rho}-\xi_{R'}|^2&\lesssim&\Exc(\rho)+\Exc(R')\lesssim \Exc(R').
\end{eqnarray*}

\medskip

We now may conclude the argument for (\ref{o.30}). By the triangle inequality in $L^2$, the definition \eqref{o.8} of the excess, and the non-degeneracy condition (\ref{p.22}), we get the two estimates
\begin{eqnarray}
  \fint_{B_r}|\nabla u|^2&\lesssim&\Exc(r)+|\xi_r|^2,\nonumber \\
  \Exc(R)+|\xi_R|^2&\lesssim& \fint_{B_R}|\nabla u|^2,\label{f08}
\end{eqnarray}
which combined with \eqref{o.6} in form of $\Exc(r)\lesssim\Exc(R)\leq\fint_{B_R}|\nabla u|^2$ and \eqref{o.32} (with $R'=R$) yields  (\ref{o.30}) as desired.


\subsection{Proof of Corollary~\ref{c}: Almost-sure Liouville property}

The Liouville property is a fairly simple consequence of
Theorem~\ref{O} and the following sublinear growth property
\begin{equation}\label{l.10}
\lim_{r\uparrow\infty}\frac{1}{r^2}\fint_{B_r}|(\phi,\sigma)-\fint_{B_r}(\phi,\sigma)|^2=0
\quad\mbox{for a.\ e.}\;a.
\end{equation}
For $\phi$ this statement (in a more involved form) is a key ingredient for the quenched invariance
principle and can be established based on ergodicity and stationarity, see \cite{Sidoravicius-Sznitman-04}.
We argue in Step~1 that the same argument can be used to establish
this property for $\sigma$.
\medskip

\step{1} Proof of \eqref{l.10}.

\nobreak

To keep notation lean, we just focus on $\sigma$ and consider only one of the components $\sigma_{ijk}$
of the tensor field $\sigma$. We drop the indices. The key property of the random, typically non-stationary 
field $\sigma(a,x)$ is that
\begin{equation}\nonumber
\nabla\sigma\quad\mbox{is stationary and of zero expectation and finite variance},
\end{equation}
see (\ref{si.2}) in the statement of Lemma~\ref{si}. 
For all $r>0$, define the rescaled tensor field $\sigma_r(x):=r^{-1} \Big(\sigma (rx)-\fint_{B} \sigma(ry)dy\Big)$.
On the one hand, by the pointwise ergodic theorem,  
\begin{equation}\label{l.9}
\nabla \sigma_r = (\nabla \sigma)(r\cdot) \,\stackrel{r\uparrow \infty}{\rightharpoonup} \, 0
\end{equation}
weakly in $L^2(B)$ almost surely, so that $\int_B |\nabla \sigma_r|^2$ is a bounded sequence almost surely.
On the other hand, by Poincar\'e's inequality on $B$ with mean value zero
\begin{equation}\label{I.9bis}
\int_B |\sigma_r|^2 \,\lesssim\, \int_B |\nabla \sigma_r|^2 \,\le \sup_{\rho \ge 1} \int_B |\nabla \sigma_\rho|^2 \,<\infty.
\end{equation}
By \eqref{l.9}, \eqref{I.9bis}, and the Rellich theorem, $\sigma_r$ thus converges strongly to 0 in $L^2(B)$ as $r\uparrow \infty$  almost surely.
Rescaling back, this yields \eqref{l.10}.

\medskip

\step 2 Conclusion.

We now give the argument for the almost-sure Liouville property. Recall
our short-hand notation $\Exc(r):=\Exc(\nabla u;B_r)$.
By
(\ref{l.10}), we may restrict ourselves to those coefficient fields  for which 
$\lim_{r\uparrow\infty}\frac{1}{r^{2}}\fint_{B_r}|(\phi,\sigma)-\fint_{B_r}(\phi,\sigma)|^2=0$.
Hence there exists a radius $r<\infty$ such
that (\ref{O.1}) holds for $C(\alpha,d,\lambda)$. Now we are given an $a$-harmonic
function $u$ with (\ref{l.11}). By Caccioppoli's estimate \eqref{f.51} (cf.~Lemma~\ref{L:reg} in Appendix~\ref{append:Cacc}), this can be upgraded to
\begin{equation}
\lim_{R\uparrow\infty}\frac{1}{R^{2\alpha}}\fint_{B_R}|\nabla u|^2=0,
\end{equation}
which in turn trivially yields
\begin{equation}
\lim_{R\uparrow\infty}\frac{1}{R^{2\alpha}}\Exc(R)=0.
\end{equation}
By (\ref{o.6}) this implies for all $\rho\ge r$
\begin{equation}
\inf_{\xi\in\mathbb{R}^d}\fint_{B_\rho}|\nabla u-(\xi+\nabla\phi_\xi)|^2=\Exc(\rho)=0,
\end{equation}
that is
\begin{equation}
\forall \rho<\infty\quad\exists \xi\in\mathbb{R}^d\quad\text{ s.t. }\quad
\nabla u=\xi+\nabla\phi_\xi\quad\mbox{a.e.~in}\;B_\rho,
\end{equation}
which upgrades to
\begin{equation}
\exists \xi\in\mathbb{R}^d\quad\text{s.t.}\quad
\nabla u=\xi+\nabla\phi_\xi\quad\mbox{a.e.~in}\;\mathbb{R}^d,
\end{equation}
and thus in turn implies (\ref{l.12}).


\subsection{Proof of Corollary~\ref{co}: Intrinsic large-scale $C^{1,1-}$-regularity}

In this proof, we use the short-hand notation $\Exc(D):=\Exc(\nabla u;D)$ for any domain $D$.
In view of the non-degeneracy condition (\ref{p.22}), for any $\rho\ge r_*(\pm x)$, 
there exists a unique $\xi_{\rho,\pm}\in\mathbb{R}^d$ such that
\begin{equation}\label{m.6}
\fint_{B_\rho(x_\pm)}|\nabla u-(\xi_{\rho,\pm}+\nabla\phi_{\xi_{\rho,\pm}})|^2
=\Exc(B_\rho(x_\pm)),
\end{equation}
so that $\xi_{\rho,\pm}$ can be interpreted as an effective gradient of $u$ at $\pm x$ on scale $\rho$.
Recall that we use the shorthand notation $\xi_{\pm}=\xi_{r_*,\pm}$.
As in (\ref{o.32}) in the proof of Theorem~\ref{O} we have
that the dependence of $\xi_{\rho,\pm}$ on the scale $\rho$ is well-controlled by the excess
in the sense that we have for all $r\ge r_*(\pm x)$
\begin{equation}\label{m.4}
|\xi_{\pm}-\xi_{r,\pm}|^2\lesssim\Exc(B_r(\pm x)),
\end{equation}
where here and in the remainder of the proof, $\lesssim$ denotes $\le$ up to a generic constant that
only depends on $d$, $\lambda$, and $\alpha$.

\medskip

We set for abbreviation
\begin{equation}\label{m.7}
r:=\max\{4|x|,2r_*(x),r_*(-x)\}\quad\mbox{so that}\;\frac{r}{4}\ge|x|,\;\frac{r}{2}\ge r_*(x),\;
r\ge r_*(-x).
\end{equation}
We now claim that on this scale $r$ (which up to the cut-off $r_*$ is essentially
the distance between the points $x$ and $-x$),
the difference of the corresponding effective gradients $\xi_{r,+}$ and $\xi_{r,-}$ is well-controlled
by the excess on that scale in the sense of
\begin{equation}\label{m.8}
|\xi_{r,+}-\xi_{r,-}|^2\lesssim \Exc(B_r(x))+\Exc(B_r(-x)).
\end{equation}
Indeed, by the non-degeneracy condition (\ref{p.22}) and thanks to (\ref{m.7}),
we have
\begin{equation}\nonumber
|\xi_{r,+}-\xi_{r,-}|^2\lesssim\fint_{B_\frac{r}{2}(x)}|(\xi_{r,+}-\xi_{r,-})+\nabla\phi_{\xi_{r,+}-\xi_{r,-}}|^2.
\end{equation}
By linearity of $\nabla\phi_\xi$ in $\xi$, the triangle inequality, 
and $B_\frac{r}{2}(x)\stackrel{(\ref{m.7})}{\subset} B_r(\pm x)$, this yields
\begin{equation}\nonumber
|\xi_{r,+}-\xi_{r,-}|^2\lesssim\fint_{B_r(x)}|\nabla u-(\xi_{r,+}+\nabla\phi_{\xi_{r,+}})|^2
+\fint_{B_r(-x)}|\nabla u-(\xi_{r,-}+\nabla\phi_{\xi_{r,-}})|^2,
\end{equation}
which turns into (\ref{m.8}) by definition of $\xi_r$ and of the excess.

\medskip

By the triangle inequality, estimates (\ref{m.4}) and (\ref{m.8}) combine to
\begin{equation}\label{m.10}
|\xi_{+}-\xi_{-}|^2\lesssim
\Exc(B_r(x))+\Exc(B_r(-x)).
\end{equation}
Since by (\ref{m.7}) we have $r\ge r_*(\pm x)$, and by assumption on $R$ we have $r\le R$,
we may apply Theorem~\ref{O} to the effect of
\begin{equation}\nonumber
\Exc(B_r(x))+\Exc(B_r(-x))\lesssim(\frac{r}{R})^{2\alpha}\big(\Exc(B_\frac{R}{2}(-x))
+\Exc(B_\frac{R}{2}(x))\big).
\end{equation}
Since by assumption $R\ge 4|x|$ we have in particular $B_\frac{R}{2}(\pm x)\subset B_{R}$ so that
trivially by definition of the excess,
\begin{equation}\nonumber
\Exc(B_\frac{R}{2}(x))
+\Exc(B_\frac{R}{2}(-x))\lesssim \Exc(B_R).
\end{equation}
The combination of the three last estimates turns into (\ref{m.11}).


\subsection{Proof of Corollary~\ref{cor:Lip}: Intrinsic large-scale Schauder-estimates}

We select $\alpha'\in(\alpha,1)$, say $\alpha'$ $:=\frac{1+\alpha}{2}$, and choose $C$ in the definition (\ref{O.1})
of $r_*$ so small that Theorem~\ref{O} holds with $\alpha'$ playing the role of $\alpha$.

\medskip

\step{1} Proof of \eqref{eq:SE:1} and \eqref{eq:SE:3}.

Let $r_*\le r\le \rho\le R$. We first argue that for some constant $C_1=C_1(d,\lambda,\alpha)$ we have
\begin{align}\label{f04}
\Exc(\nabla u+g;B_r)&\le C_1\Big((\frac{r}{\rho})^{2\alpha'}\Exc(\nabla u+g;B_{\rho})\nonumber\\
&+(\frac{\rho}{r})^d\fint_{B_{\rho}}(|g-\fint_{B_{\rho}}g|^2+|h-\fint_{B_{\rho}}h|^2)\Big).
\end{align}
To this end for $\xi:=\fint_{B_{\rho}}g$ we consider the Lax-Milgram solution $w$ of
\begin{align*}
\begin{array}{rcll}
  -\nabla\cdot a\nabla w&=&\nabla\cdot (a(g-\xi)+h) &\text{in }B_{\rho},\\
    w&=&0&\text{on }\partial B_{\rho},
\end{array}
\end{align*}
which is made such that on the one hand, by (\ref{cw51})
$x\mapsto u+\xi\cdot x-w$ is $a$-harmonic in $B_{\rho}$, so that by (\ref{o.6}),
\begin{align*}
\Exc(\nabla u+\xi-\nabla w;B_r)\lesssim(\frac{r}{\rho})^{2\alpha'}\Exc(\nabla u+\xi-\nabla w;B_{\rho}),
\end{align*}
and on the other hand, one has the energy estimate
\begin{align*}
\fint_{B_{\rho}}|\nabla w|^2\lesssim\fint_{B_{\rho}}(|g-\fint_{B_{\rho}}g|^2+|h-\fint_{B_{\rho}}h|^2).
\end{align*}
By the triangle inequality in $L^2$ and $\fint_{B_r}\le(\frac{\rho}{r})^d\fint_{B_{\rho}}$,
the combination of these implies~(\ref{f04}).

\medskip

We then argue in favor of \eqref{eq:SE:1} based on \eqref{f04}, which we rewrite in terms of $\theta=\frac{r}{\rho}$:
\begin{align*}
\Exc(\nabla u+g;B_{\theta \rho})
&\le C_1\big(\theta^{2\alpha'}\Exc(\nabla u+g;B_{\rho})\nonumber\\
&+\theta^{-d}\fint_{B_{\rho}}(|g-\fint_{B_{\rho}}g|^2+|h-\fint_{B_{\rho}}h|^2)\big),
\end{align*}
divide by $(\theta \rho)^{2\alpha}$, and take the supremum over $\rho\in[\frac{r_*}{\theta},R]$:
\begin{align*}
\sup_{r\in[r_*,\theta R]}\frac{1}{r^{2\alpha}}\Exc(\nabla u+g;B_{r})
&\le C_1\big(\theta^{2(\alpha'-\alpha)}\sup_{r\in[r_*,R]}\frac{1}{r^{2\alpha}}\Exc(\nabla u+g;B_r)\nonumber\\
&+\theta^{-d-2\alpha}
\sup_{r\in[r_*,R]}\frac{1}{r^{2\alpha}}\fint_{B_{r}}(|g-\fint_{B_r}g|^2+|h-\fint_{B_r}h|^2)\big).
\end{align*}
We now choose $\theta=\theta(d,\lambda,\alpha)\le 1$ so small that 
$C_1\theta^{2(\alpha'-\alpha)}\le\frac{1}{2}$; which yields 
\begin{align*}
\sup_{r\in[r_*,\theta R]}\frac{1}{r^{2\alpha}}\Exc(\nabla u+g;B_{r})
&\lesssim \sup_{r\in[\theta R,R]}\frac{1}{r^{2\alpha}}\Exc(\nabla u+g;B_r)\nonumber\\
&+\sup_{r\in[r_*,R]}\frac{1}{r^{2\alpha}}\fint_{B_{r}}(|g-\fint_{B_r}g|^2+|h-\fint_{B_r}h|^2)\big).
\end{align*}
Since $\sup_{r\in[\theta R,R]}\frac{1}{r^{2\alpha}}\Exc(\nabla u+g;B_r)$
$\lesssim \frac{1}{R^{2\alpha}}\Exc(\nabla u+g;B_R)$, this yields 
(\ref{eq:SE:1}) in case of $R<\infty$. 
In case of $R=\infty$ we obtain (\ref{eq:SE:3}) from (\ref{eq:SE:1}) in the limit $R\uparrow\infty$
by the square integrability of $\nabla u+g$ on $\mathbb{R}^d$ in form of 
$\Exc(\nabla u+g,B_R)\le\fint_{B_R}|\nabla u+g|^2\downarrow 0$.

\medskip

\step{2} Proof of \eqref{eq:SE:2}.

\nobreak
Starting point is \eqref{eq:SE:1}, which also holds in the more general form of: For all $r\ge r_*$,
\begin{align*}
\sup_{\rho\in[r,R]}(\frac{R}{\rho})^{2\alpha}\Exc(\nabla u+g;B_\rho)
&\lesssim \Exc(\nabla u+g;B_{R})\nonumber\\
&+\sup_{\rho\in[r,R]}(\frac{R}{\rho})^{2\alpha}\fint_{B_\rho}(|g-\fint_{B_\rho}g|^2+|h-\fint_{B_\rho}h|^2)
\end{align*}
since we may increase $r_*$ at our pleasure.
As in Step 2 of the proof of Theorem~\ref{O} we denote by $\xi_r$ the optimal $\xi$ in the definition
of $\Exc(\nabla u+g;B_r)$. An inspection of the proof of \eqref{o.32} in that step shows that we have
\begin{align*}
|\xi_r-\xi_R|^2\lesssim \sup_{\rho\in[r,R]}(\frac{R}{\rho})^{2\alpha}\Exc(\nabla u+g;B_\rho)
\end{align*}
as soon as $\alpha>0$.
Giving away some and using the triangle inequality in $\mathbb{R}^d$, the two last estimates combine to
\begin{align}\label{f10}
|\xi_r|^2+\Exc(\nabla u+g;B_r)
&\lesssim|\xi_R|^2+\Exc(\nabla u+g;B_{R})\nonumber\\
&+\sup_{\rho\in[r,R]}(\frac{R}{\rho})^{2\alpha}\fint_{B_\rho}(|g-\fint_{B_\rho}g|^2+|h-\fint_{B_\rho}h|^2).
\end{align}
Combined with the triangle inequality in $L^2$, the definition of the excess, and the non-degeneracy
property in the form of
\begin{align*}
\fint_{B_r}|\nabla u+g|^2&\lesssim |\xi_r|^2+\Exc(\nabla u+g;B_r),\nonumber\\
|\xi_R|^2+\Exc(\nabla u+g;B_R)&\lesssim \fint_{B_R}|\nabla u+g|^2,
\end{align*}
cf.~(\ref{f08}), we may pass from \eqref{f10} to \eqref{eq:SE:2}.


\subsection{Proof of Corollary~\ref{cor:Lp}: Large-scale Calder\'on-Zygmund estimates}

We follow the standard approach to Calder\'on-Zygmund in the constant-coefficients case 
that passes via a BMO-estimate (see for instance
\cite[Section 7.1.1]{Giaquinta-Martinazzi-05}). The main ingredient is excess decay for solutions of the homogeneous equation (cf.~the excess-decay estimate \eqref{eq:SE:1} in Corollary~\ref{cor:Lip}
in our variable-coefficients case).
There are two main differences with respect to \cite[Section 7.1.1]{Giaquinta-Martinazzi-05}: First, the excess decay is limited to the scale $r_*$,
and second, we work on $L^2$-based quantities rather than $L^1$-based quantities.
In Step~1, we choose a suitable minimal radius $\ru$ for the estimate (which we then simply call $r_*$ in the rest of the proof).
In Step~2, we show that we control the  energy by the  intrinsic excess on dyadic cubes.
In Step~3, we turn to the control of sub-level sets, which yields control of the  $L^p$-norm in Step~4. In Step~5 we 
prove the equivalence of discrete and continuous norms, which we use in Step~6 to prove \eqref{I1} in the range $2\le p<\infty$. In Step~7 we argue by duality to derive \eqref{I1}
in the remaining range of exponents $1<p\le 2$. 

\medskip

\step1 Choice of $\ru$.

For all $0<L\le 1$, it turns out that we may choose $\ru$ to be the largest function with Lipschitz constant $L$
below $r_*$, that is,
\begin{equation}\label{fff54}
\ru(x)=\inf_{y\in\mathbb{R}}(r_*(y)+L|x-y|).
\end{equation}
This implies that (\ref{O.1}) survives with $r_*$ replaced by $\ru$ at the expense of a worse
constant~$C$:
\begin{equation}\nonumber
\frac{1}{R}\Big(\fint_{B_R(x)}|(\phi,\sigma)-\fint_{B_R(x)}(\phi,\sigma)|^2\Big)^\frac{1}{2}\le(\frac{1}{L}+1)^{\frac{d}{2}+1}\frac1C
\quad\mbox{for all}\;R\ge\ru(x)\;\mbox{and}\;x\in\mathbb{R}^d,
\end{equation}
Indeed, for a point $x\in\mathbb{R}^d$ and a radius $R<\infty$ with $\ru(x)<R$, 
by definition (\ref{fff54}), there exists $y\in\mathbb{R}$ such that $|x-y|\le\frac{R}{L}$ and $r_*(y)\le R$.
The former implies $B_R(x)\subset B_{\bar R}(y)$ where $\bar R:=(\frac{1}{L}+1)R$ so that
$$
\frac{1}{R}\Big(\fint_{B_{R}(x)}|(\phi,\sigma)
-\fint_{B_{R}(x)}(\phi,\sigma)|^2\Big)^\frac{1}{2}
\le (\frac{\bar R}{R})^{\frac{d}{2}+1}\frac{1}{\bar R}\Big(\fint_{B_{\bar R}(y)}|(\phi,\sigma)
-\fint_{B_{\bar R}(y)}(\phi,\sigma)|^2\Big)^\frac{1}{2}.
$$
The latter implies 
$$
\frac{1}{\bar R}\Big(\fint_{B_{\bar R}(y)}|(\phi,\sigma)
-\fint_{B_{\bar R}(y)}(\phi,\sigma)|^2\Big)^\frac{1}{2}\le(\frac{1}{L}+1)^{\frac{d}{2}+1}\frac1C.
$$
Choosing $L=\felix$ and $C=\Felix^{d+2}C_0$, we thus have that $\ru$ is $\felix$-Lipschitz and satisfies
$r_*(C_0)\le \ru \le r_*(\Felix^{d+2}C_0)$, as claimed. In particular, we have the mean-value property for all $R\ge \ru$.
In the rest of the proof, we use the short-hand notation $r_*$ for $\ru$.

\medskip

\step2 Control of the energy via the intrinsic excess.

In the standard approach to Calder\'on-Zygmund for constant coefficients, 
see for instance \cite[Proposition 6.31]{Giaquinta-Martinazzi-05}, the main part consists in controlling the
(standard) energy density $\rho_{stan}$ by the (standard) excess density $e_{stan}$ (the ``sharp function'' of $\nabla u$, see
for instance \cite[Section 6.3.4]{Giaquinta-Martinazzi-05}) where
\begin{equation}\label{fff70}
\rho_{stan}:=|\nabla u|^2\quad\mbox{and}\quad
e_{stan}(x):=\sup_{r} \Big(\inf_\xi\fint_{B_r(x)}|\nabla u-\xi|^2\Big),
\end{equation}
which amounts to a result by Fefferman \& Stein,
see for instance \cite[Theorem 6.30]{Giaquinta-Martinazzi-05}. More precisely, the next simple 
step in the standard approach consists in deriving the weak-strong estimate
\begin{equation}\label{fff62}
\left.\begin{array}{l}
|B\cap\{\rho_{stan}\le 1\}|\sim|B|\\[1ex]
M\gg 1,\;\;\int_Be_{stan}\lesssim |B|
\end{array}\right\}\quad\Longrightarrow\quad|B\cap\{\rho_{stan}\ge M\}|\lesssim\frac{1}{M}\int_Be_{stan}
\end{equation}
for any ball $B\subset\mathbb{R}^d$.
In the standard theory, this is clearly based on the fact the $\xi$ in the definition of the excess is constant;
whereas in our case it is replaced by $\xi_i(e_i+\nabla \phi_i)$ and thus is only approximately constant on
scales $\gtrsim r_*$. In fact, we shall appeal to \eqref{p.22} in Theorem~\ref{O}, namely
\begin{equation}\label{fff60}
\fint_{B_R(x)}|\xi_i(e_i+\nabla \phi_i)|^2\sim|\xi|^2\quad\mbox{provided}\;R\gg  r_*(x).
\end{equation}
In what follows, it will be convenient to have a partition $\partn$ of $\mathbb{R}^d$ into dyadic cubes $Q\in \{2^k (\Z^d+[-\frac12,\frac12)^d),k\in \Z\}$ such that
\begin{equation}\label{fff75}
r_*\sim\diam Q\quad\mbox{on}\;Q.
\end{equation}
Such a partition can be constructed like a Calder\'on-Zygmund decomposition: ${\mathcal Q}$ consists of those dyadic 
cubes $Q$ which are such that $Q$ and all its ancestors $Q'$ (that is, $Q'\in  \{2^k (\Z^d+[-\frac12,\frac12)^d),k\in \Z\}$ and $Q\subset Q'$) satisfy
\begin{align*}
\fint_{Q}r_*\ge \diam Q\quad\mbox{and}\quad \fint_{Q'}r_*<\diam {Q'}.
\end{align*}
The second inequality implies
\begin{align*}
\fint_{Q}r_*\le 2^{d+1}\diam Q.
\end{align*}
Since $r_*$ is $\felix$-Lipschitz-continuous and since we may w.l.o.g.~assume that $r_*$ is bounded away from zero, ${\mathcal Q}$ defines indeed a (countable) partition 
of $\mathbb{R}^d$.
Since $r_*$ is $\felix$-Lipschitz continuous, the last inequalities imply (\ref{fff75}).
We now may pass from balls to cubes, more precisely, from  (\ref{fff60}) to
\begin{equation}\label{fff63}
\fint_{Q}|\xi_i(e_i+\nabla \phi_i)|^2\sim|\xi|^2\quad\mbox{for all}\;Q\in{\mathcal Q}.
\end{equation}
Indeed, select an $x\in Q$; by (\ref{fff75}) we have $B_r(x)\subset Q\subset B_R(x)$ for two radii $r,R\sim r_*(x)$,
so that (\ref{fff60}) translates into (\ref{fff63}).

\medskip

The price to pay for the lower scale $r_*$
is a softening of the standard version $\rho_{stan}$ of the energy density as follows
\begin{equation}\label{fff66}
\rho:=\fint_Q|\nabla u|^2 \quad\mbox{on}\;Q\quad\mbox{for all}\;Q\in{\partn}.
\end{equation}
Equipped with this modification, we recover (\ref{fff62}) for our objects: We will argue that for any 
dyadic $D$
\begin{equation}\label{fff66bis}
\left.\begin{array}{l}
|D\cap\{\rho\le 1\}| \sim|D|\\[1ex]
M\gg 1,\;\;e(D)\lesssim |D|
\end{array}\right\}\quad\Longrightarrow\quad|D\cap\{\rho\ge M\}| \lesssim\frac{1}{M}e(D),
\end{equation}
where we have set for abbreviation
\begin{equation}\label{fff73}
e(D):=  \inf_\xi \int_D |\nabla u-\xi_i(e_i+\nabla \phi_i)|^2.
\end{equation}
Note that because of the fact that $\rho$ is piecewise constant on $\partn$, cf.~\eqref{fff66}, \eqref{fff66bis}
is trivially satisfied for a $D$ that is contained in $\partn$ or finer.
We thus consider a dyadic $D$ coarser than elements of $\partn$, and let $Q\in{\partn}$ be an arbitrary cube contained in $D$.
On the one hand, we have by the upper bound
in (\ref{fff63}) and the definition (\ref{fff66}) of $\rho$:
\begin{align}\label{fff67}
\rho=\fint_Q|\nabla u|^2 \lesssim \fint_Q|\nabla u-\xi_i(e_i+\nabla \phi_i)|^2+|\xi|^2\quad\mbox{on}\;Q.
\end{align}
On the other hand, we have by the lower bound in (\ref{fff63})
\begin{align*}
|\xi|^2 \lesssim \fint_Q|\nabla u-\xi_i(e_i+\nabla \phi_i)|^2+\rho\quad\mbox{on}\;Q.
\end{align*}
Summing the integral of the latter over $Q$ for all $Q\in \partn$ such that $Q\subset D$ and $\rho_{|Q}\le 1$, we obtain
\begin{align*}
|D\cap\{\rho\le 1\}||\xi|^2 \lesssim \int_D |\nabla u-\xi_i(e_i+\nabla \phi_i)|^2 +|D|.
\end{align*}
Choosing $\xi$ to be the minimizer in (\ref{fff73}), this turns into
$|D\cap\{\rho\le 1\}||\xi|^2 \lesssim e(D)+|D|$,
so that by the assumptions in (\ref{fff66bis}) we have $|\xi|^2\lesssim 1$.
The combination of this with (\ref{fff67}) yields because of $M\gg 1$
\begin{align*}
\rho\ge M\;\mbox{on}\;Q\quad\Longrightarrow\quad \fint_Q|\nabla u-\xi_i(e_i+\nabla \phi_i)|^2  \gtrsim M,
\end{align*}
which we rewrite as (recall that $\rho$ is constant on $Q$)
\begin{align*}
|Q\cap\{\rho\ge M\}|\lesssim\frac{1}{M}  \int_Q|\nabla u-\xi_i(e_i+\nabla \phi_i)|^2.
\end{align*}
Summing over all $Q\subset D$ with $Q\in{\partn}$ yields the RHS of (\ref{fff66bis}).

\medskip

\step3 Control of the sub-level sets of the energy by the sub-level sets of the intrinsic excess.

The next step in the standard theory starts from (\ref{fff62}) and
establishes control of the global measure of sub-level sets of $\rho_{stan}$ by the one of sub-level sets of $e_{stan}$.
More precisely, it consists in passing from (\ref{fff62}) 
to
\begin{align}\label{fff71}
|\{\rho_{stan}\ge M\}|\lesssim|\{e_{stan}\ge\theta\}|+\frac{\theta}{M}|\{\rho_{stan} \ge 1\}|\quad
\mbox{for}\;\theta\ll 1\ll M,
\end{align}
see for instance \cite[Proposition 6.31]{Giaquinta-Martinazzi-05}.
This holds verbatim also in our case
\begin{align}\label{fff68}
|\{\rho\ge M\}| \lesssim|\{e\ge\theta\}|+\frac{\theta}{M}|\{\rho \ge 1\}| \quad
\mbox{for}\;\theta\ll 1\ll M,
\end{align}
where we even may relax the definition of the standard version $e_{stan}$ of the excess density,
cf.~(\ref{fff70}),
which could also be defined with help of the family of dyadic cubes $D$ as
$e_{stan}(x):=\sup_{D\ni x}\inf_\xi\Big(\fint_D|\nabla u-\xi|^2\Big)$
by restricting the supremum over dyadic cubes to those that are ancestors of cubes in the
decomposition ${\partn}$:
\begin{align}\label{fff65}
e(x):=\sup_{D\ni x}\Big\{\frac{e(D)}{|D|}\Big|Q\subset D\;\mbox{for some}\;Q\in{\partn}\Big\},
\end{align}
where $e(D)$ is defined in (\ref{fff73}).
The argument in passing from (\ref{fff66bis}) to (\ref{fff68}) is identical
to the argument for passing from (\ref{fff62}) to (\ref{fff71}).
By successive divisions we construct a Calder\'on-Zygmund partition ${\mathcal D}$ 
based on the characteristic function of $\{\rho\ge 1\}$. 
In other words, ${\mathcal D}$ consists of
those dyadic cubes $D$ such that it and all its ancestors $D'\supset D$ satisfy
\begin{align}\label{fff69}
|D\cap\{\rho\ge 1\}|>\frac{1}{2^{d+1}}|D|\quad\mbox{and}\quad|D'\cap\{\rho\ge 1\}|\le\frac{1}{2^{d+1}}|D'|.
\end{align}
This yields a disjoint decomposition of $\mathbb{R}^d$ into $\{D\}_{D\in {\mathcal D}}$
and a set where $\rho< 1$. Hence for (\ref{fff68}) it is enough to show for every
cube $D\in{\mathcal D}$:
\begin{align}\label{fff70b}
|D\cap\{\rho\ge M\}| \lesssim|D\cap\{e\ge\theta\}|+\frac{\theta}{M}|D\cap\{\rho \ge 1\}| \quad
\mbox{for}\;\theta\ll 1\ll M.
\end{align}
Note that the first and second properties in (\ref{fff69}) imply in particular
\begin{align}\label{fff74}
\frac{1}{2^{d+1}}|D| \le |D\cap\{\rho\ge 1\}| \le\frac{1}{2}|D|,
\end{align}
and therefore one of the LHS conditions in (\ref{fff66bis}). In addition, this yields that $\rho$ is not constant on $D$. By the definition of the dyadic decomposition $\partn$ and by
definition \eqref{fff66} of $\rho$
this implies that there is a $Q\in{\partn}$ such that $Q\subset D$ (strictly, in fact). Hence
in view of the definition (\ref{fff65}),
\begin{equation}\label{fff75bis}
|D| e(x)\ge e(D)\quad\mbox{for all }\;x\in Q.
\end{equation}
To recover the other LHS condition in (\ref{fff66bis}), we first consider the case 
$e(D)<\theta |D|\le |D|$
in which case the RHS of (\ref{fff66bis}) assumes the form $|D\cap\{\rho\ge M\}|\lesssim\frac{\theta}{M}|D|$,
which together with (\ref{fff74}) yields (\ref{fff70b}). It remains to consider the case $e(D)\ge\theta |D|$,
which by (\ref{fff75bis}) implies $D\cap\{e\ge\theta\}=D$ so that (\ref{fff70b}) is automatically met.

\medskip

\step4 Conversion of the control of the sub-level sets to an $L^p$-estimate.

The previous-to-last step in the standard argument is to convert (\ref{fff71})
into the $L^p$-estimate $\int\rho_{stan}^p\lesssim\int e_{stan}^p$ for any $1<p<\infty$, see for instance \cite[Theorem 6.30]{Giaquinta-Martinazzi-05}. By the same argument we may pass
from (\ref{fff68}) to
\begin{align}\label{fff73bis}
\int\rho^p  \lesssim\int e^p .
\end{align}
Indeed, by a scaling argument in form of $u=t\hat u$ we may upgrade (\ref{fff68}) to
\begin{align*}
|\{\rho\ge Mt\}| \lesssim|\{e\ge\theta t\}|+\frac{\theta}{M}|\{\rho \ge t\}|\quad
\mbox{for}\;\theta\ll 1\ll M\;\mbox{and all}\;t>0.
\end{align*}
Integrating against $t^{p-1}$ we obtain
\begin{align*}
\frac{1}{M^p}\int\rho^p \lesssim\frac{1}{\theta^p}\int e^p +\frac{\theta}{M}\int\rho^p \quad
\mbox{for}\;\theta\ll 1\ll M.
\end{align*}
This yields \eqref{fff73bis} by first fixing an $M\sim 1$ sufficiently large for which this inequality holds, and then choosing
$\theta\sim 1$ sufficiently small so that the last term may be absorbed.

\medskip

\step5 Equivalence of discrete and continuous norms.

In this step we prove that for all $1<p<\infty$ and non-negative functions $h$ we have
\begin{equation}
\Big(\int \Big( \fint_{B_{*}(x)} h\Big)^\frac p2dx\Big)^\frac2p \,\sim\, \Big(\sum_{Q\in \partn} |Q| \Big(\fint_Q h\Big)^\frac p2\Big)^\frac2p ,
\label{e.equiv-rP-1}
\end{equation}
where and $A\sim B$ means $\frac1C A \le B \le C B$ for a generic constant $C$ depending only on $d$ (and not on $p$).
In particular, for $p= 2$, this takes the form
\begin{eqnarray}
\int \fint_{B_*(x)} h \ dx&\sim& \int h.
\label{e.equiv-rP-2}
\end{eqnarray}
We split the rest of this step into two parts.

\medskip

\substep{5.1} Proof that for all $1\le p<\infty$,
\begin{equation}
\Big(\int \Big( \fint_{B_{*}(x)} h\Big)^\frac p2dx\Big)^\frac2p \,\lesssim\, \Big(\sum_{Q\in \partn} |Q| \Big(\fint_Q h\Big)^\frac p2\Big)^\frac2p .
\label{e.equiv-rP-11}
\end{equation}
For all $Q,Q'\in \partn$, we write $Q'\leadsto Q$ if there exists $x\in Q$ such that $B_*(x)\cap Q'\ne \emptyset$,
and first claim that if $Q'\leadsto Q$ then
\begin{eqnarray}
\diam {Q'}&\sim &\diam Q,
\label{e.eq-norms-2.1}\\
\dist (Q,Q')&\lesssim& \diam {Q'}.
\label{e.eq-norms-2.2}\
\end{eqnarray}
We first note that \eqref{e.eq-norms-2.1} implies \eqref{e.eq-norms-2.2}
in the form $\dist (Q,Q')\le r_*(x) \stackrel{\eqref{fff75}}\sim \diam Q\sim \diam{Q'}$.
We then prove  \eqref{e.eq-norms-2.1}, and let $y\in Q'$ be such that $y\in B_*(x)$.
By the Lipschitz continuity of $r_*$ we have
$$
|r_*(y)-r_*(x)| \leq \frac18 |x-y| \le \frac18 r_*(x),
$$
so that $\diam{Q'}\sim r_*(y) \sim r_*(x) \sim \diam Q$, that is, \eqref{e.eq-norms-2.1}.
We now argue that \eqref{e.eq-norms-2.1} and  \eqref{e.eq-norms-2.2} imply 
that
\begin{equation}\label{e.eq-norms-2.3}
\sup_{Q' \in \partn} \#\{Q \in \partn \,|\, Q' \leadsto Q\}\,\lesssim \, 1,\qquad \sup_{Q\in \partn} \#\{Q' \in \partn \,|\, Q' \leadsto Q\}\,\lesssim\, 1.
\end{equation}
We only prove the first estimate: From \eqref{e.eq-norms-2.2} we learn that $\cup_{Q:Q'\leadsto Q} Q \subset B_{C \diam {Q'}}(Q')$ for some generic $C=C(d)<\infty$ (which may change
from line to line in the estimates below),
whereas from \eqref{e.eq-norms-2.1} we learn that $|Q|\gtrsim \diam{Q'}^d$ for all $Q$ with $Q'\leadsto Q$. The combination
of these properties yields \eqref{e.eq-norms-2.3}.

We are in the position to conclude the proof of \eqref{e.equiv-rP-11}.
Let $Q \in \partn$. For all $x\in Q$, we have
$$
\fint_{B_*(x)} h \,\le \, \sum_{Q':Q'\leadsto Q} \frac{|Q'|}{|B_*(x)|} \fint_{Q'} h \stackrel{\eqref{e.eq-norms-2.1},\eqref{fff75}}\lesssim   \sum_{Q':Q'\leadsto Q} \fint_{Q'} h,
$$
so that
\begin{equation}\label{e.eq-norms-2.4}
\int \Big(\fint_{B_*(x)} h\Big)^{\frac p2} \,\le\, C^p\sum_{Q \in \partn} |Q|\Big( \sum_{Q':Q'\leadsto Q} \fint_{Q'} h \Big)^{\frac p2}.
\end{equation}
%
This implies for all $1\le p <\infty$
\begin{eqnarray*}
\int \Big(\fint_{B_*(x)} h\Big)^{\frac p2} dx &\le&C^p \sum_{Q\in \partn}   |Q|\sum_{Q':Q'\leadsto Q} \big(\sup_{Q\in \partn} \#\{Q' \in \partn \,|\, Q' \leadsto Q\}\big)^{\frac p2} \Big(\fint_{Q'} h  \Big)^{\frac p2}
\\
&\stackrel{\eqref{e.eq-norms-2.3}}\le &C^p\sum_{Q' \in \partn} \sum_{Q:Q'\leadsto Q} |Q|\Big(\fint_{Q'} h  \Big)^{\frac p2}
\\
&\stackrel{ \eqref{e.eq-norms-2.3},\eqref{e.eq-norms-2.1}}\le &C^p  \sum_{Q' \in \partn} |Q'|\Big(\fint_{Q'} h  \Big)^{\frac p2},
\end{eqnarray*}
as claimed.

\medskip
\substep{5.2} Proof that for all $1\le p<\infty$,
\begin{equation}
\Big(\int \Big( \fint_{B_{*}(x)} h\Big)^{\frac p2} dx\Big)^{\frac 2p} \,\gtrsim\, \Big(\sum_{Q\in \partn} |Q| \Big(\fint_Q h\Big)^{\frac p2}\Big)^{\frac 2p}.
\label{e.equiv-rP-12}
\end{equation}
Let $Q \in \partn$, $\ell=\diam Q$, and set $Q_r(x):=x+[-\frac{r}{2},\frac{r}{2})^d$.
Since $r_* \sim \ell$ on $Q$, for $0<\e\ll 1$ small enough, we have
\begin{equation}\label{zztop}
\fint_Q \Big(\fint_{B_*(x)} h\Big)^{\frac p2} dx\,\ge\,  C^{-p}\e^{d \frac p2  } \fint_Q \Big(\fint_{Q_{\e \ell}(x)} h\Big)^{\frac p2} dx,
\end{equation}
where $C$ denotes a constant depending only on $d$ (and that may change from line to line like in Substep~5.1).
W.l.o.g. we may assume that $\e$ is chosen such that $\{Q_{\frac12\e\ell}(z)\}_{z\in Z_\e}$ with $Z_\e:=\{z\in\frac12\e\ell\Z^d\,:\,Q_{\frac 12\e\ell}(z)\subset Q\}$ is a partition of $Q$ into disjoint cubes. Thus, \eqref{zztop} turns into
\begin{equation*}
  \fint_Q \Big( \fint_{B_{*}(x)} h\Big)^{\frac p2} dx \,\ge\, C^{-p}\e^{d\frac p2}\ell^{-d}\sum_{z\in Z_\e}\int_{Q_{\frac 12\e\ell}(z)}\Big(\fint_{Q_{\e \ell}(x)} h\Big)^{\frac p2} dx.
\end{equation*}
Since
$$
x\in Q_{\frac{1}{2}\e\ell}(z)\quad\Rightarrow\quad
 \Big(\fint_{Q_{\e \ell}(x)} h\Big)^{\frac p2}\,\ge\,2^{-d\frac{p}{2}}\Big(\fint_{Q_{\frac12 \e \ell}(z)} h\Big)^{\frac p2} \,,
$$
we conclude that
\begin{equation}\label{e.eq-norms-1.1}
  \fint_Q \Big( \fint_{B_{*}(x)} h\Big)^{\frac p2} dx \,\ge\, C^{-p}\e^{d\frac p2}\ \e^d\sum_{z\in Z_\e}\Big(\fint_{Q_{\frac12\e \ell}(z)} h\Big)^{\frac p2}.
\end{equation}
For $2\leq p<\infty$ Jensen's inequality yields
\begin{eqnarray*}
\fint_Q \Big(\fint_{B_*(x)} h\Big)^{\frac p2}\,\ge\,C^{-p} \e^{d\frac p2}  \Big(\e^d \sum_{z\in Z_\e}\fint_{Q_{\frac12\e \ell}(z)} h\Big)^{\frac p2}\, \ge\, C^{-p} \e^{d\frac p 2} \Big(\fint_Q h\Big)^{\frac p2},
\end{eqnarray*}
while for $1\le p<2$ the discrete estimate $\|\cdot\|_{\ell^\frac{2}{p}}\leq\|\cdot\|_{\ell^1}$ yields
\begin{eqnarray*}
\fint_Q \Big(\fint_{B_*(x)} h\Big)^{\frac p2}\,\ge\,C^{-p} \e^{d\frac p2}\e^d\Big(\sum_{z\in Z_\e}\fint_{Q_{\frac12\e \ell}(z)} h\Big)^{\frac p2}\, \ge\, C^{-p} \e^{d} \Big(\fint_Q h\Big)^{\frac p2}.
\end{eqnarray*}
Since $\e$ can be chosen only depending on $d$, \eqref{e.equiv-rP-12} follows.
\medskip

\step6 Proof of  \eqref{I1} for $p\ge 2$.

We now return from cubes to balls and start with the excess. 
Based on the Lipschitz continuity of $r_*$ we claim that for any point $x$
\begin{align}\label{fff84}
e(x)\lesssim\sup_{R> r_*(x)}\inf_\xi \Big(\fint_{B_R(x)}|\nabla u-\xi_i(e_i+\nabla \phi_i)|^2\Big).
\end{align}
The definition (\ref{fff65}) prompts us to prove that for any dyadic cube $D$ with $x\in D$ 
and such that there exists a $Q\in{\partn}$ with $Q\subset D$, we have
\begin{align*}
\frac{e(D)}{|D|}\lesssim\frac{e(B_R(x))}{|B_R|}\quad\mbox{for some}\;R\ge r_*(x).
\end{align*}
Since $Q\subset D$ for some $Q\in{\partn}$,
we have $\inf_Dr_*\lesssim \diam  Q\le \diam  D$ in view of (\ref{fff75}). 
By the Lipschitz continuity of $r_*$, $x\in D$ then yields $r_*(x)\lesssim \diam  D$,
which implies $D\subset B_R(x)$ for some $r_*(x)\le R\sim \diam  D$. This gives
both $e(D)\le e(B_R(x))$, cf.~\eqref{fff73}, and $|D| \gtrsim |B_R|$, so that the claim \eqref{fff84} follows.

\medskip

We continue to revert back to balls from cubes and look at the energy density.
By \eqref{e.equiv-rP-1} for $h=|\nabla u|^2$, we have 
\begin{align}\label{fff76}
\int(\fint_{B_{*}(x)}|\nabla u|^2)^\frac{p}{2}dx\lesssim\sum_{Q\in{\partn}} |Q| (\fint_Q|\nabla u|^2)^\frac{p}{2}\stackrel{(\ref{fff66})}{=}
\int\rho^p .
\end{align}
We may now conclude: We use (\ref{eq:SE:1}) with $\alpha=0$ and combine it with (\ref{fff84})
in the form
\begin{align*}
e(x)\lesssim\sup_{R>r_*(x)}\Big(\fint_{B_R(x)}|g|^2\Big)  .
\end{align*}
By \eqref{e.equiv-rP-1} for $p=2$, we have 
$$
\sup_{R>r_*(x)} \fint_{B_R(x)}|g|^2\,\lesssim\, \sup_{R>0} \fint_{B_R(x)} |g_*|^2,
$$
where $g_*(x):=\big(\fint_{B_{*}(x)} |g|^2\big)^\frac12$.
We now combine (\ref{fff73bis}) with the Maximal Function estimate applied  to $|g_*|^2$  with exponent $\frac p2>1$
to obtain for all $p>2$ 
\begin{align*}
\int \rho^p  \lesssim\int  \Big(\fint_{B_{*}(y)} |g|^2\Big)^\frac{p}{2}dy \,\lesssim \, \int   |g|^p,
\end{align*}
which in combination with (\ref{fff76}) yields the claim for $p>2$. 
For $p=2$, \eqref{fff76} takes the simpler form
\begin{align*}
\int \Big(\fint_{B_{*}(x)}|\nabla u|^2\Big) dx\sim \sum_{Q\in{\partn}} \int_Q|\nabla u|^2 =\int |\nabla u|^2,
\end{align*}
so that the result is a consequence of the simple energy estimate for \eqref{l1-0}.
The claim then follows by the Riesz-Thorin interpolation theorem.

\medskip

\step7 Proof of   \eqref{I1} for $1<p< 2$.

By  \eqref{e.equiv-rP-1}, it is enough to prove the claim by replacing the integral of averages on $\R^d$ by the sum of averages on the partition $\partn$.
We argue by a standard duality argument that appeals  to  \eqref{I1} for the dual problem (cf.~Remark~\ref{rem:duality}): By Step~6, for any exponent $2\le q<\infty$
and any decaying $v,h$ related through $-\nabla \cdot a^* \nabla v=\nabla \cdot h$, we have
\begin{equation}\label{e.for-duality}
\sum_{Q\in \partn} |Q|\Big(\fint_Q |\nabla v|^2\Big)^\frac q2 \,\lesssim \, \sum_{Q\in \partn} |Q|\Big(\fint_Q |h|^2\Big)^\frac q2.
\end{equation}
Now, by discrete duality, we have for our solution $u$ and all $1<p\le 2$ and $q=\frac{p}{p-1}$,
\begin{equation}\label{e.for-duality-2}
\Big(\sum_{Q\in \partn} |Q|\Big(\fint_Q |\nabla u|^2\Big)^\frac p2\Big)^\frac1p \,=\, \sup_{h\not\equiv 0} \frac{\sum_{Q \in \partn} |Q| \fint_Q \nabla u \cdot h }  {\Big( \sum_{Q\in \partn} |Q|\Big(\fint_Q |h|^2\Big)^\frac q2\Big)^\frac1q} .
\end{equation}
Let $h$ be a test function.
We then consider the solution $v$ of $-\nabla \cdot a^* \nabla v=\nabla \cdot h$  and compute
using the defining equations for $u$ and $v$, and H\"older's inequality with exponents~$(p,q)$
\begin{eqnarray*}
\Big|\sum_{Q \in \partn} |Q| \fint_Q \nabla u \cdot h\Big| &=&\Big|\int \nabla u \cdot h\Big|\,=\,\Big|\int \nabla u \cdot a^* \nabla v\Big|\,=\,\Big|\int \nabla v\cdot g\Big|
\\
&\le&\sum_{Q\in \partn} |Q| \Big(\fint_Q |\nabla v|^2\Big)^\frac12 \Big(\fint_Q |g|^2\Big)^\frac12 
\\
&\le&\bigg( \sum_{Q\in \partn} |Q| \Big(\fint_Q |\nabla v|^2\Big)^\frac q2\bigg)^\frac1q \bigg( \sum_{Q\in \partn} |Q| \Big(\fint_Q |g|^2\Big)^\frac p2\bigg)^\frac1p
\\
&\stackrel{\eqref{e.for-duality}}{\lesssim}&\bigg( \sum_{Q\in \partn} |Q| \Big(\fint_Q |h|^2\Big)^\frac q2\bigg)^\frac1q \bigg( \sum_{Q\in \partn} |Q| \Big(\fint_Q |g|^2\Big)^\frac p2\bigg)^\frac1p,
\end{eqnarray*}
from which   \eqref{I1}  on the level of $\partn$ follows by \eqref{e.for-duality-2} and the arbitrariness of $h$

\subsection{Proof of Corollary~\ref{cor:Lp-weight}: Large-scale weighted Calder\'on-Zygmund estimates}

We split the proof into four steps.

\medskip

\step1 Decay property in $L^2$. 

Suppose in addition that $\supp g\subset B_r$ for some $r\ge \underline{r_*}(0)$.
Then we claim for all $R\ge r$
\begin{align}\label{EM01}
\big(\frac{1}{R^d}\int_{|x|>R}|\nabla u|^2\big)^\frac{1}{2}
\lesssim(\frac{r}{R})^d\big(\frac{1}{r^d}\int_{|x|<r}|g|^2\big)^\frac{1}{2}.
\end{align}
We argue by duality: Given a square-integrable vector field $h$ supported in $\{|x|>R\}$ we denote by $v$ the Lax-Milgram
solution of $-\nabla\cdot a^*\nabla v=\nabla\cdot h$, so that we have the identity
$\int h\cdot\nabla u=\int g\cdot\nabla v$, which by the support condition on $g$ implies
\begin{align*}
\int h\cdot\nabla u\le\big(\int_{|x|<r}|g|^2\big)^\frac{1}{2}\big(\int_{|x|<r}|\nabla v|^2\big)^\frac{1}{2}.
\end{align*}
Since by the support condition on $h$, $v$ is $a^*$-harmonic in $\{|x|<R\}$, we may apply (\ref{o.30})
(recall $r\ge\underline{r_*}(0)\ge r_*(0)$)
to the effect of $\fint_{|x|<r}|\nabla v|^2\lesssim\fint_{|x|<R}|\nabla v|^2$.
Combined with the energy estimate, this yields
\begin{align*}
\int_{|x|<r}|\nabla v|^2\lesssim(\frac{r}{R})^{d}\int|h|^2.
\end{align*} 
Choosing $h$ to be the restriction of $\nabla u$ on $\{|x|>R\}$, we obtain (\ref{EM01}).

\medskip

\step2 Decay property in $L^p$. 

As in Step 1 suppose that $\supp g\subset B_r$ for some $r\ge \underline{r_*}(0)$.
Then we claim for all $R\ge 4r$ 
\begin{align}\label{EM02}
\Big(\frac{1}{R^d}\int_{|x|>R}\big(\fint_{B_*(x)}|\nabla u|^2\big)^\frac{p}{2}dx\Big)^\frac{1}{p}
\lesssim(\frac{r}{R})^d\Big(\frac{1}{r^d}\int_{|x|<4r}\big(\fint_{B_*(x)}|g|^2\big)^\frac{p}{2}dx\Big)^\frac{1}{p}.
\end{align}
By a decomposition of $\{|x|>R\}$ in dyadic annuli, it is enough to show for $R\ge 4r$
\begin{align*}
\sup_{R<|x|<2R}\fint_{B_*(x)}|\nabla u|^2
\lesssim(\frac{r}{R})^{2d}\Big(\fint_{|x|<4r}\big(\fint_{B_*(x)}|g|^2\big)^\frac{p}{2}dx\Big)^\frac{2}{p}.
\end{align*}
We note that for $x$ with $R<|x|<2R$ we have $\underline{r_*}(x)\le \underline{r_*}(0)+\frac{1}{8}|x|\le \frac{R}{2}$, so that
$B_*(x)\subset B_{\frac R2}(x)\subset\{\frac{R}{2}<|y|<4R\}$, which in turn is contained in $\{|y|>r\}$. Hence by our
support assumption on $g$, $u$ is $a$-harmonic so that by (\ref{o.30})
\begin{align*}
\sup_{R<|x|<2R}\fint_{B_*(x)}|\nabla u|^2\lesssim\sup_{R<|x|<2R}\fint_{B_\frac{R}{2}(x)}|\nabla u|^2
\lesssim\fint_{\frac{R}{2}<|y|<4R}|\nabla u|^2.
\end{align*}
According to (\ref{EM01}) in the previous step (applied to three neighboring dyadic annuli) we obtain
\begin{align*}
\fint_{\frac{R}{2}<|y|<4R}|\nabla u|^2\lesssim(\frac{r}{R})^{2d}\fint_{|x|<r}|g|^2.
\end{align*}
Finally by \eqref{e.equiv-rP-2}, c.f.~Step~5 in the proof of Corollary \ref{cor:Lp}, in conjunction with $B_*(x)\cap\{|y|<r\}=\emptyset$ provided $|x|>4r$
(which in turn relies on $\underline{r_*}(x)\le r+\frac{1}{8}|x|$),
we have
\begin{align*}
\fint_{|x|<r}|g|^2\lesssim\fint_{|x|<4r}\fint_{B_*(x)}|g|^2\le\Big(\fint_{|x|<4r}\big(\fint_{B_*(x)}|g|^2\big)^\frac{p}{2}\Big)^\frac{2}{p}.
\end{align*}
%

\medskip

\step3 Under the assumptions of the corollary we claim for $R\ge8\underline{r_*}(0)$
\begin{align}\nonumber
  &{\Big(\int_{|x|> R}\big(\fint_{B_*(x)}|\nabla u|^2\big)^\frac{p}{2}dx\Big)^\frac{1}{p}}
\\\label{EM04}
&\qquad\lesssim\Big(\int_{|x|> R}\big(\fint_{B_*(x)}|g|^2\big)^\frac{p}{2}dx\Big)^\frac{1}{p}
+\Big(\int_{|x|< R}(\frac{|x|}{R})^\gamma\big(\fint_{B_*(x)}|g|^2\big)^\frac{p}{2}dx\Big)^\frac{1}{p}.
\end{align}
W.l.o.g.~we may assume that $R$ is a dyadic multiple of $\underline{r_*}(0)$; also $r$ below runs over dyadic multiple of $\underline{r_*}(0)$.
We decompose $g$ into
\begin{align*}
&g_R:=I(|x|>\frac{R}{4})g,\\
&g_r:=I(\frac{r}{2}<|x|<r)g\;\;\mbox{for}\;\underline{r_*}(0)<r\le\frac{R}{4},\quad g_{\underline{r_*}(0)}:=I(|x|<\underline{r_*}(0))g.
\end{align*}
Let $u_r$ denote the corresponding Lax-Milgram solutions (to $-\nabla \cdot a \nabla u_r=\nabla \cdot g_r$) so that we have by the triangle inequality
\begin{eqnarray}
\Big(\int_{|x|>R}\big(\fint_{B_*(x)}|\nabla u|^2\big)^\frac{p}{2}dx\Big)^\frac{1}{p}
&\le&\Big(\int_{|x|>R}\big(\fint_{B_*(x)}|\nabla u_R|^2\big)^\frac{p}{2}dx\Big)^\frac{1}{p}\label{EM05}\\
&&+\sum_{\underline{r_*}(0)\le r\le\frac{R}{4}}\Big(\int_{|x|>R}\big(\fint_{B_*(x)}|\nabla u_r|^2\big)^\frac{p}{2}dx\Big)^\frac{1}{p}\label{EM06}.
\end{eqnarray}
We start with the RHS term in line (\ref{EM05}):
Replacing the integration over the set $\{|x|>R\}$ by the integration over the whole space and appealing to (\ref{I1}) in Corollary
\ref{cor:Lp} we have
\begin{align*}
\Big(\int_{|x|>R}\big(\fint_{B_*(x)}|\nabla u_R|^2\big)^\frac{p}{2}dx\Big)^\frac{1}{p}
\lesssim \Big(\int\big(\fint_{B_*(x)\cap\{|y|>\frac{R}{4}\}}|g|^2\big)^\frac{p}{2}dx\Big)^\frac{1}{p}.
\end{align*}
Using that $B_*(x)\cap\{|y|>\frac{R}{4}\}=\emptyset$ for $|x|<\frac{R}{8}$, we see that this term is indeed contained in
(both terms of) the RHS of (\ref{EM04}).
We now turn to the terms in line (\ref{EM06}). Since for $r\le\frac{R}{4}$, $g_r$ is supported in $\{|x|<r\}$ we apply (\ref{EM02}) in the previous step
to the couple $(u_r,g_r)$:
\begin{eqnarray}
\Big(\int_{|x|>R}\big(\fint_{B_*(x)}|\nabla u_r|^2\big)^\frac{p}{2}dx\Big)^\frac{1}{p}
&\lesssim& (\frac{r}{R})^d \Big((\frac{R}{r})^d\int_{|x|<4r}\big(\fint_{B_*(x)}|g_r|^2\big)^\frac{p}{2}dx\Big)^\frac{1}{p}\nonumber\\
&\le& (\frac{r}{R})^{d-\frac{1}{p}(d+\gamma)}\Big(\int_{|x|<R}(\frac{|x|}{R})^\gamma\big(\fint_{B_*(x)}|g|^2\big)^\frac{p}{2}dx\Big)^\frac{1}{p}.
\nonumber
\end{eqnarray}
By our assumption on $\gamma$, the exponent is positive, so that the sum over dyadic $r<R$ converges, to the effect that
also the contribution from (\ref{EM06}) is estimated by the (second term on the) RHS of (\ref{EM04}). 

\medskip

\step4 Conclusion. 

We set for abbreviation $U:=\big(\fint_{B_*(x)}|\nabla u|^2\big)^\frac{p}{2}$,
$G:=\big(\fint_{B_*(x)}|g|^2\big)^\frac{p}{2}$, and $r:=8\underline{r_*}(0)$; we fix a $\gamma'\in(\gamma,d(p-1))$. 
By (\ref{EM04}) in the previous step and by (\ref{I1}) in Corollary \ref{cor:Lp} we have
\begin{align}\label{EM09}
\int_{|x|>R}U\lesssim\int_{|x|>R}G+\int_{|x|<R}(\frac{|x|}{R})^{\gamma'}G\;\;\mbox{for}\;R\ge r,\quad
\int U\lesssim\int G.
\end{align}
From the properties (\ref{EM08}) of $\omega$ we infer by division into dyadic annuli
\begin{align*}
\int_{|x|>r}\omega(|x|)U\sim\omega(r)\int_{|x|>r}U+\sum_{R>r}(\omega(R)-\omega(\frac{R}{2}))\int_{|x|>R}U,
\end{align*}
which we also apply with $U$ replaced by $G$, so that by (\ref{EM09}) we get
%
\begin{equation*}
{\int_{|x|>r}\omega(|x|)U\lesssim\int_{|x|>r}\omega(|x|)G}
+\omega(r)\int_{|x|<r}(\frac{|x|}{r})^{\gamma'}G
+\sum_{R>r}(\omega(R)-\omega(\frac{R}{2}))\int_{|x|<R}(\frac{|x|}{R})^{\gamma'}G.
\end{equation*}
Appealing to (\ref{EM08}) in form of 
$\omega(r)I(r>|x|)(\frac{|x|}{r})^{\gamma'}\le\omega(|x|)$ (where we need $\gamma\le\gamma'$)
and in form of $\sum_{R>r}\omega(R)I(R>|x|)(\frac{|x|}{R})^{\gamma'}$ $\lesssim\omega(|x|)$ (where we need $\gamma<\gamma'$),
and recalling $r=8\underline{r_*}(0)$, this turns into
\begin{align*}
\int_{|x|>8\underline{r_*}(0)}\omega(|x|)U\lesssim\int\omega(|x|)G.
\end{align*}
Combining the latter with the last estimate in (\ref{EM09}) and appealing to
$\omega(|x|+\underline{r_*}(0))\lesssim I(|x|>8\underline{r_*}(0))\omega(|x|)+\omega(\underline{r_*}(0))$ 
and $\omega(|x|)+\omega(\underline{r_*}(0))\lesssim\omega(|x|+\underline{r_*}(0))$ (which both
follow from (\ref{EM08})), we obtain (\ref{EM10}) in form of 
$\int\omega(|x|+\underline{r_*}(0))U$ $\lesssim\int\omega(|x|+\underline{r_*}(0))G$.


\section{Optimal stochastic integrability of $r_*$: Proof of deterministic results}


\subsection{Proof of Proposition~\ref{p1}: From modified corrector to corrector}

We split the proof into three steps and start with a reduction argument.

\medskip

\step{1} Reduction.

For notational ease, we replace $\fint_{B_R}|f-\fint_{B_R}f|^2$ by $\inf_{c}\fint_{B_R}|f-c|^2$.
We start with a couple of reductions: We first claim that it is enough to establish
(\ref{gr01}) under the additional condition $R\ge\max\{r_*,r_{**}\}$. Indeed, if $C_0=C_0(d,\lambda,\alpha)$ denotes the constant 
in (\ref{O.1}), then, by definition of $r_{**}$, there exists a constant $C_1=C_1(d,\lambda,\alpha,\nu)$ such that
\begin{align*}
\frac{1}{R^2}\inf_c\fint_{B_R}|(\phi,\sigma)-c|^2\le\frac{1}{2^{d+2}}\frac{1}{C_0}
\quad\mbox{for all}\;R\ge\max\{r_*,C_1r_{**}\}.
\end{align*}
Because of the elementary inequality $\frac{1}{r^2}\fint_{B_r}$ $\le(\frac{R}{r})^{d+2}$ $\frac{1}{R^2}\fint_{B_R}$ this implies
in turn 
\begin{align*}
\frac{1}{R^2}\inf_c\fint_{B_R}|(\phi,\sigma)-c|^2\le\frac{1}{C_0}
\quad\mbox{for all}\;R\ge\frac{1}{2}\max\{r_*,C_1r_{**}\}.
\end{align*}
Since we may take $r_*$ to be the smallest radius with (\ref{O.1}), we obtain\linebreak $r_*$ $\le \frac{1}{2}\max\{r_*,C_1r_{**}\}$ and thus
$r_*$ $\le C_1r_{**}$. This yields both (\ref{gr02}) and the fact that (\ref{gr01}) holds for all $R\ge C_1 r_{**}$,
and thus, at the expense of a worse constant, for all $R\ge r_{**}$.

\medskip

Moreover, using the elementary inequality $\frac{1}{r^2}\fint_{B_r}$ $\le(\frac{R}{r})^{d+2}$ $\frac{1}{R^2}\fint_{B_R}$ again, we find that it suffices to establish (\ref{gr01}) for dyadic radii $R$.

\medskip

We now turn to the last reduction argument. 
By qualitative homogenization (see for instance \cite[Theorem~1]{Gloria-Habibi-16}), we have $$\expec{|\nabla (\phi_{T},\sigma_{T})-\nabla (\phi,\sigma)|^2} \stackrel{T\uparrow \infty}\longrightarrow 0,$$ and thus by stationarity and Poincar\'e's inequality,
\begin{equation*}
  \expec{\frac{1}{R^2}\inf_c\fint_{B_R}|(\phi_T-\phi,\sigma_T-\sigma)-c|^2}\lesssim\expec{\fint_{B_R}|\nabla(\phi_T,\sigma_T)-\nabla(\phi,\sigma)|^2}\stackrel{T\uparrow \infty}\longrightarrow 0.
\end{equation*}
Hence, there exists a sequence $\sqrt{T_k}$ (that is, a subsequence of $(2^m)_{m\in \N}$) such that for all dyadic radii $R$ we have 
\begin{equation*}
\lim\limits_{k\to\infty}\frac{1}{R^2}\inf_c\fint_{B_R}|(\phi_{T_k}-\phi,\sigma_{T_k}-\sigma)-c|^2=0\qquad\text{$\expec{\cdot}$-almost surely}.
\end{equation*}
We may thus conclude that it suffices to establish (\ref{gr01}) in the modified form:  For all 
$T_0$ with $\sqrt T_0$ dyadic we have
\begin{align}\label{gr03}
\frac{1}{R^2}\inf_c\fint_{B_R}|(\phi_{T_0},\sigma_{T_0})-c|^2\lesssim(\frac{r_{**}}{R})^{2\nu}
\quad\mbox{for}\;\max\{r_*,r_{**}\}\le R\le \sqrt{T_0},
\end{align}
where $\lesssim$ stands for $\le C(d,\lambda,\alpha,\nu)$.

\medskip

\step2 Proof of \eqref{gr03} by a Campanato-iteration argument.

\nopagebreak

We shall prove  \eqref{gr03} based on the one-step yet iterable estimate
\begin{align}\nonumber
\lefteqn{\frac{1}{r^2}\Big(\inf_c\fint_{B_r}(\phi_{T_0}-c)^2
+(\frac{r}{R})^d\inf_c\fint_{B_r}|\sigma_{T_0}-c|^2\Big)}\\
&\lesssim(\frac{1}{R^2}+(\frac{R}{r})^d\frac{1}{T_0}\big)\Big(\inf_c\fint_{B_R}(\phi_{T_0}-c)^2
+(\frac{r}{R})^d\inf_c\fint_{B_R}|\sigma_{T_0}-c|^2\Big)\nonumber\\
&+(\frac{R}{r})^{d+2}\fint_{B_R}\frac{1}{T}|(\phi_T,\sigma_T)|^2\quad
\mbox{for}\;\max\{r_*,r_{**}\}\le r\le R\le\sqrt{T_0}\;\mbox{with}\;T=R^2.\label{gr04}
\end{align}
Let us argue how to pass from (\ref{gr04}) to (\ref{gr03}). For some ratio $M=\frac Rr$ to be fixed later we introduce the abbreviation
$E(r)$ $:=\frac{1}{r^2}\inf_c\fint_{B_r}(\phi_{T_0}-c)^2$
$+M^{-d}\frac{1}{r^2}\inf_c\fint_{B_r}|\sigma_{T_0}-c|^2$, so that (\ref{gr04}) turns into
\begin{align*}
E(r)\;\lesssim\;(1+\frac{M^{d+2}r^2}{T_0})E(Mr)+M^{d+2}(\frac{r_{**}}{Mr})^{2\nu}\quad
\mbox{for}\;\max\{r_*,r_{**}\}\le r\le \frac{\sqrt{T_0}}{M},
\end{align*}
where we used (\ref{gr05}) to estimate the last term in (\ref{gr04}). Restricting the range of $r$'s a bit, this simplifies to
\begin{align*}
E(r)\;\lesssim\;E(Mr)+M^{d+2}(\frac{r_{**}}{Mr})^{2\nu}\quad
\mbox{for}\;\max\{r_*,r_{**}\}\le r\le (\frac{1}{M})^{\frac{d}{2}+1}\sqrt{T_0},
\end{align*}
which we multiply with $r^{2\nu}$ and make the constant $C=C(d,\lambda,\alpha,\nu)$ explicit:
\begin{align*}
r^{2\nu}E(r)\;\le C\big(M^{-2\nu}(Mr)^{2\nu}E(Mr)+M^{d+2}(\frac{r_{**}}{M})^{2\nu}\big).
\end{align*}
We now choose $M=M(d,\lambda,\alpha,\nu)$ so large that $CM^{-2\nu}=\frac{1}{2}$, which is possible
because of $\nu>0$, and so obtain
\begin{align*}
\sup_{\max\{r_*,r_{**}\}\le r\le\sqrt{T_0}}r^{2\nu}E(r)\;
\lesssim \sup_{(\frac{1}{M})^{\frac{d}{2}+1}\sqrt{T_0}\le R\le\sqrt{T_0}}R^{2\nu}E(R)
+r_{**}^{2\nu}.
\end{align*}
It is here that we use $T_0<\infty$, so that we only have to take the supremum over a finite range of radii
and may therefore absorb the RHS into the LHS without any a priori assumption of finiteness.
In view of the definition of $E(r)$, which in particular yields using that $R\ge r_{**}$
\begin{align*}
\sup_{(\frac{1}{M})^{\frac{d}{2}+1}\sqrt{T_0}\le R\le\sqrt{T_0}}R^{2\nu}E(R)
\lesssim\sqrt{T_0}^{2\nu}\fint_{B_{\sqrt{T_0}}}\frac{1}{T_0}|(\phi_{T_0},\sigma_{T_0})|^2
\stackrel{(\ref{gr05})}{\lesssim}r_{**}^{2\nu},
\end{align*}
this last estimate turns into (\ref{gr03}).

\medskip

\step3 Proof of \eqref{gr04}

\noindent By the triangle inequality in $L^2$ and $\fint_{B_r}\le(\frac{R}{r})^d\fint_{B_R}$,
it is enough to show
\begin{align*}
\lefteqn{\frac{1}{r^2}\Big(\inf_c\fint_{B_r}(\phi_{T_0}-\phi_T-c)^2
+(\frac{r}{R})^d\inf_c\fint_{B_r}|\sigma_{T_0}-\sigma_T-c|^2\Big)}\nonumber\\
&\lesssim(\frac{1}{R^2}+(\frac{R}{r})^d\frac{1}{T_0}\big)
\Big(\inf_c\fint_{B_R}(\phi_{T_0}-\phi_T-c)^2
+(\frac{r}{R})^d\inf_c\fint_{B_R}|\sigma_{T_0}-\sigma_T-c|^2\Big)\nonumber\\
&+(\frac{R}{r})^d\fint_{B_R}\frac{1}{T}|(\phi_T,\sigma_T)|^2\quad
\mbox{for}\;r_*\le r\le\frac{1}{4}R\le\sqrt{T_0}\;\mbox{with}\;T=R^2.
\end{align*}
(The exponent $d+2$ in the last RHS term of (\ref{gr04}) comes from the control of $\frac1{r^2} \fint_{B_r}\phi_T^2$.)
By Poincar\'e's inequality the latter follows from
\begin{align}\label{gr07}
\lefteqn{\fint_{B_r}|\nabla(\phi_{T_0}-\phi_T)|^2+(\frac{r}{R})^d\fint_{B_r}|\nabla(\sigma_{T_0}-\sigma_T)|^2}\nonumber\\
&\lesssim
(\frac{1}{R^2}+(\frac{R}{r})^d\frac{1}{T_0}\big)
\Big(\inf_c\fint_{B_R}(\phi_{T_0}-\phi_T-c)^2
+(\frac{r}{R})^d\inf_c\fint_{B_R}|\sigma_{T_0}-\sigma_T-c|^2\Big)\nonumber\\
&+(\frac{R}{r})^d\fint_{B_R}\frac{1}{T}|(\phi_T,\sigma_T)|^2,
\end{align}
always in the same range. We note that we may split (\ref{gr07}) into
\begin{align}\label{gr08}
\lefteqn{\fint_{B_r}|\nabla(\phi_{T_0}-\phi_T)|^2}\nonumber\\
&\lesssim(\frac{1}{R^2}+(\frac{R}{r})^d\frac{1}{T_0}\big)\inf_c\fint_{B_R}(\phi_{T_0}-\phi_T-c)^2
+(\frac{R}{r})^d\fint_{B_R}\frac{1}{T}\phi_T^2
\end{align}
and
\begin{align}\label{gr09}
\fint_{B_r}|\nabla(\sigma_{T_0}-\sigma_T)|^2
&\lesssim(\frac{1}{R^2}+(\frac{R}{r})^d\frac{1}{T_0}\big)
\inf_c\fint_{B_\frac{R}{2}}|\sigma_{T_0}-\sigma_T-c|^2\nonumber\\
&+(\frac{R}{r})^d\big(\fint_{B_\frac{R}{2}}\frac{1}{T}|\sigma_T|^2+\fint_{B_\frac{R}{2}}|\nabla(\phi_{T_0}-\phi_T)|^2\big).
\end{align}
Indeed, we first use (\ref{gr08}) for $r=\frac{R}{2}$, insert the result into (\ref{gr09}) and multiply with
$(\frac{r}{R})^d$; we then add (\ref{gr08}) to it to obtain (\ref{gr07}).

\medskip

We claim that both (\ref{gr08}) and (\ref{gr09}) may be inferred from the following a priori estimate: 
Suppose the functions $u$, $f$, and the vector field $g$ are related by
\begin{align}\label{gr10}
\frac{1}{T_0}u-\nabla\cdot a\nabla u=f+\nabla\times g,
\end{align}
then we have
\begin{align}\label{gr11}
\fint_{B_r}|\nabla u|^2&\lesssim(\frac{1}{R^2}+(\frac{R}{r})^d\frac{1}{T_0}\big)
\inf_c\fint_{B_R}(u-c)^2\nonumber\\
&+(\frac{R}{r})^d\big(R^2\fint_{B_R}f^2+\fint_{B_R}|g|^2\big)\quad\mbox{for}\;r_*\le r\le\frac{R}{2}.
\end{align}
In order to obtain (\ref{gr08}), we apply this to $u=\phi_{T_0}-\phi_T$, which by (\ref{o14}) satisfies
(\ref{gr10}) with $g=0$ and $f=(\frac{1}{T}-\frac{1}{T_0})\phi_T$, recalling that $T=R^2$. 
For (\ref{gr09}), we apply (\ref{gr11}) to $u=\sigma_{T_0}-\sigma_T$ and
the identity matrix playing the role of $a$ (so that trivially $r_*=0$) and $R$ replaced by $\frac{R}{2}$. 
Indeed, by (\ref{o17}) we have (\ref{gr10}) with $g=a\nabla(\phi_{T_0}-\phi_T)$
and $f=(\frac{1}{T}-\frac{1}{T_0})\sigma_T$.

\medskip

\step{4} Argument for the a priori estimate (\ref{gr11}). 

\noindent
We start by applying Caccioppoli estimate \eqref{e.add-Caccio} to (\ref{gr10}).
Because of the presence of the massive term, the Caccioppoli estimate is slightly more subtle:
We test (\ref{gr10}) with $\eta^2(u-\bar u)$, where $\eta$ is a smooth cut-off for $B_\frac{R}{2}$
in $B_R$ and where $\bar u:=\frac{\int\eta^2u}{\int\eta^2}$ is the corresponding spatial average, to the effect of
$\int\eta^2(u-\bar u) u$ $=\int\eta^2(u-\bar u)^2$. Hence we obtain
\begin{eqnarray*}
\fint_{B_R}\big(\frac{1}{T_0}\eta^2(u-\bar u)^2+|\nabla(\eta(u-\bar u))|^2\big) 
&\lesssim&\fint_{B_R}\big(\frac{1}{R^2}(u-\bar u)^2+R^2f^2+|g|^2\big)
\\
&\lesssim &\inf_c\fint_{B_R}\big(\frac{1}{R^2}(u-c)^2+R^2f^2+|g|^2\big).
\end{eqnarray*}
From testing (\ref{gr10}) with $\eta^2$ and rewriting the elliptic term like
$\nabla\eta^2\cdot a\nabla u$ $=2\nabla\eta\cdot a\nabla(\eta(u-\bar u))$ $-2(u-\bar u)\nabla\eta\cdot a\nabla\eta$
we obtain
\begin{eqnarray*}
R^2(\frac{1}{T_0}\bar u)^2
&\lesssim& \fint_{B_R}\big(\frac{1}{R^2}(u-\bar u)^2+R^2f^2+|g|^2+|\nabla(\eta(u-\bar u))|^2\big)
\\
&\lesssim& \inf_c\fint_{B_R}\big(\frac{1}{R^2}(u-c)^2+R^2f^2+|g|^2+|\nabla(\eta(u-\bar u))|^2\big).
\end{eqnarray*}
The combination of these two estimates yields
\begin{align}\label{gr13}
{R^2(\frac{1}{T_0}\bar u)^2+\fint_{B_\frac{R}{2}}\big(\frac{1}{T_0}(u-\bar u)^2+|\nabla u|^2\big)}\,\lesssim\inf_c\fint_{B_R}\big(\frac{1}{R^2}(u-c)^2+R^2f^2+|g|^2\big).
\end{align}
We now split $u$ on $B_\frac{R}{2}$ into two 
functions $v$ and $w$ defined through the auxiliary boundary value problems
\begin{align*}
\begin{array}{rl}
-\nabla\cdot a\nabla v=-\frac{1}{T_0}\bar u\;\quad\mbox{in}\;B_\frac{R}{2},&
v=u\;\mbox{on}\;\partial B_\frac{R}{2},\\
-\nabla\cdot a\nabla w=-\frac{1}{T_0}(u-\bar u)+f+\nabla\times g\;\quad\mbox{in}\;B_\frac{R}{2},&
w=0\;\mbox{on}\;\partial B_\frac{R}{2}.
\end{array}
\end{align*}
By the energy estimate, combined with Poincar\'e's estimate with vanishing boundary conditions, we have
\begin{align}
\fint_{B_\frac{R}{2}}|\nabla v|^2&\lesssim R^2(\frac{1}{T_0}\bar u)^2+\fint_{B_\frac{R}{2}}|\nabla u|^2,\label{gr14}\\
\fint_{B_\frac{R}{2}}|\nabla w|^2&\lesssim\fint_{B_\frac{R}{2}}\big(\frac{R^2}{T_0^2}(u-\bar u)^2+R^2f^2+|g|^2\big).\label{gr15}
\end{align}
In order to apply the Schauder theory from Corollary \ref{cor:Lip} to $v$, we note that its RHS 
may be rewritten as $-\nabla\cdot a\nabla v$ $=\nabla\cdot h$ with $h:=\frac{1}{T_0}\bar u\frac{x}{d}$.
Because of
\begin{equation*}
\sup_{1\le\rho\le\frac{R}{2}}(\frac{R}{\rho})^{2\alpha}\inf_{\xi\in\mathbb{R}^d}\fint_{B_\rho}|h-\xi|^2
\stackrel{\alpha\le 1}{\lessim} R^2(\frac{1}{T_0}\bar u)^2,
\end{equation*}
we obtain from (\ref{eq:SE:2}) that for $r_*\le r\le \frac{R}{2}$
\begin{align*}
\fint_{B_r}|\nabla v|^2\lesssim R^2(\frac{1}{T_0}\bar u)^2+\fint_{B_\frac{R}{2}}|\nabla v|^2.
\end{align*}
Hence (\ref{gr14}) \& (\ref{gr15}) turn into
\begin{align*}
\fint_{B_r}|\nabla v|^2&\lesssim R^2(\frac{1}{T_0}\bar u)^2+\fint_{B_\frac{R}{2}}|\nabla u|^2,\\
\fint_{B_r}|\nabla w|^2&\lesssim(\frac{R}{r})^d\fint_{B_\frac{R}{2}}\big(\frac{R^2}{T_0^2}(u-\bar u)^2+R^2f^2+|g|^2\big).
\end{align*}
Inserting (\ref{gr13}) into these estimates, we obtain
\begin{align*}
\fint_{B_r}|\nabla v|^2&\lesssim\inf_c\fint_{B_R}\big(\frac{1}{R^2}(u-c)^2+R^2f^2+|g|^2\big),\\
\fint_{B_r}|\nabla w|^2&\lesssim(\frac{R}{r})^d
\inf_c\fint_{B_R}\big(\frac{1}{T_0}(u-c)^2+R^2f^2+|g|^2\big).
\end{align*}
By the triangle inequality in $L^2$, this yields (\ref{gr11}).


\subsection{Proof of Proposition~\ref{p1bis}: Localization of averages of the modified corrector}

The main building block is the following localized energy estimate.
Suppose the function $u$ is of the class $\sup_{y}\int_{B_1(y)}(u^2+|\nabla u|^2)<\infty$ and 
satisfies $\frac{1}{T}u-\nabla\cdot a\nabla u=\frac{1}{T}f+\nabla\cdot h$ for some scalar field $f$
and some vector field $h$. Then we have
\begin{align}\label{cw16}
\int\omega_T(\frac{1}{T}u^2+|\nabla u|^2)\lesssim\int\omega_T(\frac{1}{T}f^2+|h|^2).
\end{align}
Indeed, this follows from testing the equation with $\eta^2u$ and arguing
like in case of the Caccioppoli estimate, cf.~Appendix~\ref{append:Cacc}, that
$\int(\frac{1}{T}(\eta u)^2+|\nabla(\eta u)|^2)$ $\lesssim\int(\frac{1}{T}f^2+\eta^2|h|^2+|\nabla\eta|^2u^2)$.
We then use this for $\eta=\sqrt{T}^\frac{d}{2}\omega_\frac{T}{2}$ and note that $\eta^2\sim\omega_T$ and
$|\nabla\eta|^2\ll\frac{1}{T}\omega_T$, 
provided the constant $C=C(d,\lambda)$ in (\ref{cw15}) is chosen large enough.

\medskip

We now seek to apply (\ref{cw16}) to $u=\phi_T-\phi_t$, which by (\ref{o14}) satisfies
$\frac{1}{T}u-\nabla\cdot a\nabla u=(\frac{1}{t}-\frac{1}{T})\phi_t$, but want to
bring $\frac{1}{t}\phi_t$ in divergence-form.
We denote the convolution of a function $f$ with the (centered) Gaussian 
$G_t(z)=(\frac{1}{\sqrt{2\pi t}})^d\exp(-\frac{|z|^2}{2t})$ of variance $t$ by $f_{*t}$. Since then $\partial_t f_{*t}$
$=\frac{1}{2}\triangle f_{*t}$
we have $\frac{1}{t}\phi_t$ 
$=\frac{1}{t}(\phi_t)_{*t}-\frac{1}{2t}\int_0^t\triangle(\phi_t)_{*\tau}d\tau$, so that
using (\ref{o14}) and (\ref{o16}) we obtain
\begin{align*}
(\frac{1}{t}-\frac{1}{T})\phi_t=\nabla\cdot h\quad\mbox{with}\quad
h:=(1-\frac{t}{T})\big((q_t-\langle q_t\rangle)_{*t}
-\frac{1}{2t}\int_0^t(\nabla\phi_t)_{*\tau}d\tau\big)\nonumber.
\end{align*}
We now may apply (\ref{cw16}) to $u=\phi_T-\phi_t$ and
the above $h$  
so that by the triangle inequality in $L^2(\omega_T)$ and by the definition of $g_t$, cf. \eqref{cw01},
\begin{align}\nonumber
\lefteqn{\int\omega_T(\frac{1}{T}(\phi_T-\phi_t)^2+|\nabla(\phi_T-\phi_t)|^2)}\\
&\lesssim\int\omega_T|(q_t-\langle q_t\rangle)_{*t}|^2
+\Big(\frac{1}{t}\int_0^t\big(\int\omega_T|(\nabla\phi_t)_{*\tau}|^2\big)^\frac{1}{2}d\tau\Big)^2\nonumber\\
&\stackrel{(\ref{cw01})}{\lesssim}
\frac{1}{t}\int\omega_T|(g_t)_{*t}|^2+t\int\omega_T|(\triangle g_t)_{*t}|^2
+\Big(\frac{1}{t}\int_0^t\big(\int\omega_T|(\nabla\phi_t)_{*\tau}|^2\big)^\frac{1}{2}d\tau\Big)^2.\label{cw17}
\end{align}

\medskip

In order to conclude we appeal to a couple of properties of the convolution operation.
We note that because of Jensen's inequality in form of $(f_{*t})^2\le(f^2)_{*t}$ and the
dominance of Gaussians by exponentials in form of $(\omega_T)_{*t}\lesssim\omega_T$ for $t\le T$ (up to increasing the constant $C$ in \eqref{cw15}
which we implicitly assume without changing notation) we have
\begin{align*}
\int\omega_T(f_{*t})^2\lesssim\int\omega_Tf^2\quad\mbox{for}\;t\le T.
\end{align*}
Furthermore, since for our Gaussian we have $\nabla G_t(z)$ $=-\frac{z}{t}G_t(z)$ and
$\triangle G_t=(\frac{|z|^2}{t^2}-\frac d{t})G_t$, and  thus
$|\nabla G_t|\lesssim\frac{1}{\sqrt{t}}G_\frac{t}{2}$ and $|\triangle G_t|\lesssim\frac{1}{t} G_\frac{t}{2}$, we obtain
\begin{align*}
|\nabla f_{*t}|\lesssim\frac{1}{\sqrt{t}} |f|_{*\frac{t}{2}}, \quad |\triangle  f_{*t}|\lesssim\frac{1}{t} |f|_{*\frac{t}{2}}.
\end{align*}
Equipped with these auxiliary statements, we see that (\ref{cw17}) turns into
\begin{align}\label{cw18}
\int\omega_T(\frac{1}{T}(\phi_T-\phi_t)^2+|\nabla(\phi_T-\phi_t)|^2)
\lesssim
\int\omega_T\frac{1}{t}(|g_t|^2+\phi_t^2).
\end{align}

\medskip

We now turn to the $\sigma$-part. Note that by an application of the differential operator $\sqrt T\nabla\times$ to (\ref{cw01}) we recover (\ref{o17}), and thus
\begin{align}\label{cw61}
\sigma_t=\sqrt{t}\nabla\times g_t.
\end{align}
Hence by (\ref{o16}) and once more by (\ref{o17}) and (\ref{cw01}) we get
\begin{align*}
\frac{1}{T}(\sigma_T-\sigma_t)-\triangle(\sigma_T-\sigma_t)=(1-\frac{t}{T})\nabla\times(\frac{1}{\sqrt{t}}g_t)
+\nabla\times a\nabla(\phi_T-\phi_t),
\end{align*}
so that by (\ref{cw16}) we have in particular
\begin{align*}
\int\omega_T\frac{1}{T}|\sigma_T-\sigma_t|^2\lesssim
\int\omega_T(\frac{1}{t}|g_t|^2+|\nabla(\phi_T-\phi_t)|^2).
\end{align*}
The combination of this with (\ref{cw18}) yields
\begin{align*}
\int\omega_T\frac{1}{T}|(\phi_T,\sigma_T)-(\phi_t,\sigma_t)|^2\lesssim
\int\omega_T\frac{1}{t}|(\phi_t,g_t)|^2.
\end{align*}
By the triangle inequality in $L^2$, this clearly implies (\ref{cw02}) in conjunction with
$\int\omega_T\frac{1}{t}|\sigma_t|^2$ $\lesssim\int\omega_T|\nabla g_t|^2$, which follows from (\ref{cw61}).


\subsection{Proof of Proposition~\ref{p1ter}: Locality of the modified corrector}

We start by noting that (\ref{cw20}) is an easy consequence of (\ref{cw16}) (with $T$ replaced by $t$):
We first apply it to $u=\phi_t$ and thus $h=ae$ and $f=0$, see (\ref{o14}), to the effect of
$\int\omega_t(\frac{1}{t}\phi_t^2+|\nabla\phi_t|^2)$ $\lesssim 1$; for later reference we note
\begin{align}\label{cw21}
\int\omega_t|e+\nabla\phi_t|^2\lesssim 1.
\end{align}
In view of (\ref{o16}) and
stationarity in form of $|\langle q_t\rangle|^2$ $=\langle|q_t|^2\rangle$
$\le\langle\int\omega_t|q_t|^2\rangle$ this yields in particular 
$\int\omega_t|q_t-\langle q_t\rangle|^2$ $\lesssim 1$. We then apply (\ref{cw16}) (with $\Id$ playing the role
of $a$) to
$u=g_t$ and thus $f=\sqrt{t}(q_t-\langle q_t\rangle)$ and $h=0$, see (\ref{cw01}), which yields 
$\int\omega_t(\frac{1}{t}|g_t|^2+|\nabla g_t|^2)$ $\lesssim \int\omega_t|q_t-\langle q_t\rangle|^2$.
The combination gives (\ref{cw20}).

\medskip

We now turn to (\ref{cw04}) and write 
for abbreviation $(\phi_t',q_t',q_t')=(\phi_t,q_t,q_t)(a')$ in order to reserve 
$(\phi_t,q_t,q_t)$ for $(\phi_t,q_t,q_t)(a)$. We first apply (\ref{cw16}) (always with $T$ replaced by $t$) to $u=\phi_t'-\phi_t$
so that $h=(a'-a)(e+\nabla\phi_t)$, see (\ref{o14}), and with $a'$ playing the role of $a$, to obtain
\begin{align*}
\int\omega_t(\frac{1}{t}(\phi_t'-\phi_t)^2+|\nabla(\phi_t'-\phi_t)|^2)
\lesssim\int\omega_t|(a'-a)(e+\nabla\phi_t)|^2.
\end{align*}
We then apply (\ref{cw16}) to $u=g_t'-g_t$ and $f=\sqrt{t}(q_t'-q_t)$, see (\ref{cw01}), to the effect of
\begin{align*}
\int\omega_t(\frac{1}{t}|g_t'-g_t|^2+|\nabla(g_t'-g_t)|^2)\lesssim
\int\omega_t|q_t'-q_t|^2.
\end{align*}
In view of (\ref{o16}) and the triangle inequality in $L^2(\omega_t)$ we obviously have
\begin{align*}
\int\omega_t|q_t'-q_t|^2\lesssim\int\omega_t|\nabla(\phi_t'-\phi_t)|^2+\int\omega_t|(a'-a)(e+\nabla\phi_t)|^2.
\end{align*}
These three estimates combine to
\begin{align*}
\int\omega_t(\frac{1}{t}(\phi_t'-\phi_t)^2+\frac{1}{t}|g_t'-g_t|+|\nabla(g_t'-g_t)|^2)\lesssim
\int\omega_t|(a'-a)(e+\nabla\phi_t)|^2.
\end{align*}
Using our assumption that $a'-a$ vanishes on $B_R$, appealing to the elementary estimate
$\omega_t$ $\lesssim\exp(-\frac{|x|}{2C\sqrt{t}})\omega_\frac{t}{2}$, see (\ref{cw15}),
and using that $\int\omega_\frac{t}{2}|e+\nabla\phi_t|^2\lesssim 1$ (which follows from (\ref{cw21}) since by shift-covariance 
the latter yields $\sup_{y}\int_{B_{\sqrt{t}}(y)}|e+\nabla\phi_t|^2$ $\lesssim 1$), we have
\begin{align*}
\int\omega_t|(a'-a)(e+\nabla\phi_t)|^2\lesssim \exp(-\frac{R}{2C\sqrt{t}}).
\end{align*}
By the triangle inequality in $L^2(\omega_t)$ and by (\ref{cw20}), the two last estimates yield (\ref{cw04}).


\subsection{Proof of Lemma~\ref{L:cw}: Control by averages}

We first argue that
\begin{align}\label{cw13}
\langle\phi_T^2\rangle\lesssim\int_0^T\langle|(\nabla\phi_T)_{*t}|^2
+|(q_T-\langle q_T\rangle)_{*t}|^2\rangle dt.
\end{align}
By definition of $f_{*t}$ as the convolution of $f$ with the Gaussian of variance $t$ we have
$\partial_t f_{*t}$ $=\frac{1}{2}\triangle f_{*t}$ and thus 
$\partial_t (f_{*t})^2$ $=-|\nabla f_{*t}|^2+\nabla\cdot(f_{*t}\nabla f_{*t})$. Applied to the stationary
$f=\phi_T$ we thus obtain
\begin{align}\label{cw10}
\frac{d}{dt}\langle((\phi_T)_{*t})^2\rangle=-\langle|(\nabla\phi_T)_{*t}|^2\rangle.
\end{align}
From (\ref{o14}) and (\ref{o16}) in form of $\phi_T$ $=T\nabla\cdot(q_T-\langle q_T\rangle)$
and the semi-group property of convolution with Gaussians in form of $(\cdot)_{*t}$ $=((\cdot)_{*\frac{t}{2}})_{*\frac{t}{2}}$
we obtain $(\phi_T)_{*t}$ $=T(\nabla G_\frac{t}{2})*(q_T-\langle q_T\rangle)_{*\frac{t}{2}}$, where $G_\frac{t}{2}$
denotes the Gaussian of variance $\frac{t}{2}$, so that by Jensen's inequality
\begin{align}\label{cw11}
\langle((\phi_T)_{*t})^2\rangle\le 
T^2{\textstyle(\int|\nabla G_\frac{t}{2}|)}^2\langle|(q_T-\langle q_T\rangle)_{*\frac{t}{2}}|^2\rangle
\lesssim\frac{T^2}{t}\langle|(q_T-\langle q_T\rangle)_{*\frac{t}{2}}|^2\rangle.
\end{align}
Appealing to the elementary inequality for the function $[0,T]\ni t\mapsto\langle((\phi_T)_{*t})^2\rangle$
\begin{align*}
\langle\phi_T^2\rangle\le\int_0^T|\frac{d}{dt}\langle((\phi_T)_{*t})^2\rangle|dt
+\frac{1}{T}\int_\frac{T}{2}^T\langle((\phi_T)_{*t})^2\rangle dt,
\end{align*}
into which we insert (\ref{cw10}) \& (\ref{cw11}), we obtain (\ref{cw13}).

\medskip

We now argue that
\begin{align}\label{cw14}
\langle|g_T|^2+T|\nabla g_T|^2\rangle\lesssim\int_0^T\langle|(q_T-\langle q_T\rangle)_{*t}|^2\rangle dt.
\end{align}
Indeed, from $\partial_t f_{*t}$ $=\frac{1}{2}\triangle f_{*t}$
it is easy to check that $u=\int_0^\infty\exp(-\frac{t}{2T})f_{*t}dt$ provides the solution of
$\frac{1}{T}u-\triangle u=f$, so that from (\ref{cw01}) we obtain
\begin{align*}
g_T=\frac{1}{\sqrt{T}}\int_0^\infty\exp(-\frac{t}{2T})(q_T-\langle q_T\rangle)_{*t} dt.
\end{align*}
Testing (\ref{cw01}) with $g_T$, using the stationarity of $g_T$ and the above representation, we obtain
\begin{align*}
\langle\frac{1}{T}|g_T|^2+|\nabla g_T|^2\rangle=\frac{1}{T}\int_0^\infty\exp(-\frac{t}{2T})
\langle|(q_T-\langle q_T\rangle)_{*\frac{t}{2}}|^2\rangle dt,
\end{align*}
where on the RHS we used the semi-group property and symmetry of $(\cdot)_{*t}$ in form of
$\langle(q_T-\langle q_T\rangle)_{*t}\cdot(q_T-\langle q_T\rangle)\rangle$ 
$=\langle|(q_T-\langle q_T\rangle)_{*\frac{t}{2}}|^2\rangle$. After the change of variables $\frac{t}{2}=t'$,
splitting the integral into $\int_0^T dt'+\int_T^\infty dt'$ and using that by Jensen's inequality
$\langle|(q_T-\langle q_T\rangle)_{*t'}|^2\rangle$ $\le\langle|(q_T-\langle q_T\rangle)_{*t''}|^2\rangle$ for $t'\ge t''$
we obtain (\ref{cw14}).


\subsection{Proof of Lemma~\ref{lem:sensitivity}: Deterministic sensitivity estimate}
We split the proof into five steps. In Step~1 we reformulate the
carr\'e du champ via partition norms, which enables us to argue by duality. 
In Step 2 we provide a deterministic sensitivity estimate
for $(\nabla\phi_T,q_T)$. It relies on the
hole-filling argument provided in Step~5. In Step~3 we treat the case of the
functional derivative and in Step~4 the case of the oscillation.

\medskip 

\step 1 Reformulation by duality.

Set $\partn:=\{x+Q_0\,:\,x\in\Z^d\}$, where $Q_{0}=[-\frac{1}{2},\frac12)^d$ denotes the unit cube.
Let $F=F(a)$ and $\ell\ge 0$.
On the one hand we argue that for the functional derivative we have the implication 
\begin{equation}\label{e.diff-F-fun}
\sup_{a'\ne a} \frac{|F(a')-F(a)|}{\|a-a'\|_{\ell+1,*}} \leq 1
\,\quad \, \Rightarrow\qquad   (\ell+1)^{-d} \sup_{a}    \int  |\partial^{\fun}_{x,\ell+1} F|^2 dx\lesssim 1,
\end{equation}
where $\|a-a'\|_{\ell+1,*}^2:=\sum_{Q\in (\ell+1)\partn} \sup_{x \in Q} |a-a'|^2$.
On the other hand, for the oscillation, we have the corresponding implication
%
\begin{equation}\label{e.diff-F-osc}
\sup_{z\in \R^d} \sup_{a}\sum_{Q\in z+2(\ell+1)\partn} |\partial^{\osc}_{Q} F|^2\,\le\, 1
\qquad \Rightarrow\qquad  (\ell+1)^{-d} \sup_{a} \int |\partial^{\osc}_{x,\ell+1} F|^2dx \lesssim1,
\end{equation}
where in line with Definition~\ref{def:sLSI}
\begin{equation*}
|\partial^\osc_{Q} F(a)|:=\sup\{F(a')-F(a'')\,:\,a'=a''= a \text{ in } \R^d\setminus Q\,\}.
\end{equation*}

\medskip

We start with the proof of \eqref{e.diff-F-fun}. By Definition~\ref{def:sLSI}, we obviously have for $x\in Q$
%
$$
|\partial^{\fun}_{x,\ell+1}F|\leq\int_{B_{\ell+1}(Q)}|\frac{\partial F}{\partial a(z)}|\,dz,
$$
where $B_{\ell+1}(Q):=\{x\,:\,\dist(x,Q)<\ell+1\}$.
Combined with the additivity of the functional derivative with respect to sets, this yields
\begin{eqnarray*}
{\int \big|\partial^{\fun}_{x,\ell+1}F\big|^2dx}&=& \sum_{Q\in(\ell+1)\partn}\int_{Q}  \big|\partial^{\fun}_{x,\ell+1}F\big|^2dx
\\
&\leq & \sum_{Q\in(\ell+1)\partn} (\ell+1)^d\big(\int_{B_{\ell+1}(Q)}|\frac{\partial F}{\partial a(z)}|\,dz\big)^2\\
&\le& 3^{2d}(\ell+1)^d  \sum_{Q\in(\ell+1)\partn}\big(\int_{Q}|\frac{\partial F}{\partial a(z)}|\,dz\big)^2.
\end{eqnarray*}
Since $(\ell+1)\partn$ is a partition and the norm $\|\cdot \|_{\ell+1,*}$ is adapted to this partition, we have by duality
\begin{equation*}
\Big(\sum_{Q\in(\ell+1)\partn}\big(\int_{Q}|\frac{\partial F}{\partial a(z)}|\,dz\big)^2\Big)^\frac12 \, =\, \sup_{\|\delta a\|_{\ell+1,*}=1}\int \frac{\partial F}{\partial a(x)}\cdot \delta a(x)\,dx,
\end{equation*}
and the desired estimate \eqref{e.diff-F-fun} follows by bounding the functional derivative by the Lip\-schitz norm.

\medskip

We turn now to  the proof of \eqref{e.diff-F-osc}.
For all $x\in\mathcal\R^d$, by Definition~\ref{def:sLSI},
\begin{eqnarray*}
|\partial^\osc_{x,\ell+1} F(a)|=\sup\{F(a')-F(a'')\,:\,a'=a''=a\text{ in }\R^d\setminus B_{\ell+1}(x)\,\}\leq|\partial_{x+2(\ell+1)Q_0}^\osc F(a)|,
\end{eqnarray*}
so that
\begin{eqnarray*}
  \int|\partial^\osc_{x,\ell+1} F|^2\,dx&\leq& (2(\ell+1))^d\sum_{z\in 2(\ell+1)\Z^d}\fint_{z+2(\ell+1)Q_0}|\partial_{x+2(\ell+1)Q_0}^\osc F|^2\,dx\\
&=&(2(\ell+1))^d\fint_{2(\ell+1)Q_0}\sum_{Q\in z+2(\ell+1)\partn}|\partial_{Q}^\osc F|^2\,dz,
\end{eqnarray*}
and  \eqref{e.diff-F-osc} follows.

\medskip

\step 2 Deterministic sensitivity estimate using hole-filling.

Fix two $\lambda$-uniformly elliptic coefficient fields $a,a'$ and set for abbreviation
\begin{equation*}
\delta F:=(\nabla\phi_T(a'),q_T(a'))-(\nabla\phi_T(a),q_T(a)).
\end{equation*}
Then there exists an exponent $\e=\e(d,\lambda)>0$ (coming from hole-filling) such that
\begin{equation}\label{e.dSomegaT}
\int\omega_T|\delta F|^2\lesssim \|a-a'\|_{\ell+1,*}^2 (\frac{\ell+1}{\sqrt{T}}\wedge 1)^{\e d}.
\end{equation}
For the argument set $\delta a:=a'-a$, $\delta\phi_T:=\phi_T(a')-\phi_T(a)$ and $\delta q_T:=q_T(a')-q_T(a)$ and note that
\begin{eqnarray}
\label{aao14}
(\frac{1}{T}-\nabla\cdot a \nabla)\delta\phi_T&=&\nabla\cdot\delta a(\nabla\phi_T(a')+e),\\
\label{aao16}
\delta q_T&=&\delta a(\nabla\phi_T(a')+e)+a\nabla\delta\phi_T.
\end{eqnarray}
The energy estimate~\eqref{cw16} applied to \eqref{aao14} yields
\begin{equation*}
\int \omega_T(\frac{1}{T}\delta\phi_T^2+|\nabla
\delta\phi_T|^2)\,\lesssim\, \int \omega_T|\delta a(\nabla\phi_T(a')+e)|^2,
\end{equation*}
and thus by \eqref{aao16}
\begin{eqnarray*}
\int \omega_T|\delta F|^2&\lesssim& \int \omega_T|\delta a(\nabla\phi_T(a')+e)|^2\\
&\leq&\sum_{Q\in(\ell+1)\partn}\sup_Q|a'-a|^2\int_Q\omega_T|\nabla\phi_T(a')+e|^2\\
&\leq&\|a'-a\|_{\ell+1,*}^2\sup_{Q\in(\ell+1)\partn}\int_Q\omega_T|\nabla\phi_T(a')+e|^2.
\end{eqnarray*}
Hence for \eqref{e.dSomegaT}, it suffices to prove
\begin{equation}\label{cdecay}
\sup_{Q\in(\ell+1)\partn}\int_Q\omega_T|\nabla\phi_T(a')+e|^2\lesssim \big(\frac{\ell+1}{\sqrt T}\wedge 1\big)^{\e d},
\end{equation}
which we postpone to the last step.

\medskip

\step 3 Proof of \eqref{e.lem:sensitivity} for the functional derivative.  

Set $F_{*t}:=(\nabla\phi_T,q_T)_{*t}$. In view of Step~1, cf.~\eqref{e.diff-F-fun}, it suffices to show that for any pair of coefficient fields $a,a'$ we have
\begin{equation}\label{felix+1}
  |F_{*t}(a')-F_{*t}(a)|^2\lesssim \Big(\frac{\sqrt T}{\sqrt t}\Big)^d (\frac{\ell+1}{\sqrt{T}}\wedge 1)^{\e d}\,\|a'-a\|_{\ell+1,*}^2.
\end{equation}
Using Jensen's inequality for the Gaussian measure, followed by the relation $G_t \lesssim \big(\frac{\sqrt T}{\sqrt t}\big)^d \omega_T$ (since $T \ge t$)
between the Gaussian and exponential weights $G_t$ and $\omega_T$, we obtain with $\delta F$ defined as in Step~2,
\begin{equation*}
|F_{*t}(a')-F_{*t}(a)|^2\,\le\, (|\delta F|^2)_{*t} \,\lesssim\, \big(\frac{\sqrt T}{\sqrt t}\big)^d \int \omega_T |\delta S|^2.
\end{equation*}
The desired estimate \eqref{felix+1}  then follows from \eqref{e.dSomegaT} in  Step~2.
\medskip

\step 4 Proof of \eqref{e.lem:sensitivity} for the oscillation.

\nobreak

Set again $F_{*t}:=(\nabla\phi_T,q_T)_{*t}$. In view of Step~1, cf.~\eqref{e.diff-F-osc}, using 
\begin{equation*}
 | \partial^{\osc}_{Q} F_{*t}(a)|\leq 2\sup\{|F_{*t}(a_Q)-F_{*t}(a)|\,\big|\,a_Q=a\text{ in }\R^d\setminus Q\,\},
\end{equation*}
and by discrete duality, it suffices to show that for any coefficient field $a$, any shift $z\in\R^d$, any family $\{a_Q\}_{Q\in z+(\ell+1)\partn}$ of coefficient fields $a_Q$ with $a=a_Q$ outside of $Q$, and any real sequence $\omega=\{\omega_Q\}_{Q\in z+(\ell+1)\partn}$, we have
\begin{equation}\label{e.lem:sensitivity-4.1}
\Big(\sum_{Q\in z+(\ell+1)\partn} \omega_Q \big(F_{*t}(a_Q)-F_{*t}(a)\big)\Big)^2 \lesssim\, \Big(\frac{\sqrt T}{\sqrt t}\Big)^d
    (\frac{\ell+1}{\sqrt{T}}\wedge 1)^{\e d}\,\sum_{Q\in z+(\ell+1)\partn} \omega_Q^2.
\end{equation}
For notational convenience we replace $z$ by $0$.
For any $Q\in(\ell+1)\partn$ set
\begin{equation*}
  \delta_Q a:=a_Q-a,\qquad \delta_Q\phi_T:=\phi_T(a_Q)-\phi_T(a),\qquad \delta_Q q_T:=q_T(a_Q)-q_T(a),
\end{equation*}
and
\begin{gather*}
  \delta F:=\sum_{Q\in(\ell+1)\partn}\omega_Q(\nabla \delta_Q\phi_T,\delta_Qq_T).
\end{gather*}
Then as in Step~3 we have
\begin{equation*}
  \Big(\sum_{Q\in (\ell+1)\partn} \omega_Q \big(F_{*t}(a_Q)-F_{*t}(a)\big)\Big)^2\leq (|\delta F|^2)_{*t}\,\lesssim\, \big(\frac{\sqrt T}{\sqrt t}\big)^d \int \omega_T |\delta S|^2.
\end{equation*}
Moreover,  $(\delta_Q \phi_T,\delta_Q q_T)$ are decaying solutions of
\begin{eqnarray*}
(\frac{1}{T}-\nabla\cdot a \nabla)\delta_Q\phi_T&=&\nabla\cdot\delta_Q a(\nabla\phi_T(a_Q)+e),\\
\delta_Q q_T&=&\delta_Q a(\nabla\phi_T(a_Q)+e)+a\nabla\delta_Q\phi_T.
\end{eqnarray*}
Multiplying with $\omega_Q$ and summing over $Q$, the energy estimate~\eqref{cw16} yields as in Step~2
\begin{equation*}
  \begin{aligned}
    \int \omega_T|\delta F|^2&\lesssim \int \omega_T\big|\sum_{Q\in (\ell+1)\partn} \omega_Q \delta_Q a(\nabla\phi_T(a_Q)+e)\big|^2.
  \end{aligned}
\end{equation*}
Since $\supp \delta_Q a\subset Q$ this yields
\begin{eqnarray*}
  \int \omega_T|\delta F|^2& \lesssim &\sum_{Q\in(\ell+1)\partn} (\sup_{Q}|a_Q-a|)^2 \omega_Q^2 \int_{Q}\omega_T |\nabla\phi_T(a_Q)+e|^2 \\
&\lesssim&\Big(\sum_{Q\in(\ell+1)\partn}\omega_Q^2\Big) \sup_{Q\in (\ell+1)\partn} \int_{Q}\omega_T |\nabla\phi_T(a_Q)+e|^2,
\end{eqnarray*}
so that also \eqref{e.lem:sensitivity-4.1} follows from \eqref{cdecay}.

\medskip

\step5 Proof of  \eqref{cdecay}

In view of the energy estimate \eqref{cw16}, the LHS of \eqref{cdecay} is $\lesssim 1$, so that it is enough to consider the case of $\ell+1\leq\sqrt T$, for which \eqref{cdecay} assumes the form 
\begin{equation}\label{c.sqrtT}
  \fint_Q|\nabla\phi_T+e|^2\lesssim \big(\frac{\sqrt T}{\ell+1}\big)^{d(1-\e)}.
\end{equation}
This estimate follows from 
%
\begin{equation}\label{o23}
\fint_Q\frac{1}{T}\phi_T^2+|\nabla\phi_T+e|^2
  \,  \stackrel{\eqref{cw16}}{\lesssim} \,  (\frac{\sqrt T}{\ell+1})^d
\end{equation}
by Widman's hole-filling argument, see for instance 
\cite[p.81]{Giaquinta-83}. Since we could not find the variant with a massive term in the literature, we give the argument for \eqref{o23} $\Rightarrow$ \eqref{c.sqrtT} presently.
W.l.o.g.~we may assume that $Q$ is centered at $0$ so that $Q\subset B_R$ for $R\sim \ell+1$. 
In view of \eqref{o23}, it is enough to show
\begin{equation}\nonumber
  \int_{B_R}1+\frac{1}{T}\phi_T^2+|\nabla\phi_T+e|^2
  \lesssim (\frac{R}{\sqrt{T}})^{\e d}\int_{B_{\sqrt{T}}}1+\frac{1}{T}\phi_T^2+|\nabla\phi_T+e|^2
\end{equation}
for all (dyadic) radii $R\le\sqrt{T}$.
With $\e=\frac{\ln\frac{1}{\theta}}{\ln 2}$, this estimate is obtained
via iteration from
\begin{equation}\nonumber
  \int_{B_R}1+\frac{1}{T}\phi_T^2+|\nabla\phi_T+e|^2
  \le\theta \int_{B_{2R}}1+\frac{1}{T}\phi_T^2+|\nabla\phi_T+e|^2
\end{equation}
for some $\theta=\theta(d,\lambda)<1$.
In order to derive the latter with $\theta=\frac{C_0}{C_0+1}$, it is enough to establish
\begin{equation}\label{o62}
  \int_{B_R}1+\frac{1}{T}\phi_T^2+|\nabla\phi_T+e|^2
  \le C_0 \int_{B_{2R}\setminus B_R}1+\frac{1}{T}\phi_T^2+|\nabla\phi_T+e|^2
\end{equation}
for some $C_0=C_0(d,\lambda)<\infty$ -- this is the origin of the name ``hole-filling''. The last estimate
is obtained from the following Caccioppoli estimate: Test (\ref{o14}) with $\eta^2(\phi_T-c)$,
where $\eta$ is a cut-off for $B_R$ in $B_{2R}$ and $c=\fint_{B_{2R}\setminus B_{R}}\phi_T$ to the effect of
\begin{equation}\nonumber
  \int\eta^2\frac{1}{T}\phi_T(\phi_T-c)+\eta^2\nabla\phi_T\cdot a(\nabla\phi_T+e)=-2\int\eta(\phi_T-c)\nabla\eta\cdot
  a(\nabla\phi_T+e),
\end{equation}
which by uniform ellipticity implies 
\begin{equation}\nonumber
  \int\eta^2(\frac{1}{T}\phi_T^2+|\nabla\phi_T+e|^2)\lesssim\int\eta^2\frac{1}{T}|c||\phi_T|+
  (\eta^2+\eta|\phi_T-c||\nabla\eta|)
  |\nabla\phi_T+e|.
\end{equation}
By Young's inequality, this yields
\begin{equation}\nonumber
  \int\eta^2(\frac{1}{T}\phi_T^2+|\nabla\phi_T+e|^2)\lesssim\int\eta^2(\frac{1}{T}|c|^2+1)
  +(\phi_T-c)^2|\nabla\eta|^2,
\end{equation}
and thus by the choice of the cut-off $\eta$
\begin{equation}\nonumber
  \int_{B_R}\frac{1}{T}\phi_T^2+|\nabla\phi_T+e|^2\lesssim R^d(\frac{1}{T}|c|^2+1)
  +\frac{1}{R^2}\int_{B_{2R}\setminus B_R}(\phi_T-c)^2.
\end{equation}
Because of the choice of $c$ we obtain by Jensen's inequality for the first RHS term
and by the Poincar\'e estimate with mean value zero
on the annulus $B_{2R}\setminus B_R$ for the last one
\begin{equation}\nonumber
  \int_{B_R}\frac{1}{T}\phi_T^2+|\nabla\phi_T+e|^2\lesssim R^d+ \frac{1}{T}\int_{B_{2R}\setminus B_R}\phi_T^2
  +\int_{B_{2R}\setminus B_R}|\nabla\phi_T|^2,
\end{equation}
which can be rewritten in the iterable form (\ref{o62}).


\section{Optimal stochastic integrability of $r_*$: Proof of stochastic results}

\subsection{Proof of Theorems~\ref{g},~\ref{gbis}, and~\ref{gbis2}: Optimal stochastic integrability of $r_*$}
We only prove Theorems~\ref{g} and~\ref{gbis}. The argument for Theorem~\ref{gbis2} is simpler (we need not prove properties of $\pi_*$ in that case). Unless stated otherwise we use $\lesssim$ for $\leq$ up to a multiplicative constant only depending on $d,\lambda$ and $\pi$.

\medskip

\step 1 Reformulation.

\nobreak


In this proof, for all $\nu>0$ we define $r_{**}$ (implicitly depending on $\nu$) as the smallest (possibly infinite) random variable that satisfies 
\begin{align}\label{e.theo-1}
\int\omega_T \frac{1}{T}|(\phi_T,\sigma_T)|^2\le(\frac{r_{**}}{R})^{2\nu}\quad\mbox{for all dyadic}\;R\ge r_{**}\;\mbox{and}\;T=R^2,
\end{align}
which, in view of \eqref{cw15}, is a slightly stronger form of \eqref{gr05} in Proposition~\ref{p1}.
Proposition~\ref{p1} then implies that $r_*\le Cr_{**}$ for some constant $C$.
Since the function $\pi_*$ is increasing and satisfies the scaling relation \eqref{min.decay}, the above implies
$\frac1{C'}\pi_*(r_*)\le \pi_*(r_{**})$ for some $C'$ depending only on $C$ and $\pi_*$.
Hence, the claim of the theorem follows if we prove that $r_{**}$ satisfies the moment bound \eqref{g.0} (resp. \eqref{g.0bis}) for a suitable range of $\nu>0$ depending on $\pi$ (which also encodes the dependence on  $\kappa$ for the standard LSI).
In this first step of the proof, we focus on super-level sets of $r_{**}$, and shall show that there exist some dyadic threshold $r_0=r_0(d,\lambda,\pi)$ and some (generic) constants $C(d,\lambda,\pi),\bar C(d,\lambda,\pi)$
such that for all $0<\nu\le \e$ (with $\e=\e(d,\lambda,\pi)$ given in Corollaries~\ref{c.fluct.T-LSI} and~\ref{c.fluct.T}), the random variable $r_{**}$ associated with the exponent $\nu$ (through \eqref{e.theo-1})
satisfies
\begin{equation}\label{e.theo-2}
  \forall \text{ dyadic }r\geq r_0\,:\quad \expec{I(r_{**}\geq r)}\leq \sum_{R\geq r\text{ dyadic}}\expec{I\bigg(\int \omega_{T}(F_t-\expec{F_t}) >\frac{1}{C}\big(\frac{r}{R}\big)^{2\nu}\bigg)},
\end{equation}
where $F_t:=\int\omega_t\big(\frac{1}{t}\phi_t^2+\frac{1}{t}|g_t|^2+|\nabla g_t|^2\big)$ and for $R,T,r$, and $t$ related via
\begin{equation}\label{e.def-t}
  \sqrt T=R\qquad\text{ and }\qquad t=\bar C (\frac{R}{r})^{2\frac{\nu}{\e}}\le \frac T2.
\end{equation}
Here comes the argument. By contraposition of the definition \eqref{e.theo-1} of $r_{**}$  and a union bound, we have for all dyadic $r$
\begin{equation*}
  \expec{I(r_{**}\geq r)}\,\le\,\sum_{R\geq r\text{ dyadic}}\expec{I\bigg(\int \omega_{T} \frac1{T} |(\phi_{T},\sigma_{T})|^2>\big(\frac{r}{R}\big)^{2\nu}\bigg)}.
\end{equation*}
Since for all functions $h\ge 0$ and all $t\leq \frac T2$ we have $\int \omega_T h\lesssim \int \omega_T(h*\omega_t)$,  Proposition~\ref{p1bis} yields
for some $C_1=C_1(d,\lambda)$ and all $t\leq T$,
\begin{equation}\label{e.ant-fin-1}
  \int \omega_{T} \frac1{T} |(\phi_{T},\sigma_{T})|^2\leq C_1\int\omega_T F_t.
\end{equation}
In turn, Corollary~\ref{c.fluct.T-LSI} (resp. Corollary~\ref{c.fluct.T}) yields for all $T,t\geq 1$
\begin{equation}\label{e.ant-fin-2}
\expec{F_t}\leq C_2t^{-\e}
\end{equation}
for some constant $C_2=C_2(d,\lambda,\pi)$.
We choose now a dyadic $r_0=r_0(d,\lambda,\pi)$ so large that
\begin{equation}\label{e.ant-fin-3}
(2C_1C_2)^\frac{1}{\e} \frac{1}{r_0^2} \le \frac12.
\end{equation}
Given $0<\nu\le \e$, and dyadic $R\ge r\ge r_0$ we specify $t$ in line with \eqref{e.def-t} as
\begin{equation}\label{e.ant-fin-4}
t:= (2C_1C_2)^\frac{1}{\e} (\frac{R}{r})^{2\frac{\nu}{\e}},
\end{equation}
and note that our choice \eqref{e.ant-fin-3} of $r_0$ ensures that
$$
t\,\stackrel{\nu\le \e,R\ge r}{\le} \,  (2C_1C_2)^\frac{1}{\e}(\frac{R}{r})^{2} \,=\,  (2C_1C_2)^\frac{1}{\e}T \frac{1}{r^2} \,\stackrel{r\ge r_0}\le\,  T (2C_1C_2)^\frac{1}{\e} \frac{1}{r_0^2}\,\stackrel{\eqref{e.ant-fin-3}}\le \, \frac T2.
$$
The combination of \eqref{e.ant-fin-1} and \eqref{e.ant-fin-2} thus yields for all $0<\nu \le \e$ and all 
dyadic $R\ge r\ge r_0$, $\sqrt{T}=R$, and $t$ given by  \eqref{e.ant-fin-4}
\begin{eqnarray*}
  \int \omega_{T} \frac1{T} |(\phi_{T},\sigma_{T})|^2>(\frac{r}{R})^{2\nu}\ &\stackrel{\eqref{e.ant-fin-1}}{\Rightarrow}&\ \int \omega_T F_t>\frac{1}{C_1}(\frac{r}{R})^{2\nu}\\
  &\stackrel{\eqref{e.ant-fin-2},\eqref{e.ant-fin-4}}\Rightarrow&\ \int \omega_T (F_t-\expec{F_t})>\frac{1}{2C_1}(\frac{r}{R})^{2\nu}.
\end{eqnarray*}
This implies \eqref{e.theo-2} in the regime of parameters \eqref{e.def-t} for some $C,\bar C$ depending only on $C_1$, $C_2$, and $\e$.

\medskip

\step2 Proof for the standard LSI.

\nobreak

\noindent In this case, Proposition~\ref{p1ter} ensures that $F_t$ satisfies the assumptions of Lemma~\ref{lem:concLSI},
which yields the existence of a positive constant $C'=C'(d,\lambda,\kappa)$ such that for all $0<\nu \le \e$, all dyadic $\sqrt{T}=R\ge r\ge r_0$
and $t$ given by  \eqref{e.def-t},
\begin{equation}\label{e.sLSI}
  \expec{I\bigg(\int \omega_{T}(F_t-\expec{F_t}) >\frac{1}{C}\big(\frac{r}{R}\big)^{2\nu}\bigg)}\leq \exp\Big(-\frac1{C'}\big(\frac{r}{R}\big)^{4\nu}  \big(\frac{\sqrt{T}}{\sqrt{t}}\big)^d\Big).
\end{equation}
It remains to choose $\nu$.
By \eqref{e.def-t}, 
\begin{equation}\label{e.est-t}
  \big(\frac{\sqrt{T}}{\sqrt{t}}\big)^d\gtrsim r^d(\frac{R}{r})^{d(1-\frac\nu\e)},\qquad\text{and thus }
  \big(\frac{r}{R}\big)^{4\nu} \big(\frac{\sqrt{T}}{\sqrt{t}}\big)^d\gtrsim r^d(\frac{R}{r})^{d-4\nu(1+\frac d{4\e})},
\end{equation}
so that provided $0<\nu<\frac{d\e}{4(\e+d)}$ the RHS of \eqref{e.sLSI} is summable w.r.t. $R\geq r$ dyadic, 
with the contribution from $R=r$ being dominant.
Fixing such a $\nu=\nu(d,\lambda)$, we conclude with help of \eqref{e.theo-2} that for some positive constant $C''=C''(d,\lambda,\kappa)$ and for all dyadic $r\geq r_0$,
$$
\expec{I(r_{**}\ge r)}\,\le\, C''\exp(-\frac1{C''} r^d).
$$
By summation over dyadic $r\ge r_0$, this yields the desired  (stretched) exponential moment bound on $r_{**}$, and therefore on $r_*$ by Step~1.

\medskip

\step3 Proof for MLSI.

As in the previous step, the combination of Proposition~\ref{p1ter} with Lemma~\ref{lem:conc}~\emph{(v)} yields
for some positive constant $C'=C'(d,\lambda,\pi)$, for all $0<\nu \le \e$, all dyadic $\sqrt{T}=R\ge r\ge r_0$
and $t$ given by  \eqref{e.def-t},
$$
 \expec{I\bigg(\int \omega_{T}(F_t-\expec{F_t}) >\frac{1}{C}\big(\frac{r}{R}\big)^{2\nu}\bigg)}
\,\le\,\exp\Big(-\frac1{C'} \big(\frac{r}{R}\big)^{4\nu}\pi_*\big(\frac{\sqrt{T}}{\sqrt{t}}\big) \Big).
$$
In Step~4 below, we shall argue that for any $0<\tilde \beta<\beta$, and in particular for $\tilde \beta =\frac \beta 2$, there exists $\ell_0=\ell_0(\pi)\gg1$ such that
\begin{equation}\label{as:gamma}
\forall K\geq 1\text{ and }\ell\geq \ell_0\,:\,K^{\tilde \beta} \pi_*(\ell)\,\leq \, \pi_*(K\ell) .
\end{equation}
Combined with the definition \eqref{e.def-t} of $t$ and the first estimate in \eqref{e.est-t}, this implies
\begin{eqnarray*}
 \big(\frac{r}{R}\big)^{4\nu}  \pi_*\big(\frac{\sqrt{T}}{\sqrt{t}}\big)
  &\gtrsim&(\frac{R}{r})^{\tilde\beta(1-\frac{\nu}{\e})-4\nu}  \pi_*(r).
\end{eqnarray*}
Choosing $\nu=\nu(d,\lambda,\pi)$ in the range $0<\nu<\frac{\tilde \beta \e}{4\e+\tilde \beta}$, we continue to argue as in Step~2.

\medskip

\step4 Proof of \eqref{as:gamma}.

Recall that $\frac1{\pi_*(\ell)}=\fint_{B_\ell} \gamma(|x|)dx$.
Using spherical coordinates, we have
\begin{equation*}
 \frac1{\pi_*(\ell)\gamma(\ell)}=d\int_0^1\theta^{d-1}\frac{\gamma(\theta\ell)}{\gamma(\ell)}\,d\theta.
\end{equation*}
By appealing to \eqref{min.decay} and Fatou's lemma we deduce that
\begin{equation*}
  \liminf\limits_{\ell\to\infty}\frac1{\pi_*(\ell)\gamma(\ell)}\geq d\int_0^1\theta^{d-1}\liminf\limits_{\ell\to\infty}\frac{\gamma(\theta\ell )}{\gamma(\ell)}\,d\theta\geq d\int_0^1\theta^{d-1-\beta}\,d\theta=\frac{d}{d-\beta}.
\end{equation*}
Hence, for any $0<\tilde\beta<\beta$ we can find $\ell_0\gg 1$ such that
\begin{equation}\label{ZZ.3}
\frac1{\pi_*(\ell)\gamma(\ell)}\geq \frac{d}{d-\tilde\beta}\qquad\text{for all }\ell\geq \ell_0.
\end{equation}
In addition, we have $(\frac{1}{\pi_*})'(\ell)\,=\,d\int_0^1\theta^d \gamma'(\theta\ell)d\theta=-\frac{d}{\ell} \frac{1}{\pi_*}(\ell)(1-\gamma(\ell)\pi_*(\ell))$, so that for $K\ge 1$
\begin{equation*}
 \log \pi_*(\ell)- \log \pi_*(K\ell)=-\int_\ell^{K\ell}\frac{d}{r}(1-{\gamma(r)}{\pi_*(r)})\,dr.
\end{equation*}
Combined with \eqref{ZZ.3}, this yields for $\ell \geq \ell_0$
\begin{equation*}
\log \pi_*(\ell)- \log \pi_*(K\ell)\leq-\int_\ell^{K\ell}\frac{d}{r}(1-\frac{d-\tilde\beta}{d})\,dr=-\tilde\beta\int_1^K\frac{1}{s}\,ds=\log(K^{-\tilde\beta}),
\end{equation*}
which we rewrite as the desired estimate $K^{\tilde \beta} \pi_*(\ell)\,\le \, \pi_*(K\ell)$.
%


\subsection{Proof of Corollaries~\ref{c.fluct.T-LSI} and~\ref{c.fluct.T}: Control of the expectation}\label{sec:clfuctT}

By Lemma~\ref{L:cw}, it is enough to control the quantity 
\begin{equation*}
\frac{1}{T}\int_0^T\langle|(\nabla\phi_{T})_{*t}|^2+|(q_{T})_{*t}-\langle q_T\rangle|^2\rangle dt,
\end{equation*}
which is the integral of a variance. 
For $t\le T^{1-\e}$, by Jensen's inequality and stationarity, we have
\begin{equation}\label{e.fluct-1}
\frac{1}{T}\int_0^{T^{1-\e}} \langle|(\nabla\phi_{T})_{*t}|^2+|(q_{T})_{*t}-\langle q_T\rangle|^2\rangle dt
\,\lesssim\, T^{-\e} \expec{|\nabla \phi_T|^2+1} \,\lesssim\, T^{-\e}.
\end{equation}
For $t\ge T^{1-\e}$, we appeal MSG (which follows from MLSI) in the form
\begin{equation}\label{e.st1.2nd-mom}
\langle|(\nabla\phi_{T})_{*t}|^2+|(q_{T})_{*t}-\langle q_T\rangle|^2\rangle\,\leq\, \expec{\int_0^\infty \pi(\ell) (\ell+1)^{-d}  \int_{\R^d} \big|\partial^{\fun/\osc}_{x,\ell+1} F_{*t} \big|^2dxd\ell}
\end{equation}
with $F_{*t}:=(\nabla\phi_T,q_T)_{*t}$. (For the standard LSI, replace the integral over $\ell$ by the integrand for $\ell=0$.)
We split the rest of the proof into two steps.
In the first step, we prove the claim for the standard LSI, and conclude with the general case in Step~2.

\medskip

\step1 Proof of \eqref{o8-LSI}.

By Lemma~\ref{lem:sensitivity} (applied with $\ell=0$),
\begin{equation*}
 {\int_{\R^d} \big|\partial^{\fun/\osc}_{x,1}F_{*t}\big|^2dx }\,\lesssim\, \Big(\frac{\sqrt T}{\sqrt t}\Big)^d \sqrt{T}^{-\e d},
\end{equation*}
so that
\begin{equation*}
\frac1T \int_{T^{1-\e}}^T\int_{\R^d} \big|\partial^{\fun/\osc}_{x,1}F_{*t} \big|^2dxdt \,\lesssim\,\sqrt{T}^{-2-\e d+d} \int_{T^{1-\e}}^T \sqrt{t}^{-d}dt\,
\lesssim\, T^{-\e},
\end{equation*}
where in dimension $d=2$, we slightly reduced $\e>0$ to absorb the logarithm.
Combined with \eqref{e.fluct-1} and \eqref{e.st1.2nd-mom}, the desired result follows.

\medskip

\step2 Proof of  \eqref{o8}.

Starting from a general partition $(\ell+1)\partn$ for $\ell\ge 0$, Lemma~\ref{lem:sensitivity} 
yields
\begin{equation*}
(\ell+1)^{-d}{\int_{\R^d} \big|\partial^{\fun/\osc}_{x,\ell+1}F_{*t} \big|^2dx}\,\lesssim\, \Big(\frac{\sqrt T}{\sqrt t}\Big)^d (\frac{\ell+1}{\sqrt{T}}\wedge 1)^{\e d}.
\end{equation*}
In particular, this implies
\begin{eqnarray*}
\lefteqn{\int_0^\infty \pi(\ell)(\ell+1)^{-d}  \int_{\R^d} \big|\partial^{\fun/\osc}_{x,\ell+1} F_{*t} \big|^2dxd\ell}
\\
&\lesssim&  \Big(\frac{\sqrt{T}}{\sqrt{t}}\Big)^{d}  \int_0^\infty (\frac{\ell+1}{\sqrt{T}}\wedge 1)^{\e d} \pi(\ell)d\ell
\\
&\lesssim&\Big(\frac{\sqrt{T}}{\sqrt{t}}\Big)^d \Big(\sqrt{T}^{-\e d}\int_0^{\sqrt{T}} \pi(\ell) (\ell+1)^{\e d}\,d\ell+\int_{\sqrt{T}}^\infty \pi(\ell)d\ell\Big).
\end{eqnarray*}
By \eqref{m2} \& \eqref{min.decay}, and an integration by parts,
$$
\int_0^{\infty} \pi(\ell) (\ell+1)^{\e d}\,d\ell\,\stackrel{\eqref{m2}}{\lesssim} \, 1+ \int_0^\infty \gamma(\ell)(\ell+1)^{\e d-1}d\ell
\,\stackrel{\eqref{min.decay}}{\lesssim}\, 1+ \int_0^\infty (\ell+1)^{\e d-1-\beta}d\ell \,\lesssim \,1
$$
provided $\e d-\beta<0$ (which we may assume w.l.o.g.~by reducing the hole-filling exponent $\e$).
By~\eqref{m2}, $\gamma$ is non-increasing, so that for all $R>0$ we have using \eqref{m_*} 
\begin{equation*}
\int_{R}^\infty \pi(\ell)d\ell\,\stackrel{\eqref{m2}}{=}\,\gamma(R)\,\leq \, \fint_{B_R}\gamma(|x|)\,dx
\,\stackrel{\eqref{m_*}}=\,\pi_*(R)^{-1}.
\end{equation*}
Hence,
$$
{\int_0^\infty  \pi(\ell) (\ell+1)^{-d}\int_{\R^d} \big|\partial^{\fun/\osc}_{x,\ell+1}F_{*t} \big|^2dxd\ell}\,\lesssim\, \Big(\frac{\sqrt{T}}{\sqrt{t}}\Big)^d(\sqrt{T}^{-\e d}+\pi_*(\sqrt{T})^{-1}).
$$
Combined with  \eqref{min.decay} in form of  $\pi_*(r)^{-1}\lesssim r^{-\beta}$ and the relation  $\e d-\beta<0$,  this turns into 
$$
{\int_0^\infty \pi(\ell) (\ell+1)^{-d} \int_{\R^d} \big|\partial^{\fun/\osc}_{x,\ell+1}F_{*t} \big|^2dxd\ell}\,\lesssim\, \Big(\frac{\sqrt{T}}{\sqrt{t}}\Big)^d\sqrt{T}^{-\e d}.
$$
As in Step~1, this yields
$$
\frac1T \int_{T^{1-\e}}^T{\int_0^\infty \pi(\ell) (\ell+1)^{-d} \int_{\R^d} \big|\partial^{\fun/\osc}_{x,\ell+1}F \big|^2dxd\ell}dt\,\lesssim\, T^{-\e},
$$
and therefore proves the desired estimate in combination with  \eqref{e.fluct-1} and \eqref{e.st1.2nd-mom}.

\appendix

\section{Caccioppoli's inequality}\label{append:Cacc}

In the proofs, we shall make intensive use of the classical Caccioppoli  argument,
which we state for future reference, and prove for the reader's convenience.
\begin{lemma}\label{L:reg}
Let $R\geq 1$. Consider $u,g$ related in a distributional sense by
\begin{equation}\label{stef.1}
    -\nabla\cdot a\nabla u=-\nabla\cdot g\qquad\text{in }B_R.
\end{equation}
There exists a constant $C=C(d,\lambda)>0$ such that for any constant $c$ we have
\begin{equation}\label{f.51}
\forall 0<\rho<R\,:\qquad  \int_{B_{R-\rho}}|\nabla u|^2\leq C\left(\int_{B_R}|g|^2+\frac{1}{\rho^2}\int_{B_R\setminus B_{R-\rho}}(u-c)^2\right).
\end{equation}
\qed
\end{lemma}

\begin{proof}[Proof of Lemma~\ref{L:reg}]

For the convenience of the reader, we recall the standard argument under
the weak ellipticity assumption (\ref{f.40}), which thanks to the homogeneity
of the coefficients could be weakened further,
see \cite[Proposition 2.1]{Giaquinta-83}. By scaling, we may w.\ l.\ o.\ g.\
assume that $R=1$ and by adding a constant, $c=0$, so that it remains to show
\begin{equation}\label{f.61}
    \int_{B_{1-\rho}}|\nabla u|^2\lesssim\int_{B_1}|g|^2+\frac{1}{\rho^2}\int_{B_1\setminus B_{1-\rho}}u^2.
\end{equation}
where here and below $\lesssim$ stands for $\leq$ up to a constant that depends on $d$ and $\lambda$.
To this purpose we test $-\nabla\cdot a\nabla u=0$
with $\eta^2 u$, where $\eta$ is a cut-off for $B_{1-\rho}$ in $B_{1}$; using Leibniz' rule 
in form of
\begin{eqnarray*}\nonumber
\lefteqn{\nabla(\eta^2u)\cdot a\nabla u}\\
&=&\nabla(\eta u)\cdot a\nabla(\eta u)+u\nabla\eta\cdot a\nabla(\eta u)
-u\nabla(\eta u)\cdot a\nabla\eta-u^2\nabla\eta\cdot a\nabla\eta,
\end{eqnarray*}
we obtain with \eqref{stef.1} the identity
\begin{equation}\nonumber
    \int\nabla(\eta u)\cdot a\nabla(\eta u)=\int g\cdot \nabla(\eta^2u)+ \int(-u\nabla\eta\cdot a\nabla(\eta u)
    +u\nabla(\eta u)\cdot a\nabla\eta+u^2\nabla\eta\cdot a\nabla\eta).
\end{equation}
With $|\nabla(\eta^2u)|\leq |u|\eta|\nabla\eta|+\eta|\nabla(u\eta)|$ and by uniform ellipticity and boundedness of $a$, cf.\ (\ref{f.40}) and (\ref{f.56}), this yields
\begin{equation}\nonumber
\lambda\int|\nabla(\eta u)|^2\le\int |g|\eta(|\nabla(\eta u)|+|\nabla \eta||u|)+\int(2|u||\nabla\eta||\nabla(\eta u)|+u^2|\nabla\eta|^2).
\end{equation}
By Young's inequality this entails
\begin{equation}\nonumber
\int|\nabla(\eta u)|^2\lesssim\int |g|^2\eta^2+\int u^2|\nabla\eta|^2,
\end{equation}
so that by the properties of the cut-off function we obtain (\ref{f.61}).
Note that this also yields
\begin{equation}\label{e.add-Caccio}
\int  |\nabla u|^2\eta^2 \lesssim\int |g|^2\eta^2+\int u^2|\nabla\eta|^2.
\end{equation}
\end{proof}



\section*{Acknowledgements}
We thank Peter Bella, Mitia Duerinckx, and Julian Fischer for suggestions on the manuscript.
AG acknowledges financial support from the European Research Council under
the European Community's Seventh Framework Programme (FP7/2014-2019 Grant Agreement
QUANTHOM 335410). SN acknowledges support by the DFG in the context of TU Dresden's Institutional Strategy ``The Synergetic University''.
AG and FO acknowledge the hospitality of the Mittag-Leffler Institute and the support of the Chaire
Schlumberger  at IH\'ES.

\bibliographystyle{plain}

\begin{thebibliography}{10}

\bibitem{ADS16}
S.~Andres, J.-D.~Deuschel, and M.~Slowik.
\newblock Heat kernel estimates for random walks with degenerate weights.
\newblock {\em Electron. J. Probab.}, 21:21 pp., 2016.

\bibitem{Andres-Neukamm-1}
S.~Andres and S.~Neukamm.
\newblock Berry-Esseen theorem and quantitative homogenization for the random conductance model with degenerate conductances.
\newblock {\em Stoch. Partial Differ. Equ. Anal. Comput.}, 7(2):240--296, 2019.

\bibitem{AD}
S.~N. {Armstrong} and J.-P. {Daniel}.
\newblock {Calder\'on-Zygmund estimates for stochastic homogenization}.
\newblock {\em J. Funct. Anal.}, 270(1):312--329, 2016.

\bibitem{Armstrong-Mourrat-14}
S.~N. {Armstrong} and J.-C. {Mourrat}.
\newblock {Lipschitz regularity for elliptic equations with random coefficients}.
\newblock {\em Arch. Ration. Mech. Anal.}, 219(1):423--481, 2016.

\bibitem{Armstrong-Smart-14}
S.~N. {Armstrong} and C.~K. {Smart}.
\newblock {Quantitative stochastic homogenization of convex integral
  functionals}.
\newblock {\em Ann. Sci. \'Ec. Norm. Sup\'er. (4)}, 4(2), 423--481, 2016.

\bibitem{Armstrong-Mourrat-Kuusi-16}
\newblock S.~N. Armstrong, T. Kuusi and J.-C. Mourrat. 
\newblock {\em Mesoscopic higher regularity and subadditivity in elliptic homogenization}. 
\newblock {\em Comm. Math. Phys.}, 347:315--361,  2016.

\bibitem{Armstrong-Mourrat-Kuusi-17}
\newblock  S.~N. Armstrong, T. Kuusi and J.-C. Mourrat. 
\newblock The additive structure of elliptic homogenization. 
\newblock {\em Invent. Math.}, 208:999--1154, 2017.

\bibitem{AKM-book}
S.~Armstrong, T.~Kuusi, and J.-C. Mourrat.
\newblock {\em Quantitative stochastic homogenization and large-scale
  regularity}, volume 352 of {\em Grundlehren der Mathematischen
  Wissenschaften}.
\newblock Springer, Cham, 2019.

\bibitem{Avellaneda-Lin-87}
M.~Avellaneda and F.-H. Lin.
\newblock Compactness methods in the theory of homogenization.
\newblock {\em Comm. Pure and Applied Math.}, 40(6):803--847, 1987.

\bibitem{Avellaneda-Lin-89}
M.~Avellaneda and F.-H. Lin.
\newblock Un th\'eor\`eme de {L}iouville pour des \'equations elliptiques \`a
  coefficients p\'eriodiques.
\newblock {\em C. R. Acad. Sci. Paris S\'er. I Math.}, 309(5):245--250, 1989.

\bibitem{Barron}
E.~N. Barron, P.~Cardaliaguet, and R.~Jensen.
\newblock Conditional essential suprema with applications.
\newblock {\em Appl. Math. Optim.}, 48(3):229--253, 2003.

\bibitem{BFO16}
P.~Bella, B.~Fehrman, and F.~Otto.
\newblock A Liouville theorem for elliptic systems with degenerate ergodic coefficients.
\newblock {\em  Ann. Applied Probab.}, 28(3):1379--1422, 2018.

\bibitem{BFFO}
P. Bella, B. Fehrman, J. Fischer, and F. Otto.
\newblock{Stochastic Homogenization of Linear Elliptic Equations: Higher-Order Error Estimates in Weak Norms Via Second-Order Correctors}.
\newblock {\em SIAM J. Math. Anal.}, 49(6):4658--4703, 2017

\bibitem{BGO17}
P.~Bella, A.~Giunti, and F.~Otto.
\newblock {Effective Multipoles in Random media},
\newblock {\em arXiv:1708.07672}, 2017.

\bibitem{BO17}
P.~Bella and F.~Otto.
\newblock Corrector Estimates for Elliptic Systems with Random Periodic Coefficients.
\newblock {\em Multiscale Model. Simul.}, 14(4):1434--1462, 2016.

\bibitem{BMN17}
J. Ben-Artzi, D. Marahrens and S. Neukamm.
\newblock Moment bounds for the corrector in stochastic homogenization of
  discrete linear elasticity.
\newblock {\em Comm. Partial Differential Equations}, 42(2):179-234, 2017.

\bibitem{BG}
A. Benoit and A. Gloria.
\newblock Long-time homogenization and asymptotic ballistic transport of classical waves.
\newblock {\em Ann. Sci. \'Ec. Norm. Sup\'er. (4)}, 52:703--760, 2019.

\bibitem{Bensoussan-Lion-Papa-78}
\newblock A. Bensoussan, J.~L. Lions, G.~Papanicolaou.
\newblock {\em Asymptotic analysis for periodic structures}.
\newblock Studies in Mathematics and its Applications, Vol. 5, North-Holland Publishing Co., Amsterdam, 1978.

\bibitem{Biskup-11}
M. Biskup.
\newblock Recent progress on the random conductance model.
\newblock {\em Probability Surveys}, 8, 2011.

\bibitem{BCKY-11}
I.~{Benjamini}, H.~{Duminil-Copin}, G.~{Kozma}, and A.~{Yadin}.
\newblock {Disorder, entropy and harmonic functions}.
\newblock {\em Ann. Probab.}, 43(5):2332--2373, 2015.

\bibitem{CGO-17}
J. G. Conlon, A. Giunti, and F. Otto.
\newblock{Green's function for elliptic systems: existence and Delmotte-Deuschel bounds}.
\newblock {\em Calc. Var. Partial Differential Equations}, 56(6), art.~163, 2017.

\bibitem{DeGiorgi-75}
€'E.~De~Giorgi.
\newblock {\em Sulla convergenza di alcune successioni d'integrali del tipo dell'area}.
\newblock Rend. Mat. (6) , 8:277--294,  1975.

\bibitem{DNS17}
J.-D.~Deuschel, T.~A.~Nguyen, and M.~Slowik.
\newblock Quenched invariance principles for the random conductance model on a random graph with degenerate ergodic weights.
\newblock {\em Probab. Theory Related Fields},170(1-2):363?386, 2018.

\bibitem{DGCP-72}
E.~De~Giorgi, F.~Colombini, and L.~C. Piccinini.
\newblock {\em Frontiere orientate di misura minima e questioni collegate}.
\newblock Scuola Normale Superiore, Pisa, 1972.

\bibitem{DG1}
M. Duerinckx and A. Gloria.
\newblock {\em Multiscale functional inequalities: Concentration properties}.
\newblock  {\em arXiv:1711.03148}, 2017.

\bibitem{DG2}
M. Duerinckx and A. Gloria.
\newblock {\em Multiscale functional inequalities: Constructive approach}.
\newblock  {\em arXiv:1711.03152}, 2017.

\bibitem{DGO}
M. Duerinckx, A. Gloria and F. Otto.
\newblock {\em The theory of fluctuations in stochastic homogenization}.
\newblock  {\em arXiv:1602.01717}, 2016.

\bibitem{DGO-Gaus}
M. Duerinckx, A. Gloria and F. Otto.
\newblock {\em Robustness of the pathwise structure of fluctuations in stochastic homogenization}.
\newblock  {\em arXiv:1807.11781}, 2018.

\bibitem{DFGO-Gaus}
M. Duerinckx, J. Fischer, and A. Gloria.
\newblock {\em Scaling limit of the homogenization commutator
for correlated Gaussian coefficient fields.}
\newblock In preparation.

\bibitem{Fischer-Otto-15}
J. Fischer and F. Otto.
\newblock A higher-order large-scale regularity theory for random elliptic operators
\newblock {\em  Comm. Partial Differential Equations}, 41(7):1108--1148, 2016.

\bibitem{Fischer-Otto-17}
J. Fischer and F. Otto.
\newblock  Sublinear growth of the corrector in stochastic homogenization: optimal stochastic estimates for slowly decaying correlations. 
\newblock {\em Stoch. Partial Differ. Equ. Anal. Comput.} 5(2):220--255, 2017.

\bibitem{FR-16}
J. Fischer and C. Raithel.
\newblock Liouville principles and a large-scale regularity theory for random elliptic operators on the half-space.
\newblock {\em SIAM J. Math. Anal.}, 49(1):82--114, 2017.

\bibitem{Giaquinta-83}
M.~Giaquinta.
\newblock {\em Multiple integrals in the calculus of variations and nonlinear
  elliptic systems}, volume 105 of {\em Annals of Mathematics Studies}.
\newblock Princeton University Press, 1983.

\bibitem{Giaquinta-Martinazzi-05}
M.~Giaquinta and L.~Martinazzi.
\newblock {\em An Introduction to the Regularity Theory for Elliptic Systems, Harmonic Maps and Minimal Graphs}.
\newblock Pisa, Edizioni Della Normale, 2012.

\bibitem{Gilbarg-Trudinger-98}
D.~Gilbarg and N.S. Trudinger.
\newblock {\em Elliptic partial differential equations of second order}.
\newblock Classics in Mathematics. Springer-Verlag, Berlin, 2001.
\newblock Reprint of the 1998 edition.

\bibitem{Gloria-Habibi-16}
A.~Gloria and Z.~Habibi.
\newblock Reduction in the resonance error in numerical homogenization II: Correctors and extrapolation.
\newblock {\em Found. Comput. Math.}, 16:217--296, 2016.

\bibitem{Gloria-Marahrens-14}
A.~Gloria and D. Marahrens.
\newblock Annealed estimates on the Green functions and uncertainty quantification.
\newblock {\em Ann. Inst. H. Poincar\'e Anal. Non Lin\'eaire}, 33(5):1153--1197, 2016.

\bibitem{Gloria-Neukamm-Otto-14}
A.~Gloria, S.~Neukamm, and F.~Otto.
\newblock Quantification of ergodicity in stochastic homogenization: optimal
  bounds via spectral gap on {G}lauber dynamics.
\newblock {\em Invent. Math.}, 199(2):455--515, 2015.

\bibitem{GNO-preprint}
A.~Gloria, S.~Neukamm, and F.~Otto.
\newblock A regularity theory for random elliptic operators.
\newblock   {\em arXiv:1409.2678v1}, 2014.

\bibitem{GNO-quant}
A.~Gloria, S.~Neukamm, and F.~Otto.
\newblock Quantitative estimates in stochastic homogenization for correlated fields.
\newblock Preprint.

\bibitem{Gloria-Otto-09}
A.~Gloria and F.~Otto.
\newblock An optimal variance estimate in stochastic homogenization of discrete
  elliptic equations.
\newblock {\em Ann. Probab.}, 39(3):779--856, 2011.

\bibitem{Gloria-Otto-14}
A.~Gloria and F.~Otto.
\newblock Quantitative estimates on the periodic approximation of the corrector
  in stochastic homogenization.
\newblock Proceedings of the CEMRACS'13 ``Modelling and simulation of complex
  systems: stochastic and deterministic approaches".

\bibitem{Gloria-Otto-10b}
A.~Gloria and F.~Otto.
\newblock Quantitative results on the corrector equation in stochastic
  homogenization.
\newblock {\em J. Eur. Math. Soc.}, 19(11):3489--3548, 2017.

\bibitem{GO-16}
A.~Gloria and F.~Otto.
\newblock The corrector in stochastic homogenization: optimal rates, stochastic integrability, and fluctuations.
\newblock {\em arXiv:1510.08290}, 2015.

\bibitem{Mourrat-Gu-16}
Y.~Gu and J.~C.~ Mourrat.
\newblock Scaling limit of fluctuations in stochastic homogenization.
\newblock {\em Multiscale Model. Simul.} 14(1):452--481, 2016.

\bibitem{JKO-94}
V.V. Jikov, S.M. Kozlov, and O.A. Oleinik.
\newblock {\em Homogenization of Differential Operators and Integral
  Functionals}.
\newblock Springer-Verlag, Berlin, 1994.

\bibitem{JKON-79}
V. Zhikov, S. Kozlov, O. Oleinik, K. Ngoan.
\newblock {\em Averaging and G-convergence of differential operators}.
\newblock Russian Math. Surveys 34:69--147,  1979.


\bibitem{Kozlov-78}
S.~M. Kozlov.
\newblock The averaging of random operators.
\newblock {\em Mat. Sb. (N.S.)}, 109(151)(2):188--202, 327, 1979.
  
\bibitem{Krengel-85}
U.~Krengel.
\newblock {\em Ergodic theorems}, volume~6 of {\em de Gruyter Studies in
  Mathematics}.
\newblock De Gruyter, 1985.

\bibitem{Kumagai-14}
T. Kumagai.
\newblock Random walks on disordered media and their scaling limits: {\'E}cole
  d'{\'e}t{\'e} de probabilit{\'e}s de Saint-Flour 2010.
\newblock {\em Lecture Notes in Mathematics}, 2014.

\bibitem{Ledoux-97}
M.~Ledoux.
\newblock Concentration of measure and logarithmic {S}obolev inequalities.
\newblock Notes (Berlin, 1997).

\bibitem{Ledoux-01}
M.~Ledoux.
\newblock {\em The concentration of measure phenomenon}, volume~89 of {\em
  Mathematical Surveys and Monographs}.
\newblock American Mathematical Society, Providence, RI, 2001.

\bibitem{Marahrens-Otto-13}
D.~Marahrens and F.~Otto.
\newblock Annealed estimates on the {Green's} function.
\newblock {\em Probab. Theory. Relat. Fields}, 163(3-4):527--573, 2015.

\bibitem{Murat-78}
F.~Murat.
\newblock {H}-convergence.
\newblock {S}\'eminaire d'Analyse fonctionnelle et num\'erique, {U}niv.
  {A}lger, multigraphi\'e, 1978.

\bibitem{Murat-97}
F.~Murat and L.~Tartar.
\newblock {H}-convergence.
\newblock In A.V. Cherkaev and R.V. Kohn, editors, {\em Topics in the
  Mathematical Modelling of Composites Materials}, volume~31 of {\em Progress
  in nonlinear differential equations and their applications}, pages 21--44.
  Birkh\"auser, 1997.


\bibitem{Naddaf-Spencer-98}
A.~Naddaf and T.~Spencer.
\newblock Estimates on the variance of some homogenization problems.
\newblock Preprint, 1998.

\bibitem{NPhd17}
T.~A.~Nguyen.
\newblock The random conductance model under degenerate conditions.
\newblock {\em PhD-Thesis, TU Berlin}, 2017.

\bibitem{Papanicolaou-Varadhan-79}
G.C. Papanicolaou and S.R.S. Varadhan.
\newblock Boundary value problems with rapidly oscillating random coefficients.
\newblock In {\em Random fields, {V}ol. {I}, {II} ({E}sztergom, 1979)},
  volume~27 of {\em Colloq. Math. Soc. J\'anos Bolyai}, pages 835--873.
  North-Holland, Amsterdam, 1981.
  
\bibitem{Penrose-01}
M.~D. Penrose.
\newblock Random parking, sequential adsorption, and the jamming limit.
\newblock {\em Comm. Math. Phys.}, 218(1):153--176, 2001.  

\bibitem{Sidoravicius-Sznitman-04}
V.~Sidoravicius and A.-S. Sznitman.
\newblock Quenched invariance principles for walks on clusters of percolation
  or among random conductances.
\newblock {\em Probab. Theory Related Fields}, 129(2):219--244, 2004.

\bibitem{Simon-97}
L.~Simon.
\newblock Schauder estimates by scaling.
\newblock {\em Calc. Var. Partial Differential Equations}, 5(5):391--407, 1997.

\bibitem{Spagnolo-75}
\newblock S.~Spagnolo.
\newblock Convergence in energy for elliptic operators.
\newblock in {\em Numerical Solutions of Partial Differential Equations, III (Proc. Third Sympos. (SYNSPADE), Univ. Maryland, College Park, Md., 1975)}, Academic Press, New York, 1976.



\bibitem{Stein}
E. M. Stein.
\newblock {\em Harmonic analysis: real-variable methods, orthogonality, and
              oscillatory integrals.}
\newblock Princeton Mathematical Series, 1993.		

\bibitem{Tartar-CP}
L. Tartar.
\newblock {\rm Cours Peccot au Coll{\'e}ge de France}, partially written by F.~Murat in S{\'e}minaire d'Analyse Fonctionelle et Num{\'e}rique de l'Universit{\'e} d'Alger, unpublished.


\bibitem{Torquato-02}
S.~Torquato.
\newblock {\em Random heterogeneous materials}, volume~16 of {\em
  Interdisciplinary Applied Mathematics}.
\newblock Springer-Verlag, New York, 2002.
\newblock Microstructure and macroscopic properties.

\bibitem{Yurinskii-86}
V.~V. Yurinski{\u\i}.
\newblock Averaging of symmetric diffusion in random medium.
\newblock {\em Sibirskii Matematicheskii Zhurnal}, 27(4):167--180, 1986.

\end{thebibliography}

\end{document}